\documentclass[10pt]{article}
\usepackage{latexsym,amsfonts,amsthm,amsmath,amscd,amssymb,color}
\usepackage{makeidx}
\usepackage{todonotes}
\usepackage{mathabx}
\usepackage{float}
\usepackage[colorlinks=true,citecolor=blue,linkcolor=black,]{hyperref}
\usepackage[all]{xy}

\usepackage{geometry}
\usepackage{xypic,color}

\theoremstyle{plain}
\newtheorem{lemma}{Lemma}[section] 
\newtheorem{theorem}[lemma]{Theorem}
\newtheorem{proposition}[lemma]{Proposition}
\newtheorem{corollary}[lemma]{Corollary}

\theoremstyle{definition}
\newtheorem{definition}[lemma]{Definition}
\newtheorem{example}[lemma]{Example}
\newtheorem{remark}[lemma]{Remark}
\newtheorem*{definition*}{Definition}

\theoremstyle{remark}

\newtheorem{ques}[lemma]{Question}

\newtheorem{convention}[lemma]{Convention}

\newcommand{\dd}{\mathrm{d}}

\title{The holonomy Lie $\infty$-groupoid of a singular foliation I}

\author{Camille Laurent-Gengoux\thanks{Institut \'Elie Cartan de Lorraine, UMR 7502, Universit\'e de Lorraine, Metz, France.}\\  Ruben Louis\thanks{Department of Mathematics, University of Illinois Urbana-Champaign, Urbana, IL, USA.}}

\usepackage{cite}

\usepackage{titling}

\begin{document}
\maketitle
\begin{abstract}We construct a finite-dimensional higher Lie groupoid integrating a singular foliation $\mathcal{F}$, under the mild assumption that the latter admits a geometric resolution. More precisely, a recursive use of bi-submersions, a tool coming from non-commutative geometry and invented by Androulidakis and Skandalis, allows us to integrate any universal Lie $ \infty$-algebroid of a singular foliation to a Kan simplicial manifold, where all components are made of non-connected manifolds which are all the same finite dimension that can be chosen to be equal to the ranks of a given geometric resolution. Its $1$-truncation is the Androulidakis-Skandalis holonomy groupoid.

\end{abstract}
\tableofcontents
\section*{Introduction}

As in \cite{Hermann,Cerveau,AS,AndroulidakisZambon, Debord,LLS} in smooth differential geometry or \cite{zbMATH03423310,Ali-Sinan} in holomorphic differential geometry, we define a singular foliation on a smooth, complex,  or real analytic manifold $M$,  with sheaf of functions $\mathcal O$, to be a subsheaf $\mathcal F \colon U\longrightarrow\mathcal F(U )$ of the sheaf of vector fields $\mathfrak X$, which is closed under the Lie bracket and locally finitely generated as an $\mathcal O$-module, see \cite{LLL1} for a modern introduction. 
By Hermann's theorem \cite{Hermann}, this is enough to induce a partition of the manifold $M$ into {immersed} submanifolds of possibly different dimensions, called \emph{leaves} of the singular foliation. Singular foliations generalize the notion of regular foliations by allowing leaves  with possibly different dimensions. They appear for instance as orbits of Lie group actions, or, more generally, as the orbits of a Lie groupoid. For instance, in Poisson geometry, we encounter a particularly intricate class of singular foliations known as “the symplectic leaves of a Poisson structure”, e.g., see \cite{LPV,CFM} which,  under some assumption \cite{Weinstein,Cattaneo-Felder,Crainic-Fernandes}, are the orbits of a symplectic groupoid. We refer to \cite[\S 4]{LLL1} for a detailed list of examples. In contrast to regular foliations, it is an open question whether a singular foliation can be derived from a Lie groupoid or the image of the anchor map of a Lie algebroid as in the examples above. This paper is part of a long series of efforts aimed at understanding the following question

\begin{ques}\label{question0}
    Given a singular foliation $\mathcal{F}$ on a manifold $M$, is there a Lie groupoid $\mathcal G\rightrightarrows M$ whose associated Lie algebroid $A\mathcal{G}\stackrel{\rho}{\rightarrow}TM$ induces $\mathcal{F}$, i.e., $\rho(\Gamma(A\mathcal{G}))=\mathcal{F}$?
\end{ques}
If $\mathcal{F}$ is a regular foliation or a Debord singular foliation (i.e., projective $\mathcal O$-module), then the response is positive, cf., \cite{Debord}. The response to that question is obviously “no” when the minimal number of local generators is not bounded globally (in \cite[Lemma 1.3]{AZ}, Androulidakis-Zambon gave an example of such a singular foliation). But Question \ref{question0} is open if this number is globally bounded.  Also, \cite[\S 4.5]{LLS} the first author, S. Lavau and T. Strobl defined an obstruction class for having a Lie algebroid of minimal rank defining $\mathcal{F}$ in a neighborhood of a point, and an example where it is not zero.
However, it does not answer the  question since it does not exclude the possibility that the dimension of the Lie groupoid is non-minimal, i.e., that the dimensions of the fiber of the anchor map are strictly greater than the upper bound of the number of local generators. 

One of the ideas that has emerged in recent years is to avoid the question by exploring higher-level structures, i.e., by looking for higher Lie groupoids or higher Lie algebroids. Initial work in this direction began with S. Lavau’s PhD thesis, followed by a collaborative paper with the first author and T. Strobl \cite{LLS}. This approach was later generalized for arbitrary Lie-Rinehart algebras in a purely algebraic framework by the authors in \cite{CLRL}, which brings us to a reformulation of Question \ref{question0}. We primarily aim to address the following question

\begin{ques}\label{ques1}
    Given a singular foliation $\mathcal{F}$ on a manifold $M$, is there a Lie $\infty$-groupoid, i.e., a finite-dimensional Kan simplicial manifold \begin{equation}
   K_\bullet\colon \xymatrix{&\cdots\; \ar@<5pt>[r]\ar@<-5pt>@{->}[r]\ar@<1pt>[r]\ar@<-2pt>@{.}[r]& \ar@/^{0.8pc}/[l]\ar@/^{1.1pc}/[l]\ar@{-}@/^{1.4pc}/[l] \ar@/^{1.7pc}/[l] K_3\ar@<5pt>[r]\ar@<-5pt>@{->}[r]\ar@<1pt>[r]\ar@<-2pt>[r]&\ar@/^{0.8pc}/[l]\ar@/^{1.1pc}/[l]\ar@/^{1.4pc}/[l] K_2\ar@<-2pt>[r]\ar@<6pt>[r]\ar@<2pt>[r]&\ar@/^{0.8pc}/[l]\ar@/^{1.2pc}/[l]K_1 \ar@<-2pt>[r]\ar@<2pt>[r] & \ar@/^{0.8pc}/[l]K_0},
   \end{equation}
   
   \vspace{0.3cm}
whose differentiation is a Lie $\infty$-algebroid that induces $\mathcal{F}$ on the base manifold $M$? If yes, is there a natural one and how unique is it? Last, for $i=0, 1,\ldots,$ can we make the dimension of the $K_i's$ minimal?
\end{ques}


To address Question \ref{ques1}, we consider singular foliations that admit  geometric resolutions, i.e., those for which there exists an anchored complex of vector bundles 

$$ (E,\dd,
\rho) \colon  \xymatrix{  \ar[r] & E_{-i-1} \ar[r]^{{\dd^{(i+1)}}} \ar[d] & 
     E_{ -i} \ar[r]^{{\dd^{(i)}}}
     \ar[d] & E_{-i+1} \ar[r] \ar[d] & \ar@{..}[r] & \ar[r]^{{\dd^{(2)}}}& E_{-1} \ar[r]^{\rho} \ar[d]& TM \ar[d] \\ 
      \ar@{=}[r] & \ar@{=}[r] M  &  \ar@{=}[r] M 
      &  \ar@{=}[r] M  &\ar@{..}[r] & \ar@{=}[r]   &  \ar@{=}[r] M  & M}
    $$
    such that the following complex of sheaves      \begin{equation}
           \label{resolution1}{\longrightarrow} \Gamma({E_{ -i-1}})
     \stackrel{\dd^{(i+1)}}{\longrightarrow} \Gamma({E_{-i}})
     \stackrel{\dd^{(i)}}{\longrightarrow}{\Gamma(E_{-i+1})}{\longrightarrow}\cdots  {\longrightarrow} \Gamma(E_{-1})  
     \stackrel{\rho}{\longrightarrow} \mathcal F.
      \end{equation}  is exact. It is quite natural to work with this class of singular foliation, as it contains the class of (locally) real analytic singular foliations. It is also a natural object in the holomorphic setting, since $\mathcal F$ is then a coherent sheaf and such geometric resolutions always exist locally.  It is proven in \cite{LLS,CLRL} that any geometric resolution $(E,\dd,\rho)$ of a singular foliation $\mathcal{F}$ can be equipped with a Lie $\infty$-algebroid 
      which is unique up to homotopy. The latter is referred to as a \underline{universal Lie $\infty$-algebroid} $\mathbb U_\mathcal{F}$ of $\mathcal F$. {The universal  Lie $\infty$-algebroid $\mathbb U_\mathcal{F}$ of a singular foliation $\mathcal F$ should be the infinitesimal counterpart of a Lie $\infty$-groupoid.} Therefore, finding a Kan simplicial manifold of Question \ref{ques1} amounts to integrating this universal Lie $\infty$-algebroid of $\mathcal{F}$.

      Our integration method gives a finite-dimensional object and is different from the work \cite{Siran-Severa}, where  M. {\v{S}}ira{\v{n}} and P. {\v{S}}evera constructed, following methods from \cite{SullivanDennis}, a Kan simplicial (Fréchet) Banach manifold from any negatively graded Lie $\infty$-algebroid (or $NQ$-manifold $(E,Q)$). The points of this manifold are solutions to an adapted Maurer-Cartan equation, and the manifold itself is infinite-dimensional. From this Kan simplicial manifold, they derived a non-canonical local Lie $\infty$-groupoid of finite dimension by imposing a specific gauge-fixing condition, as also employed by E. Getzler in \cite{Getzler} also \cite{Henriques}. This finite-dimensional local Lie $\infty$-groupoid depends on the gauge condition, the choice of local coordinates of the base manifold, and the local trivialization of $E$.

       The present work introduces a novel approach to integrating the universal Lie $\infty$-algebroid of a singular foliation into a finite-dimensional Kan simplicial manifold,
       \begin{equation}\label{Kan0}
   K_\bullet\colon \xymatrix{&\cdots\; \ar@<5pt>[r]\ar@<-5pt>@{->}[r]\ar@<1pt>[r]\ar@<-2pt>@{.}[r]& \ar@/^{0.8pc}/[l]\ar@/^{1.1pc}/[l]\ar@{-}@/^{1.4pc}/[l] \ar@/^{1.7pc}/[l] K_3\ar@<5pt>[r]\ar@<-5pt>@{->}[r]\ar@<1pt>[r]\ar@<-2pt>[r]&\ar@/^{0.8pc}/[l]\ar@/^{1.1pc}/[l]\ar@/^{1.4pc}/[l] K_2\ar@<-2pt>[r]\ar@<6pt>[r]\ar@<2pt>[r]&\ar@/^{0.8pc}/[l]\ar@/^{1.2pc}/[l]K_1 \ar@<-2pt>[r]\ar@<2pt>[r] & \ar@/^{0.8pc}/[l]K_0}
   \end{equation}
   
   \vspace{0.3cm}
       effectively creating a global Lie $\infty$-groupoid. Our method does not rely at all on the previous idea, but on the fundamental concept of “bi-submersion $M \stackrel{p}{\leftarrow} W \stackrel{q}{\rightarrow} N$ between singular foliations”, extending a notion introduced by \cite{AS}. A bi-submersion between two singular foliations $(M,\mathcal{F}_M)$ and $(N,\mathcal{F}_N)$ consists of two (surjective) submersions $p\colon W\to M$ and $q\colon W\to N$ satisfying some compatible conditions with respect to the singular foliations $\mathcal{F}_M$ and $\mathcal{F}_N$ as in Definition \ref{def:bi-submersion}. When $(M,\mathcal{F})=(M,\mathcal{F}_M)=(N,\mathcal{F}_N)$, we say that $W$ is a bi-submersion over $\mathcal
       F$. The latter can be understood as a Lie groupoid-like object without a product. Bi-submersions also come with a Morita-like equivalence between bi-submersions. This notion and all of its features were introduced by non-commutative geometers Androulidakis and Skandalis as local charts of the holonomy groupoid of a singular foliation \cite{AS}. The significant advantage of this notion is its finite-dimensional nature, situating it within the realm of traditional differential geometry. In order to use bi-submersions to build higher groupoid-like structures, the second author has introduced an extension of this notion in the paper \cite[Theorem 5.6]{RubenSymetries}, and proved that a singular foliation $\mathcal{F}$ admits a geometric resolution if and only if there exists a sequence of bi-submersions of the form \begin{equation}\label{tower0}T_\bullet\colon \xymatrix{\cdots\phantom{X}\ar@/^/[r]^{p_{i+1}}\ar@/_/[r]_{q_{i+1}}&{T_{i+1}}\ar@/^/[r]^{p_{i}}\ar@/_/[r]_{q_{i}}&{T_{i}}\ar@/^/[r]^{p_{i-1}}\ar@/_/[r]_{q_{i-1}}&{\cdots}\ar@/^/[r]^{p_1}\ar@/_/[r]_{q_1}&T_1\ar@/^/[r]^{p_0}\ar@/_/[r]_{q_0}&M}.
\end{equation} called \emph{bi-submersion tower over $\mathcal{F}$}, where $M \stackrel{p_0}{\leftarrow} T_{1} \stackrel{q_0}{\rightarrow} M$ is a bi-submersion over $\mathcal{F}$  and for each $i\geq 1$,\, $T_i \stackrel{p_i}{\leftarrow} T_{i+1} \stackrel{q_i}{\rightarrow} T_i$ is a bi-submersion over the  singular foliation $\mathcal{F}_i=\Gamma(\ker Tp_{i-1})\cap  \Gamma(\ker Tq_{i-1})$ on $T_{i+1}$ (referred to as the bi-vertical singular foliation of $T_{i+1}$) which exists under the existence of a geometric resolution of $\mathcal{F}$. This result allows the construction by recursion of a finite-dimensional Kan simplicial manifold whose simplices are given by gluing \underline{equivalences of bi-submersions} along a bi-submersion tower over $\mathcal{F}$ as in  \eqref{tower0}. This Kan simplicial manifold can be interpreted as follows: $K_0= M$ the space of object; $K_1$ is a bi-submersion over the singular foliation $(M,\mathcal{F})$; $K_2$ is an equivalence between bi-submersions $K_1$ and $\Lambda^2_\ell K$ (all the fiber products of $K_1$ with itself) which consists of a compatible bi-submersion $\Lambda^2_\ell K\stackrel{p_\ell^2}{\leftarrow}K_2\stackrel{d_\ell^2}{\rightarrow} K_1$ between the respective bi-vertical singular foliations of $K_1$ and $\Lambda^2_\ell K$ for $\ell=0, 1,2$; $K_3$ is an equivalence between equivalences of bi-submersions, that is an equivalence between $K_2$ and the horn spaces $\Lambda^3_\ell K$ which consists of a compatible bi-submersion $\Lambda^3_\ell K\stackrel{p_\ell^3}{\leftarrow}K_3\stackrel{d_\ell^3}{\rightarrow} K_2$ between the respective bi-vertical singular foliations of $K_2$ and $\Lambda^3_\ell K$ for $\ell=0,1,2,3$ and so on. By this construction, the complex of vector bundles
\begin{equation}
     \xymatrix{\cdots\ar[r]&\ker Tp_0^3\ar[r]^{Td_0^3}\ar[d]&\ker Tp_0^2\ar[r]^{Td_0^2}\ar[d]&\ker Td_1^1\ar[r]^{Td_0^1}\ar[d]&TM \ar[d]\\\cdots\ar[r]&K_3\ar[r]^{d_0^3}&K_2\ar[r]^{d_0^2}&K_1\ar[r]^{d_0^1}&M}
\end{equation}  

satisfies $Td^1_0\left (\Gamma(\ker  Td^1_1)\right)=(d^1_0)^*\mathcal F$, and is exact at the level of sections in the sense that for every $k\geq 1$, the intersection $\Gamma(\ker Td_0^k)\cap \Gamma(\ker Tp_{0}^k)$ is equal to the image of $d_{0}^{k+1}$-projectable vector fields in $\Gamma(\ker Tp_{0}^{k+1})$. In particular, the tangent complex of this Lie $\infty$-groupoid  $K_\bullet$ \begin{equation}
        \xymatrix{\cdots\ar[r]&\ker Tp_0^3|_M\ar[r]^{Td_0^3|_M}&\ker Tp_0^2|_M\ar[r]^{Td_0^2|_M}&\ker Td_1^1|_M\ar[r]^{Td_0^1|_M}&TM}
    \end{equation}is a geometric resolution of $\mathcal{F}$. As a consequence, the differentiation functor (cf. \cite{li2023differentiating}) applied to this Lie $\infty$-groupoid is a Lie $\infty$-algebroid built over a geometric resolution of $\mathcal F$. The latter is a universal Lie $\infty$-algebroid of the singular foliation in the sense of \cite{LLS}. Since a universal Lie $\infty$-algebroid is unique up to homotopy equivalence, this Lie $\infty$-groupoid $K_\bullet$ is an integration of any universal Lie $\infty$-algebroid of $\mathcal{F}$.
    
    We would like to emphasize that our construction has a limitation: the simplicial relations between the degeneracy maps $(s_\bullet)$, namely $s_i\circ s_j=s_{j+1}\circ s_i$, $i\leq j$, are not satisfied. {This is weaker than a simplicial manifold, but stronger than a semi-simplicial manifold \cite{Dorsch} (which has no degeneracies). Therefore, we do not have an honest simplicial manifold—at least not until we take a quotient. We refer to this as a \emph{para-simplicial structure}. The Kan condition still makes sense, since \begin{itemize}
        \item the Kan condition depends only on the face maps $d_i$;
\item degeneracies $s_i$ do not enter the horn-filling condition.
    \end{itemize}}  According to \cite[Theorem 5.7]{Rourke1971} and \cite{McClure2013}, there exists a system of degeneracies that satisfies all the simplicial conditions included $s_i\circ s_j=s_{j+1}\circ s_i$, $i\leq j$, since $K_\bullet$ satisfies Kan condition. However, the constructions of these articles do not allow them to be smooth in general.

    We call a {para-simplicial structure} that satisfies the Kan condition a \emph{para-Lie $\infty$-groupoid}.  
    
    This paper represents a preliminary attempt to address Question \ref{ques1} posed in the introduction. The key contribution of this initial work is to highlight the significance of the Androulidakis-Skandalis concept of bi-submersions in constructing a finite-dimensional para-Lie $\infty$-groupoid over a singular foliation. 
We claim that a similar method can be applied to integrate an arbitrary Lie $\infty$-algebroid whose linear part does not necessarily form  a geometric resolution as in Equation \eqref{resolution1} of its underlying singular foliation. 

The organization of this paper is as follows. In \S \ref{sec:1}, we review the fundamental concept of singular foliations and their associated universal Lie $\infty$-algebroids.  In \S \ref{sec:results}, we clarify what we mean by integrating a singular foliation and  present the main results of this paper.  In \S \ref{sec:construction}, we prove the main theorem of the paper by constructing a holonomy para-Lie $\infty$-groupoid that integrates a singular foliation. The appendix provides supplementary material supporting the development of the paper. Specifically, \S \ref{app:exact-sequence-of-vb} and  \S \ref{Appendix:higher-groupoids} provide a concise review of the definition of Lie $\infty$-groupoids and their tangent spaces, and exact sequence of vector bundles. In \S \ref{sec:bi-submersion}, we revisit the notion of bi-submersions between singular foliations. In \S \ref{sec:Equivalentbi-submersions}, we explore various notions of equivalence between bi-submersions. Furthermore, we derive new and significant results on bi-submersions, which play a crucial role in the main arguments of the paper. Last, in \S \ref{sec:tower}, we introduce the notion of “bi-submersion tower” a generalization of the notion of bi-submersion between two singular foliations.

\subsubsection*{Some conventions}
We recall some general notations and conventions of differential geometry.

\begin{itemize}
    \item For every $n\in \mathbb N$, $\mathcal B^n\subset \mathbb R^n$ stands for an open ball centered at zero.
    \item $\mathcal O$ stands for the sheaf of (smooth,
real analytic or holomorphic) functions on (smooth, analytic, complex manifold—depending on the context) $M$, and by $\mathfrak{X}$ the sheaf of vector fields on $M$. For simplicity, we primarily work in the smooth setting. In that case, we shall denote by $C^\infty(M)$ the algebra of smooth functions on $M$.
\item We shall denote by $\mathrm{rk}(E)$ the rank of a vector bundle $E\to M$.

    \item The $C^\infty(M)$-module of sections of a vector bundle $E\to M$ shall be denoted by $\Gamma(E)$. The $C ^\infty(U)$-module made of the restrictions of sections of $E$ to an open subset $\mathcal U\subset M$ shall be denoted by $\Gamma(E)|_\mathcal{U}$.

    \item For a vector bundle morphism $u\colon E\to F$ over the identity of $M$, we shall denote by the same $u$ the $C^\infty (M)$-linear map obtained at the level of sections $\Gamma(E)\to \Gamma(F),\;e\to u\circ e$.
    \item For a vector bundle morphism over different bases $$\xymatrix{E\ar[d]\ar[r]^d& F\ar[d]\\M\ar[r]^\phi & N},$$ we shall denote by $\Gamma_{proj}(E)\subseteq \Gamma (E) $  the subspace of all sections $e$ of $ E$ over $ M$ which are \emph{projectable}, i.e., such that there exists a section $\eta$ of $F$ with $d\circ e=\eta\circ \phi$.

    \item The pull-back of a vector bundle $E\to M$ along a smooth map $\varphi \colon N\to M$ is denoted by $\varphi^*E$. Also, the pull-back of a vector bundle morphism $u\colon E\to F$ along $\varphi$ shall be denoted by $\varphi^*u\colon \varphi^*E\to \varphi^*F$.
    \item The tangent map at a point $x\in M$ of a smooth map $\varphi \colon M\to N$  shall be denoted by $T_x\varphi \colon T_xM\to T_{\varphi(x)}N$. 
\end{itemize}

\begin{definition}
Let $\varphi\colon N\to M$ be a submersion. 
\begin{enumerate}
    \item A vector field $Z$ on $N$ is said to be \emph{$\varphi$-related} to a vector field on $X$ on $M$ if for every point $x\in N$ we have :
$T_x\varphi (Z(x)) = X\circ \varphi (x)$.  
\item A vector field $Z$ on $N$ is said to be $\varphi$-\emph{projectable} if $Z$ is $\varphi$-related to some vector field  on $M$.  
\end{enumerate}
\end{definition}

\section{Singular foliations}\label{sec:1}
In this section, we recall the concept of singular foliations and their features, as developed in 
\cite{Hermann, Nagano, Cerveau, Dazord,Debord,AS,AZ,LLS} and recently reviewed in \cite{LLL1}.

\begin{definition}\label{def:singfoliation}
    A \emph{singular foliation} on a smooth, real analytic, or complex manifold $M$   is a subsheaf $\mathcal{F}\subseteq\mathfrak{X}(M)$ of modules over functions that fulfills the following conditions
\begin{enumerate}
     \item \textbf{Stability under Lie bracket}: $[\mathcal{F},\mathcal{F}]\subseteq \mathcal{F}$.
     \item \textbf{Local finite generateness}: every $m\in M$ admits an
open neighborhood $\mathcal{U}\subseteq M$ together with a finite number of vector fields  $X_1, \ldots, X_r \in \mathfrak{X}(\mathcal U)$ such that for
every open subset $\mathcal V \subseteq \mathcal U$ the vector fields ${X_1}_{|_\mathcal{V}} ,\ldots, {X_r}_{|_\mathcal{V}}$ generate $\mathcal{F}$ on $\mathcal{V}$ as a $\mathcal{O}_\mathcal{V}$-module.
 \end{enumerate}
 \end{definition} 
From now on, we will deal with the smooth case: most result, however, extend to real analytic or complex settings.

\begin{definition}
Let $(M, \mathcal{F}_M)$ and $(N,\mathcal{F}_N)$ be singular foliations and $\phi\colon N\to M$ a diffeomorphism. We say that $\phi$ is an \emph{isomorphism of singular foliations} if $\phi_*(\mathcal{F}_M)=\mathcal{F}_N$. Here  $\phi_*\colon \mathfrak X(M)\to \mathfrak X(N),\, X\to T\phi\circ X\circ \phi^{-1}$ denotes push-forward map along $\phi$. When $M=N$ and $ \mathcal F_M = \mathcal F_N=\mathcal{F}$, we shall speak of a \emph{symmetry of $ \mathcal F$}.
\end{definition}

\begin{remark}
There are several alternative approaches to defining singular foliations on a manifold $M$. These approaches share a common feature: they are defined as Lie-Rinehart subalgebras $\mathcal F$ of the Lie-Rinehart algebra $\mathfrak X (M) $ of vector fields on $M$ (or compactly supported vector fields $\mathfrak X_c(M) $ on $M$). Here are some important consequences of Definition \ref{def:singfoliation}.

\begin{enumerate}
     \item A singular foliation admits leaves: there exists a partition of $M=\cup_{m\in M}L_m$ into immersed submanifolds of $M$ called leaves such that for all $m\in M$, the image $T_m\mathcal{F}\subseteq T_mM$ of the evaluation map $\mathcal F \to T_m M, X\to X(m)$  is the tangent space of the leaf $L_m$ through $m$. 
     \item \emph{Singular foliations are self-preserving}: the flow of  a vector field in $\mathcal F$, whenever defined, is a symmetry of $\mathcal F$, cf.,  \cite{AS,GarmendiaAlfonso}, \cite{RubenSymetries}-also \cite[\S 7.2]{LLL1}.
\end{enumerate}
\end{remark}

The following proposition is very useful.
\begin{proposition}\cite{AS}
Let $(M,\mathcal{F})$  be a singular foliation. Let $m\in M$ and denote by $\mathcal{I}_m$ the ideal of functions vanishing  at $m$. The dimension of the vector space $\frac{\mathcal{F}}{\mathcal{I}_m\mathcal{F}}$ is the minimal number of local generators of $\mathcal{F}$ on an open neighborhood of $m$. More precisely, for every family of vector fields $X_1, \ldots, X_n\in \mathcal{F}$ such that their classes  $[X_1]_m, \ldots, [X_n]_m\in \frac{\mathcal{F}}{\mathcal{I}_m\mathcal{F}}$ form a basis, there exists an open neighborhood of $m$ on which $X_1, \ldots, X_n$ are local generators for $\mathcal{F}$. This set of generators is minimal in the sense that none of the $X_i$ can be
written as a $C^\infty(M)$ linear combination of the others. 
\end{proposition}

\subsection{Pull-back of a singular foliation and clean intersections}
For a singular foliation $(M,\mathcal{F})$ and $m\in M$ we shall denote by $T_m\mathcal{F}:=\{X(m), X\in \mathcal{F}\}\subseteq T_mM$ the \emph{tangent space of $\mathcal{F}$ at $m$}.
\begin{definition}
Let $P, M$ be manifolds and $\mathcal{F}$ a singular foliation on $M$. We say that a map $\varphi\colon P\longrightarrow M$ is \emph{transverse} to $\mathcal{F}$ and write $\varphi \pitchfork\mathcal{F}$ if the following equivalent statements are satisfied:

\begin{enumerate}
    \item $\forall p\in P$, $\mathrm{im}(T_p\varphi)+ T_{\varphi(p)}\mathcal{F}= T_{\varphi(p)}M$.
    \item the sheaf morphism $\varphi^*\mathcal{F}\oplus\mathfrak X(P)\longrightarrow \Gamma(\varphi^*TM),\; (\alpha, Z)\longmapsto \alpha+ T\varphi(Z)$ is onto.
\end{enumerate}
Here, $\varphi ^*\mathcal F\subset \Gamma(\varphi^*TM)$ is the $C^{\infty}(P)$-module generated the sections  $X\circ \varphi\in \Gamma(\varphi^*TM)$ with $X\in \mathcal{F}$.
\end{definition}
\begin{example}
    Every submersion $\varphi\colon P\longrightarrow M$ is transverse to any singular foliation  $(M, \mathcal{F})$. 
\end{example}
We refer to  \cite[Proposition 1.10 and Proposition 1.11]{AS} for the proof of the following statement.
\begin{proposition}\label{prop:transverse-foliation}
  Let $(M,\mathcal{F})$ be  a singular foliation and $\varphi\colon P\longrightarrow M$ a smooth map transverse to $\mathcal{F}$ then 
\begin{enumerate}
      \item the $C^{\infty}(P)$-submodule $\varphi^{-1}(\mathcal F):= \{Z\in \mathfrak X(P)\mid T\varphi(Z)\in \varphi^*\mathcal{F}\}\subseteq \mathfrak X(P)$ is a singular foliation on $P$. Moreover, $$ T_p\left(\varphi^{-1}(\mathcal F)\right)=(T_p\varphi)^{-1}(T_{\varphi(p)}\mathcal{F}).$$
      \item Let $Q$ be a manifold and $\psi\colon  Q\longrightarrow P$ be a smooth map. Then $\psi\pitchfork \varphi^{-1}(\mathcal F)$ if and only if
$\varphi\circ \psi \pitchfork \mathcal{F}$ and we have $$(\varphi\circ \psi )^{-1}(\mathcal{F})= \psi^{-1}(\varphi^{-1}(\mathcal F)).$$
  \end{enumerate}
  The singular foliation $\varphi^{-1}(\mathcal{F}_M)\subset \mathfrak X(P)$ is called the \emph{pullback} foliation of $\mathcal{F}$ along $\varphi$.
\end{proposition}
\begin{remark}
   The submodule $\varphi^{-1}(\mathcal{F})\subset \mathfrak X(P)$  in Proposition \ref{prop:transverse-foliation}(1), is always stable under Lie bracket even if $\varphi$ is not transverse to $\mathcal{F}$, but it may be infinitely generated in the smooth case, cf. \cite[Exercice 5.13(8)]{LLL1}. The “transverse” assumption guarantees that it is a locally finitely generated submodule, hence a singular foliation.  
\end{remark}

\begin{definition}Let $(M, \mathcal F)$ be a singular foliation. Let $S\subset M$ be a submanifold. \begin{enumerate}
    \item We say that $S$ is \emph{intersects $\mathcal{F}$ cleanly}  if $\iota_S\pitchfork \mathcal{F}$, where  $S\stackrel{\iota_S}{\hookrightarrow}M$ is the inclusion map.
    
    \item Let  $m\in S$ be a point of $S$. We say that $S$ is \emph{transverse to $\mathcal{F}$ at $m$} if $\iota_S\pitchfork \mathcal{F}$ and $T_mS\,\oplus\, T_m\mathcal{F}=T_mM$. In particular, $\{m\}$ is a leaf of $\mathcal{F}_S$. In this case, we denote by  $\mathcal T_m^\mathcal{F} $ the \emph{transverse} singular foliation $\mathcal{F}_S=\mathcal{F}|_S:=\iota_S^{-1}(\mathcal{F})$ at $m$.
\end{enumerate}
\end{definition}
\begin{remark}
    In particular, for $S\subseteq M$ a sub-manifold intersecting a singular foliation $(M,\mathcal{F})$ cleanly, the singular foliation $\mathcal{F}_S$ is made of the restrictions to $S$ of  all vector fields of $\mathcal{F}$ that are tangent to $S$.
\end{remark}

\begin{proposition}\label{prop:operations-on-transversals}
Let $(M, \mathcal{F})$ be a singular foliation and $\varphi\colon P\rightarrow M$ a submersion. For every immersed sub-manifold $S$

\begin{enumerate}
    \item If $S\subset M$ intersects $\mathcal{F}$ cleanly, then the inverse image $\varphi^{-1}(S)$ intersects $\varphi^{-1}(\mathcal{F})$ cleanly. 
    \item If $\Sigma\subset P$  transverse to $\varphi^{-1}(\mathcal{F})$ at a point $p\in P$, then the image  $\varphi(\Sigma)$ is a submanifold transverse to $\mathcal{F}$ at $\varphi(p)\in M$. In particular, the corresponding leaves of the singular foliations $\varphi^{-1}(\mathcal{F})$ and $\mathcal{F}$ have the same codimensions.
\end{enumerate}
\end{proposition}
 \begin{proof}
For $\sigma\in p^{-1}(S)$ and  $v\in T_\sigma P$, there exist a vector field $X\in \mathcal{F}_M$ and  $u\in T_{\sigma}(p^{-1}(S))$ such that  $Tp(v)=\underbrace{X(p(\sigma))}_{T_{p(\sigma)}\mathcal{F}}+\underbrace{Tp(u)}_{\in T_{p(\sigma)}S}$. Since $p$ is a submersion, there exists a vector field $\widetilde{X}\in p^{-1}(\mathcal{F})$ such that $Tp(\widetilde X)=X\circ p$. Hence, $v-(\widetilde X(\sigma)+ u)\in \ker T_\sigma p$. This proves item 1 of the claim.

    Let us prove item 2. 
    For every $u\in T_{\varphi(p)}M$ we have $u=T_p \varphi(\widetilde u)$ for some $\widetilde u\in T_p P$. There is a decomposition $\widetilde u=\widetilde u_\Sigma+ Z(p)$ with $\widetilde u_\Sigma\in T_p\Sigma,\;  Z\in \varphi^{-1}(\mathcal{F})$. This implies that $$u= \underbrace{T_p\varphi(\widetilde u_\Sigma)}_{T_{\varphi(p)}\left(\varphi(\Sigma)\right)}+ \underbrace{T_p\varphi(Z(p))}_{\in T_{\varphi(p)}\mathcal{F}}.$$ It is clear that $\varphi(\Sigma)$ intersects cleanly the leaf that passes though $\varphi(p)$ at a single point. This proves the claim.
    \end{proof} 

\subsection{Geometric resolutions of a singular foliation.}
 Throughout this paper, we are interested in singular foliations that admit geometric resolutions in the sense of \cite{LLS}. Let us recall  this notion.\\

\noindent
\textbf{Geometric resolutions of a singular foliation.} Let $\mathcal{F}$ be a singular foliation on $M$. A \emph{anchored complex of vector bundles over $\mathcal{F}$} consists of a triple $(E_\bullet,\dd^\bullet,\rho)$, where
   \begin{enumerate}
       \item $ E_\bullet = (E_{-i})_{i \geq 1}$ is a family of vector bundles over $M$, indexed by negative integers.
       \item  $\dd^{(i+1)}\in \mathrm{Hom}(E_{-i-1}, E_{-i})$ is a vector bundle morphism over the identity of $M$ called the \emph{differential map}
       \item $\rho \colon E_{-1} \longrightarrow TM $ is a vector bundle morphism over the identity of $M$ called the \emph{anchor map} with $\rho(\Gamma(E_{-1}))=\mathcal{F}$. 
   \end{enumerate}
   such that \begin{equation}\label{eq:geom-ch-complex}\xymatrix{ \cdots \ar[r] & E_{-i-1} \ar[r]^{{\dd^{(i+1)}}} \ar[d] & 
     E_{ -i} \ar[r]^{{\dd^{(i)}}}
     \ar[d] & E_{-i+1} \ar[r] \ar[d] & \ar@{..}[r] & \ar[r]^{{\dd^{(2)}}}& E_{-1} \ar[r]^{\rho} \ar[d]& TM \ar[d] \\ 
      \ar@{=}[r] & \ar@{=}[r] M  &  \ar@{=}[r] M 
      &  \ar@{=}[r] M  &\ar@{..}[r] & \ar@{=}[r]   &  \ar@{=}[r] M  & M}\end{equation}
  which form a chain complex, i.e.,
  $$ \dd^{(i)}\circ\dd^{(i+1)}=0  \hbox{ and }  \rho \circ \dd^{(2)}=0.$$

{\textbf{Cohomology at the level of sections}}.
     The complex of vector bundles \eqref{eq:geom-ch-complex} induces a complex of sheaves of modules over functions. More explicitly, for every open subset $\mathcal U \subset M $, there is a complex of $C^\infty(\mathcal U)$-modules: $$ \cdots {\longrightarrow} \Gamma_{\mathcal U}({E_{ -i-1}})
     \stackrel{\dd^{(i+1)}}{\longrightarrow} \Gamma_{\mathcal U}({E_{-i}})
     \stackrel{\dd^{(i)}}{\longrightarrow}{\Gamma_{\mathcal U}(E_{-i+1})}{\longrightarrow}\cdots  \stackrel{\dd^{(2)}}{\longrightarrow} \Gamma_{\mathcal U}(E_{-1})  
     \stackrel{\rho}{\longrightarrow} \mathcal F_{\mathcal U} \subset \mathfrak X(\mathcal U).$$   Since $\mathrm{Im}(\dd^{(i+1)})\subseteq\ker \dd^{(i)}$ for every $i\in\mathbb N$, we are allowed to define the quotient spaces,
      $$ H^{-i}(E_{\bullet},\mathcal U) = \left\{ \begin{array}{ll} \frac{\ker\left(\Gamma_\mathcal{U}(E_{-1}) \overset{\rho}{\longrightarrow }\mathcal{F}_\mathcal U\right) 
      }{\mathrm{Im}\left(\Gamma_\mathcal{U}(E_{-2})\overset{\dd^{(2)}}{\longrightarrow}\Gamma_\mathcal U(E_{-1})\right)} & \hbox{for $i=1 $}\\&
      \\ \frac{\ker \left(\Gamma_\mathcal{U}(E_{-i})\overset{\dd^{(i)}}{\longrightarrow }\Gamma_\mathcal{U}(E_{i+1})\right)} {\mathrm{Im}\left(\Gamma_\mathcal{U}(E_{-i-1})\overset{\dd^{(i+1)}}{\longrightarrow}\Gamma_\mathcal{U}(E_{-i})\right)} &
      \hbox{ for $ i \geq 2$.}\end{array}\right.$$
  They are  modules over functions on $\mathcal U$. For each $i\geq 1$, $H^{-i}(E_{\bullet},\mathcal U)$ is called the \emph{$i$-th cohomology} of $(E_\bullet, \dd^\bullet, \rho) $ at the level of sections on $\mathcal{U}$.  
    \begin{definition}\label{def:geometric-resolution}
A \emph{geometric resolution} of $\mathcal{F}$ is an anchored complex of vector bundles $(E_\bullet, \dd^\bullet, \rho) $ over $\mathcal{F}$ such that every point $m\in M$ admits an open neighborhood $\mathcal{U}\subset M$ such that $(E_\bullet, \dd^\bullet, \rho) $ induces an exact complex at the level of sections, that is $$H^{-i}(E_{\bullet},\mathcal U')=\{0\}$$ for every open subset $m\in \mathcal U'\subset \mathcal U$ and $i\geq 1$. \end{definition}

    \begin{remark}
        Although free/projective resolution of a singular foliation $\mathcal{F}\subseteq \mathfrak X(M)$ as a $\mathcal{O}$-module always exist (maybe given by infinitely generated projective  $\mathcal{O}$-modules), geometric resolution may not exist in the smooth case. For instance, consider the singular foliation $\mathcal{F}$ on $M=\mathbb R$ generated by the single vector field  $\varphi(t)\frac{d}{dt}$ where $\varphi(t)=e^{-\frac{1}{t^2}}$ for $t>0$ and $\varphi(t)=0$ for $t\leq 0$. The latter does not admit a geometric resolution. However, there is a wide class of singular foliations for which geometric resolutions of a singular foliation automatically exist, at least locally, and are of finite length, e.g., (locally) real analytic and holomorphic cases. We refer to \cite[\S 2]{LLS} and \cite[\S 6]{LLL2} for a more detailed discussion.
    \end{remark}

\subsection{Universal  Lie $\infty$-algebroids of a singular foliation}
\noindent
\textbf{Negatively graded   Lie $\infty$-algebroids}
\begin{definition}
\label{NGLA}
A \emph{ Lie $\infty$-algebroid} $\left(E,(\ell_k)_{k\geq 1}, \rho\right)$ over a manifold $M$ is a collection of vector bundles $E =(E_{-i})_{i\geq 1}$ over $M$ endowed with a sheaf of $L_\infty$-algebra
structures $(\ell_k)_{k\geq 1}$ over the sheaf of sections of $E$ together with a vector bundle morphism $\rho\colon E_{-1}\to TM$, called the \emph{anchor map}, such that the $k$-ary-brackets $$\ell_k: \underbrace{\Gamma(E)\times \cdots \times \Gamma(E)}_{k {\text{-times}}}\longrightarrow \Gamma(E)$$ are all $\mathcal O$-multilinear  except when $k=2$ and at least one of the arguments is of degree $-1$, while the $2$-ary bracket  satisfies the Leibniz identity

\begin{equation}
    \ell_2(x, f y) = \rho(x)[f]y + f\ell_2(x, y),\; x \in \Gamma(E_{-1}), y\in\Gamma(E), \;f\in \mathcal{O}.
\end{equation}
\end{definition}

\begin{remark}
Definition \ref{NGLA} implies the following facts.
\begin{enumerate}

\item For all $ x, y \in\Gamma(E_{-1})$, $\rho(\ell_2(x,y) ) = [\rho(x), \rho(y)]$. Hence, the image $\mathcal{F}:=\rho(\Gamma(E_{-1}))\subseteq \mathfrak X(M)$ is a singular foliation on $M$ called the \emph{basic singular foliation} of $\left(E,(\ell_k)_{k\geq 1}, \rho\right)$. We say then that the  Lie $\infty$-algebroid $\left(E,(\ell_k)_{k\geq 1}, \rho\right)$ is \emph{over $\mathcal{F}$}.
\item The sequence of morphisms of vector bundles
$$\xymatrix{\cdots\ar[r]^{\ell_1}&E_{-2}\ar[r]^{\ell_1}&E_{-1}\ar[r]^{\rho}&TM}$$ is a anchored complex of vector bundles over $ \mathcal F$  that we call the \emph{linear part}.
\end{enumerate}
\end{remark}

We now present the main statement of this section. The following theorem holds for arbitrary Lie-Rinehart algebras; however, we provide the statement specifically in the context of singular foliations.
\begin{theorem}[\cite{LLS,CLRL}]\label{thm:existence} 
Let $\mathcal{F}$ be a singular foliation on a manifold $M$.
\begin{enumerate}
    \item Any resolution of $\mathcal F $ by free or projective $\mathcal{O}$-modules 
\begin{equation}
    \label{eq:resolutions}
\cdots \stackrel{\dd} \longrightarrow P_{-3} \stackrel{\dd}{\longrightarrow} P_{-2} \stackrel{\dd}{\longrightarrow} P_{-1} \stackrel{\rho}{\longrightarrow} \mathcal{F}\longrightarrow 0 \end{equation}
  carries a Lie $\infty $-algebroid structure over $\mathcal{F}$ whose unary bracket is $\ell_1:=\dd $.
  \item In particular, when $\mathcal F$ admits a geometric resolution $(E_{-\bullet
}, \dd^{\bullet}, \rho)$, there exists a  Lie $\infty$-algebroid $(\ell_\bullet, \rho)$ over $\mathcal{F}$ whose linear part is $(E_{-\bullet
}, \ell_1=\dd^{\bullet}, \rho)$.

\item Such a  Lie $\infty$-algebroid structure is unique up to homotopy equivalence and is denoted by $\mathbb U_\mathcal{F}$. The latter is called a \emph{universal  Lie $\infty$-algebroid of $\mathcal{F}$}.
\end{enumerate} 
\end{theorem}
\begin{remark}
    Theorem \ref{thm:existence} generalizes the main result of  the work \cite{LLS} with the first author, together with S. Lavau and T.
Strobl. That is, it applies to singular foliations that do not admit geometric resolutions and even to arbitrary Lie subalgebra of $\mathfrak X(M)$. For instance, for $M=\mathbb R^2$, consider the $C^\infty(M)$-module $\mathcal F$ generated by the vector field $\frac{\partial}{\partial x}$ and vector fields of the form $\varphi(x) \frac{\partial}{\partial y}$ where $\varphi \colon \mathbb R\to \mathbb R$ is a smooth function vanishing on $x\leq 0$ and strictly non-zero on $x>0$. The  $C^\infty(M)$-submodule $\mathcal F\subset \mathfrak X(M)$ is stable under Lie bracket, but it is not locally finitely generated and there is no leaf through the points on the $y$-axis. Thus, $\mathcal F$ is not a singular foliation. However, $\mathcal F$ is a Lie-Rinehart algebra and there is a universal Lie $\infty$-algebroid of $\mathcal F$ built over a free resolution\footnote{Free resolutions of any $C^\infty(M)$-module always exist.} of $\mathcal F$, see \cite{CLRL} for more details.
\end{remark}

{The universal  Lie $\infty$-algebroid $\mathbb U_\mathcal{F}$ of a singular foliation $\mathcal F$ is the infinitesimal counterpart of a Lie $\infty$-groupoid. In \S \ref{sec:results} and \S \ref{sec:construction},  we will address the integration problem by constructing  a finite-dimensional para-Lie $\infty$-groupoid $K_\bullet$ integrating $\mathbb U_{\mathcal{F}}$. Moreover, for $i\geq 0$, $\dim K_i$ will be given explicitly as a function.}

\section{Main results}\label{sec:results}

In \cite{AS}, Androulidakis and Skandalis have constructed the holonomy groupoid of a singular foliation, which is a topological groupoid, but not a Lie groupoid in general. In this section, we announce the construction of the holonomy para-Lie $ \infty$-groupoid. We refer to Appendix \ref{Appendix:higher-groupoids} for the notion of (para) Lie $\infty$-groupoid on a manifold $M$. In one word, a (para) Lie $\infty$-groupoid are Kan semi-simplicial manifolds $ K_\bullet $ such  that for every $n \geq 1$, $K_n$ is a \underline{finite} dimensional manifold, which may be non-connected but has to have connected components that are \underline{all of the same dimension}. Also, $ K_0=M$, and degeneracies have to exist, although they may not satisfy all the simplicial relations.

The construction of Androulidakis and Skandalis holonomy groupoid is based on the notion of bi-submersion introduced by themselves. So is our construction:
we have provided detailed appendices, \ref{sec:bi-submersion} and \ref{sec:Equivalentbi-submersions}, presenting further developments and new results on this concept: in the particular the notion of equivalence, of total relations, and of bi-vertical vector fields are explained in detail in these appendices.

For now, we assume that the reader is familiar with the concepts above, allowing us to present the main results of this paper.

        

To begin, we first clarify what we mean by the holonomy para-Lie  $ \infty$-groupoid of a singular foliation. We start with a few generalities on para-Lie $ \infty$-groupoids. 

Consider a para-Lie $\infty$-groupoid \begin{equation}\label{eq:L-inftyG}
   K_\bullet\colon \xymatrix{&\cdots\; \ar@<5pt>[r]\ar@<-5pt>@{->}[r]\ar@<1pt>[r]\ar@<-2pt>@{.}[r]& \ar@/^{0.8pc}/[l]\ar@/^{1.1pc}/[l]\ar@{-}@/^{1.4pc}/[l] \ar@/^{1.7pc}/[l] K_3\ar@<5pt>[r]\ar@<-5pt>@{->}[r]\ar@<1pt>[r]\ar@<-2pt>[r]&\ar@/^{0.8pc}/[l]\ar@/^{1.1pc}/[l]\ar@/^{1.4pc}/[l] K_2\ar@<-2pt>[r]\ar@<6pt>[r]\ar@<2pt>[r]&\ar@/^{0.8pc}/[l]\ar@/^{1.2pc}/[l]K_1 \ar@<-2pt>[r]\ar@<2pt>[r] & \ar@/^{0.8pc}/[l]K_0}
   \end{equation}
   \vspace{0.1cm}

It is easily checked from the para-simplicial identities that the sequence of vector bundle morphisms as follows:
  \begin{equation}\label{eq:bigcomplex}
     \xymatrix{\cdots\ar[r]&\ker Tp_0^3\ar[r]^{Td_0^3}\ar[d]&\ker Tp_0^2\ar[r]^{Td_0^2}\ar[d]&\ker Td_1^1\ar[r]^{Td_0^1}\ar[d]&TM \ar[d]\\\cdots\ar[r]&K_3\ar[r]^{d_0^3}&K_2\ar[r]^{d_0^2}&K_1\ar[r]^{d_0^1}&M}
\end{equation} 
is a complex of vector bundles. From now on, it will be referred to as \emph{full tangent complex}.  Since $M$ can be seen as a submanifold of $K_k, k\geq 1$ through the degeneracies, the full tangent complex can be restricted to a complex of vector bundles over $M$, which matches the \emph{tangent complex} of \cite{Mehta,CUECA2023108829}: 
  \begin{equation}\label{eq:smallcomplex}
     \xymatrix{\cdots\ar[r]&\mathfrak i_M^* \ker Tp_0^3 \ar[r]^{\mathfrak i_M^* Td_0^3}\ar[d]&\mathfrak i_M^* \ker Tp_0^2\ar[r]^{\mathfrak i_M^* Td_0^2}\ar[d]&\mathfrak i_M^* \ker Td_1^1\ar[r]^>>>>>>{\mathfrak i_M^* Td_0^1}\ar[d]&TM \ar[d]\\\cdots\ar[r]&M\ar[r]^{\mathrm{id}}&M\ar[r]^{\mathrm{id}}&M\ar[r]^{\mathrm{id}}&M}
\end{equation} 
Here, $\mathfrak i_M^* $ stands for the restriction to $M$ of the vector bundles and vector bundle morphisms of Equation \eqref{eq:bigcomplex}.
We call \emph{basic singular foliation} of $ K_\bullet$ the sub-sheaf $ \mathcal F$ of vector fields  on $M$ defined above as the image of the last vector bundle morphism of the tangent complex \eqref{eq:smallcomplex}. In equation:
 $$\mathcal F := \mathfrak i_M^* Td_0^1\left( \Gamma(\mathfrak i_M^* \ker Td_1^1)\right) $$

Now, let us open some discussion. What is a correct definition of the notion of holonomy para-Lie $\infty$-groupoid of a singular foliation?
 It is tempting to define it as a para-Lie $\infty$-groupoid that satisfies the following two conditions:
    \begin{enumerate}
        \item its basic singular foliation is $ \mathcal F$,
        \item the full tangent complex is exact.
    \end{enumerate}
   However, we claim that it is not a correct definition, although it certainly has to satisfy these two items. A correct definition is more involved, and we claim that it has to go through the following three conditions:
   
\begin{definition}
\label{def:universal-Grpoid}
We call \emph{holonomy para-Lie $\infty$-groupoid} of a singular foliation $ \mathcal F$ a para-Lie $ \infty$-groupoid $ K_\bullet$ that satisfies the following conditions. 
     First, $M \leftarrow K_1\rightarrow M$ has to be a holonomy bi-submersion atlas of $\mathcal{F}$, and 
for every $k\geq 1$ and $0\leq \ell\leq k+1$, \begin{enumerate}
            \item the submodules of bi-vertical vector fields $$\begin{cases}
                \mathcal{BV}_{K_k}=\bigcap_{j=0}^k\Gamma\left(\ker Td^{k}_{j}\right)\subset \mathfrak X(K_k)\\&\\
     \mathcal{BV}_{\Lambda_\ell^{k+1}K}=\left(\bigcap_{j=1}^{\ell}\Gamma\left(\ker T(d_{\ell-1}^k\circ \mathrm{pr}_j)\right)\right)\cap\left(\bigcap_{j=\ell+1}^{k+1}\Gamma\left(\ker T(d_\ell^k\circ \mathrm{pr}_j)\right)\right)\subset\mathfrak X(\Lambda_\ell^{k+1}K)
            \end{cases} $$
            
            have to be singular foliations\footnote{We should remove terms that do not make sense in $\mathcal{BV}_{\Lambda_\ell^{k+1}K}$ when $\ell=0$ and $\ell=k+1$.} on the manifold $K_k$ and on the horn $\Lambda_\ell^{k+1}K$ respectively.

     \item Both
      $\Lambda_\ell^kK \stackrel{p_\ell^{k}}{\longleftarrow} K_k  \stackrel{d_\ell^{k}}{\longrightarrow} K_{k-1}$  and
     $\Lambda_\ell^kK\leftarrow\Lambda_{\ell}^{k+1} K\rightarrow K_{k-1}$ 
     have to be bi-submersions between $\mathcal{BV}_{\Lambda^{k}_\ell K}$ and $\mathcal{BV}_{K_{k-1}}$. For convenience, for $\ell=k+1$  we let $\Lambda^k_{k+1}=\Lambda^k_{k}$.
     \item These bi-submersions have to be equivalent.
     \item Moreover, 
     $\Lambda_\ell^{k+1}K \stackrel{\;
     \,p^{k+1}_\ell}{\longleftarrow} K_{k+1} \stackrel{d_\ell^{k+1}}{\longrightarrow} K_{k}$ is a total relation of bi-submersions between $\Lambda^{k+1}_\ell K$ and $K_k$. 
     
\end{enumerate}

\end{definition}

Definition \ref{def:universal-Grpoid} uses some notations that we now explain. For $j=1,\ldots, k+1$, $$\mathrm{pr}_j\colon \underbrace{K_k\times \cdots \times K_k}_{k+1-\text{times}}\to K_k$$ is the $j$-th projection on $K_k$. We shall let $K_0=M$ and  $\mathcal{BV}_{\Lambda^1_\ell K}=\mathcal{BV}_{K_0}=\mathcal{F}$. Also,
for a para-Lie $ \infty$-groupoids $K_\bullet$ and for $k\geq 1$, we consider the following pairs of maps on the horn spaces

\begin{align}\label{eq:pairs-of-maps-horn}
        \ell=0\colon \xymatrix{\Lambda^k_{0}K&&& \Lambda^{k+1}_0 K\ar[lll]_<<<<<<<<<<<<<<{\bigtimes^{i>1} d_{0}^{k}\circ \mathrm{pr}_{i}}\ar[rrr]^{d_{0}^{k}\circ \mathrm{pr}_{1}}&&& K_{k-1}}\\ \nonumber 0<\ell < k\colon \xymatrix{\Lambda^k_{\ell}K&&& \Lambda^{k+1}_\ell K\ar[lll]_<<<<<<<<<<<<<<{d_{\ell-1}^k\circ \mathrm{pr}_{i\leq \ell}\times d_{\ell}^k\circ \mathrm{pr}_{i>\ell+1}}\ar[rrr]^{d_{\ell}^k\circ \mathrm{pr}_{\ell+1}}&&& K_{k-1}}\\\nonumber \ell=k\colon \xymatrix{\Lambda^k_{k}K&&&\Lambda^{k+1}_{k} K\ar[lll]_<<<<<<<<<<<<<<{\bigtimes^{i\leq k} d_{k-1}^k\circ \mathrm{pr}_{i}}\ar[rrr]^{d_{k}^k\circ \mathrm{pr}_{k+1}}&&& K_{k-1}}\\ \nonumber\ell=k+1\colon \xymatrix{\Lambda^k_{k}K&&&\Lambda^{k+1}_{k+1} K\ar[lll]_<<<<<<<<<<<<<<{\bigtimes^{i\leq k}d_{k}^k\circ \mathrm{pr}_{i}}\ar[rrr]^{d_{k}^k\circ \mathrm{pr}_{k+1}}&&& K_{k-1}}.
    \end{align}

This somewhat formidable definition needs to be justified. So does the following proposition:

\begin{proposition}
 For any holonomy para-Lie $ \infty$-groupoid of a singular foliation $ \mathcal F$ 
\begin{enumerate}
\item Its $1$-truncation is the Androulidakis-Skandalis holonomy groupoid of the singular foliation $ \mathcal F$.
\item The full tangent complex is exact at the level of sections.
\item The differentiation functor (cf. \cite{li2023differentiating}) applied to $ K_\bullet$ gives a universal Lie $\infty$-algebroid of $\mathcal F$.
\end{enumerate}

\end{proposition}

\begin{proof}
    \begin{enumerate}
\item By the definition of a holonomy  para-Lie $\infty$-groupoid of a singular foliation, {the quotient $K_1/_{\sim_1}\rightrightarrows M$ is the holonomy groupoid of Androulidakis-Skandalis of $\mathcal{F}$,  with $x\sim_1 y \in K_1$ if and only if  $d_i^1x=d_i^1y$ for $i=0,1$ and there exists a $2$-simplex $z\in K_2$ such that $d_0^2z=s_0^0d_0^1(x)=s_0^0d_0^1(y)$ and $d_1^2z=x,\;d_2^2z=y$}: Indeed, {the equivalence relation in the sense $x\sim_1 y$ through a point $z\in K_2$  implies that the points $x,y\in K_1$ are related w.r.t the bi-submersion $K_1$, i.e., there exists a local bi-submersion morphism $\phi\colon K_1\to K_1$ that sends $x$ to $y$: indeed, $\phi$ is the compositions of $d_2^2\colon K_2\to K_1$ with a local section of $p_2^2\colon K_2\rightarrow \Lambda_2^2K$ that sends $(s_0^0d_0^1x, x)$ to $z$ and with the map $x'\in K_1\mapsto (s_0^0d_0^1x', x')\in \Lambda_2^2K$. Conversely, given a morphism of bi-submersions $\phi\colon K_1\to K_1$ sending $x\in K_1$ to $y\in K_1$, the points $x$ and $(s_0^0d_0^1x, y)$ are related through the morphism $K_1\mapsto \Lambda_1^2K,\; x'\to (s_0^0d_0^1x', \phi(x'))$. Since $K_2$ is total, it implies that there exists a $2$-simplex $z\in K_2$ such that $d_0^2z=s_0^0d_0^1(x)=s_0^0d_0^1(y)$ and $d_1^2z=x,\;d_2^2z=y$.  Therefore, $x\sim_1 y \in K_1$ if and only if  there exists a morphism of bi-submersions  sending $x$ to $y$. Hence, the Androulidakis-Skandalis holonomy groupoid  $\mathcal{H}(M,\mathcal{F})\rightrightarrows M$ is $K_{1}/{\sim_1} \rightrightarrows $ M.}

\item {In the definition of a holonomy para-Lie $\infty$-groupoid of a singular foliation $\mathcal{F}$, the two maps $\Lambda^k_0K\stackrel{p_0^k}{\longleftarrow}K_k\stackrel{d_0^k}{\longrightarrow}K_{k-1}$ form a bi-submersion between the bi-vertical singular foliations $\mathcal{BV}_{\Lambda^k_0K}$ and $\mathcal{BV}_{K_{k-1}}$ for every $k\geq 1$. By the exioms of a bi-submersion \S \ref{sec:bi-submersion}, for every $k\geq 1$, we have $(p_0^k)^{-1}(\mathcal{BV}_{\Lambda^k_0K})=(d_0^k)^{-1}(\mathcal{BV}_{K_{k-1}})=\Gamma(\ker Tp_0^k)+ \Gamma(\ker Td_0^k)$.} This implies that $Td_0^k(\Gamma(\ker T p_0^k))=(d_0^k)^*\mathcal{BV}_{K_{k-1}}$ for $k\geq 1$, which expresses the exactness of the full tangent complex of $K_\bullet$.

\item The restriction to $M$ of the full tangent complex of a holonomy para-Lie $\infty$-groupoid of a singular foliation $\mathcal{F}$ is a geometric resolution of $\mathcal{F}$. Now, any Lie $\infty$-algebroid built on that geometric resolution is a universal Lie $\infty$-algebroid of $\mathcal{F}$, see \cite[Theorem 2.7]{LLS} and \cite[Theorem 2.1]{CLRL}.
This completes the proof.
        \end{enumerate}
\end{proof}

The main statement of this paper is the following:

\begin{theorem}\label{int:theorem}
    Every singular foliation $\mathcal{F}\subset \mathfrak X(M)$ on a manifold $M$ that admits  a geometric resolution $$E_{-\bullet}\colon \xymatrix{\cdots\ar[r]^{}&E_{-2}\ar[r]^{}&E_{-1}\ar[r]^{}&TM}$$
    admits a holonomy para-Lie $\infty$-groupoid  $ \mathcal K_\bullet(\mathcal{F})$:
    \begin{equation*}
   \mathcal{K}_\bullet(\mathcal{F})\colon \xymatrix{&\cdots\; \ar@<5pt>[r]\ar@<-5pt>@{->}[r]\ar@<1pt>[r]\ar@<-2pt>@{.}[r]& \ar@/^{0.8pc}/[l]\ar@/^{1.1pc}/[l]\ar@{-}@/^{1.4pc}/[l] \ar@/^{1.7pc}/[l] K_3\ar@<5pt>[r]\ar@<-5pt>@{->}[r]\ar@<1pt>[r]\ar@<-2pt>[r]&\ar@/^{0.8pc}/[l]\ar@/^{1.1pc}/[l]\ar@/^{1.4pc}/[l] K_2\ar@<-2pt>[r]\ar@<6pt>[r]\ar@<2pt>[r]&\ar@/^{0.8pc}/[l]\ar@/^{1.2pc}/[l]K_1 \ar@<-2pt>[r]\ar@<2pt>[r] & \ar@/^{0.8pc}/[l] M}
    \end{equation*} \\
 Moreover, one can assume the  following items to hold: \begin{enumerate}
\item For every $k \geq 1$, all the connected components of the manifold $K_k$ are of the same dimension given by $$\displaystyle{\dim K_k=\dim M+ \sum_{i=1}^k\begin{pmatrix}k\\i
\end{pmatrix} \mathrm{rk}(E_{-i})}.$$
\item[]   In addition, for all $k\geq 1$, there are vector bundle isomorphisms:
$$
\xymatrix{
   \ar[rr]^{\simeq}\left(d_0^1\circ \cdots  \circ d_0^k\right)^* E_{-k}   \ar[dr]& &\ar[dl] \ker Tp^k_0 \\
     & K_k &
}
$$

\item If $ E_{-k}=0$ for all $ k \geq n+1$,
for some $n\in \mathbb N$,
         the horn projections $p^k_\ell\colon K_k\to \Lambda^k_\ell K$ are surjective local diffeomorphisms for all $0\leq \ell \leq k$. {Moreover, for all $k\geq n+1$,  $K_{k+1} =  \coprod_{i=0}^{k+1} W_i$ where $W_i\subset\Lambda_i^{k+1} K  \times K_k$ is the submanifold made of pairs $(y,x)$ which are equivalent.}
    \end{enumerate}
\end{theorem}

This construction extends a construction presented in the PhD of the second author \cite[Chapter 10]{louis2023universalhigherliealgebras} and \cite[\S 5]{RubenSymetries} that integrates a singular foliation that admits a geometric resolution into an exact bi-submersion tower, see Appendix \ref{sec:tower}.



\section{Proof of Theorem \ref{int:theorem}}\label{sec:construction}

This section is dedicated to the proof of Theorem \ref{int:theorem}. We construct by recursion a holonomy para-Lie $ \infty$-groupoid of a singular foliation $(M,\mathcal{F})$, under the (commonly satisfied) assumption that it admits a geometric resolution \begin{equation}
    \cdots \stackrel{\dd^{(3)}}{\longrightarrow} E_{-2}\stackrel{\dd^{(2)}}{\longrightarrow}E_{-1}\stackrel{\rho}{\longrightarrow} TM
\end{equation} as in Definition \ref{def:geometric-resolution}.
The proof of Theorem \ref{int:theorem} will consist mostly in successive applications of generalities about bi-submersions that we recall in Appendices \ref{sec:bi-submersion} and \ref{sec:Equivalentbi-submersions}. More precisely, the construction is based on Theorems \ref{thm:bi-sub-equiv} and \ref{th:reducedimension}. Theorem \ref{thm:bi-sub-equiv} ensures the existence of a total relation of bi-submersions, while Theorem \ref{th:reducedimension} allows us to reduce the dimension.

Since our method is recursive, we need to start with a technical definition, to be compared with Definition \ref{def:sMfd}.

\begin{definition}
 A \emph{para-simplicial manifold up to order $N\in \mathbb N$}:
 $$K_{\leq N}:\xymatrix{&K_N\ar@<5pt>[r]\ar@<-5pt>@{->}[r]\ar@<1pt>[r]\ar@<-2pt>@{.}[r]&\ar@/^{0.8pc}/[l]\ar@/^{1.2pc}/[l]\ar@{.}@/^{1.4pc}/[l] \ar@/^{1.7pc}/[l]\cdots\; \ar@<5pt>[r]\ar@<-5pt>@{->}[r]\ar@<1pt>[r]\ar@<-2pt>@{.}[r]&\ar@/^{0.8pc}/[l]\ar@/^{1.1pc}/[l]\ar@{-}@/^{1.4pc}/[l] \ar@/^{1.7pc}/[l]K_3\ar@<5pt>[r]\ar@<-5pt>@{->}[r]\ar@<1pt>[r]\ar@<-2pt>[r]&\ar@/^{0.8pc}/[l]\ar@/^{1.2pc}/[l]\ar@/^{1.6pc}/[l]K_2\ar@<-2pt>[r]\ar@<6pt>[r]\ar@<2pt>[r]&\ar@/^{0.8pc}/[l]\ar@/^{1.2pc}/[l]K_1 \ar@<-2pt>[r]\ar@<2pt>[r]& \ar@/^{0.8pc}/[l]K_0=M}$$
    
\vspace{0.2cm}
is a finite sequence $ K_0,\dots,K_N$ are manifolds equipped with maps $\left(d_i^k\colon K_k\to K_{k-1}\right)_{0\leq i\leq k\leq N}$ and $\left(s_i^k\colon K_k\to K_{k+1}\right)_{0\leq i\leq k\leq N-1}$  that satisfy the simplicial axioms \eqref{eq:faces-faces} and \eqref{eq:faces-deneg} whenever they make sense.
\end{definition}


 Let us now define the precise para-simplicial manifolds up to order $N$ that we are going to construct.

\begin{definition}
\label{def:para-Grp-upto}
A para-simplicial manifold up to order $N\in \mathbb N$ is said to be a para-Lie $ \infty$-groupoid up to order $N$ when the natural projections from $K_k$ to the horn spaces 
$\left(\Lambda_\ell^kK\right)_{0\leq \ell\leq k \leq N+1}$ (which in this context are automatically smooth manifolds for the same reason as in Remark \ref{prop:horns}) are surjective submersions.
It is said to be a \emph{holonomy para-Lie $\infty$-groupoid up to order $N\in \mathbb N$ over a singular foliation $(M,\mathcal{F})$} if the items of Definition \ref{def:universal-Grpoid} for all indices for which they make sense, more precisely:
\begin{enumerate}
\item the first, second, and third items hold for $ k=1, \dots, N$,
\item the fourth item holds for  $k=1, \dots, N-1$. 
\end{enumerate}
 
\end{definition}




    \begin{convention}
        {We believe that the terminology “holonomy 
$N$-bi-submersion atlas” would be a more suitable designation for a 
“holonomy para-Lie $\infty$-groupoid up to order $N$", as it naturally captures the idea that a holonomy 
1-bi-submersion atlas corresponds to a holonomy bi-submersion atlas. However, for consistency, we will retain the terminology “holonomy para-Lie 
$\infty$-groupoid up to order $N$” throughout this section.}
    \end{convention}

\subsection{A holonomy para-Lie $\infty$-groupoid up to order  $1$}

We choose $K_0:=M$ and we construct $ K_1$. 

 
 
{\textbf{Step 1.} } 
Since $(E_{-1},\rho) $ is an anchored bundle over $M$ defining $ \mathcal F$,  there exists a neighborhood $ \mathcal U_{E_{-1}} \subset E_{-1} $ of the zero section equipped with a bi-submersion structure over $\mathcal{F}$ as in Example \ref{ex:holonomy-biss2}. 
As explained in Example \ref{ex:atlas}, 
the disjoint union $$K_1' = \coprod_{ n \in \mathbb N}  \underbrace{\mathcal U_{E_{-1}}^{\pm}*\cdots*\mathcal U_{E_{-1}}^{\pm}}_{n-\text{times}}  $$ is an atlas for $ \mathcal F$, where $^{\pm}$ means that we consider $\mathcal U_{E_{-1}}$ or its inverse $\mathcal U_{E_{-1}}^{-1}$. We denote by $ d_0^1,d_1^1$ its source and target, respectively. So far, this corresponds to the construction in \cite{AS} of the path-holonomy atlas in \cite{AS} when the vector bundle $ E_{-1}$ is trivial, see \cite[\S 5.2]{LLL2} for a description enlarged to the present context.

{\textbf{Step 2.}}  
$K_1'$ is still not a Lie $\infty$-groupoid up to order $1$ since $ K_1'$ is a disjoint union of manifolds which are not all of the same dimension (and moreover, their dimensions are even not bounded).
To solve this issue, we apply Theorem \ref{th:reducedimension}. Since $\mathcal F$ admits a geometric resolution, bi-vertical vector fields form a singular foliation on $$\underbrace{\mathcal U_{E_{-1}}^{\pm}*\cdots*\mathcal U_{E_{-1}}^{\pm}}_{n-\text{times}},\; n\in \mathbb N$$ by Lemma \ref{lem:anchoredbundleforbivertical}. Moreover, the codimension of the leaves of this singular foliation are bounded by the integer described in Equation \eqref{eq:dim:min}. This integer is smaller than the dimension of $E_{-1}$, seen as a manifold, i.e., the sum of the dimension of $M$ with the rank of $E_{-1}$. By Theorem \ref{th:reducedimension}, there exists, for all $ n \geq 2$, a submanifold $\mathcal{A}_n \subset  \mathcal U_{E_{-1}}^{\pm} * \cdots * \mathcal U_{E_{-1}}^{\pm} $ ($n$ times), made of connected submanifolds of dimension $\mathrm{dim}(M)+\mathrm{rk}(E_{-1})$, that satisfies the following two properties: each submanifold intersects cleanly the bi-vertical singular foliation and each leaf of the bi-vertical singular foliation is intersected at least once: as stated in  Theorem \ref{th:reducedimension}, this implies that $ \mathcal A_n$ is a bi-submersion equivalent to $ \mathcal U_{E_{-1}}^{\pm} * \cdots * \mathcal U_{E_{-1}}^{\pm} $ ($n$ times). 

{\textbf{Step 3.}} 
We choose $K_1$ to be the union of $\mathcal U_{E_{-1}}$, $ \mathcal U_{E_{-1}}^{-1}$ with all the submanifolds $\mathcal A_n$ for $n \geq 2$. We still denote by $ d_0^1$ and $ d_1^1$ the source and target maps obtained by restrictions of the same maps. 
The manifold $K_1$ hence constructed is again a bi-submersion, and it is equivalent to $ K_1'$. 
 Moreover, a unit map still exists, since $\mathcal U_{E_{-1}}$ still belongs to that atlas, and since the zero section is a unit map.

\begin{lemma}\label{lem:step1}
$ K_1$ is a holonomy para-Lie $ \infty$-groupoid up to order $1$ that satisfies condition 1 in Theorem \ref{int:theorem}.
\end{lemma}
\begin{proof}
Item 1 in Definition \ref{def:universal-Grpoid} is satisfied, since bi-vertical vector fields on $ K_1$ and on the horn space do form a singular foliation in view of Lemma \ref{lem:anchoredbundleforbivertical}. Item 2 holds because the horns in $K_1$, i.e., $(\Lambda_i^2K)_{0\leq i\leq 2}$ are fiber products of bi-submersions. Item 3 holds, by definition of an atlas, $K_1$ is equivalent to the three horn spaces $(\Lambda_i^2K)_{0\leq i\leq 2}$ according to Remark \ref{rmk:3hornspaces}. {By construction, we have  $(d_0^1)^*E_{-1}\subseteq\ker Td^1_1$. Since $\dim K_1=\mathrm{rk}(E_{-1})+\dim M$, we have $(d_0^1)^*E_{-1}=\ker Td^1_1$. Hence,  condition 1 in Theorem \ref{int:theorem} is satisfied.}
\end{proof}

\subsection{A holonomy para-Lie $\infty$-groupoid up to order  $2$}

This section is a special case of the general case of the general situation studied in \S \ref{sec:Ngeq3}: we include it for pedagogical reasons, since it is easier to visualize triangles than $N$-simplices.

We continue the construction of the para-Lie $\infty$-groupoid initiated in Lemma \ref{lem:step1} as follows.
Recall that, in that lemma, an atlas $ K_1$ of the holonomy Lie groupoid has been constructed, made of a disjoint union of manifolds all of dimension $ \mathrm{dim}(M)+\mathrm{rk}(E_{-1})$. 

 We need to construct $K_2$ so that

\begin{enumerate}
\item $\dim K_2=\mathrm{rk}(E_{-2})+\dim \Lambda^2_0K$,
    \item the horns spaces in $K_2$, i.e., $\left(\Lambda^3_\ell K\right)_{\ell=0,1,2,3}$ are bi-submersions between bi-vertical singular foliations $\mathcal{BV}_{K_2}$ and $\mathcal{BV}_{\Lambda_\ell^2 K}$ for $\ell=0,1,2$: this  will hold automatically,

    \item the bi-vertical singular foliations on $K_2$ and $\Lambda^3_\ell K$ exist for $\ell=0,1,2,3$: this will be a consequence of the existence of a geometric resolution $$ \cdots \stackrel{}{\longrightarrow} (d_0^1)^*E_{-2}\stackrel{}{\longrightarrow} TK_1$$of $\mathcal{BV}_{K_1}$.

    \item the horn space $\Lambda^3_\ell K$ is equivalent to $K_2$ for all $\ell=0,1,2,3.$
\end{enumerate}

\textbf{Step 1: construction of a first candidate for $K_2$ using $K_1$, but
whose dimension is not constant and there is no degeneracies: application of Theorem \ref{thm:bi-sub-equiv}.}

Since $ K_1$ is an atlas, by definition, it is equivalent to the horn space 
$$\Lambda^2_0 K = K_1 * K_1^{-1}.$$
There exists therefore by Theorem \ref{thm:bi-sub-equiv} a total relation $W_2' $ between them:
\begin{equation}\label{diag:S_2^{0}0}
    \scalebox{0.7}{ \xymatrix{&& (M,\mathcal{F})&& \\ \Lambda^2_0K \ar[urr]^{d_0^1\circ\mathrm{pr}_2}\ar[drr]_{d_0^1\circ\mathrm{pr}_1}&&\ar@{->>}[ll]_<<<<<<<<<<{p_0^2} W_2'\ar@{..>}[u]\ar@{..>}[d]\ar@{->>}[rr]^{d_0^2} && K_1\ar[llu]_{d_1^1}\ar[dll]^{d_0^1}\\ &&(M,\mathcal{F})&& }}
 \end{equation}
By definition, $W_2'$ is in particular a bi-submersion between the respective bi-vertical singular foliations of $ K_1$ and $\Lambda^2_0 K$ (which exists since the sequence $\Gamma(E_{-2})\stackrel{\dd}{\rightarrow}\Gamma(E_{-1})\stackrel{\rho}{\rightarrow}\mathcal{F}$ in exact in the middle, see Lemma \ref{lem:anchoredbundleforbivertical}).\\

\textbf{Step 2: making this candidate of constant dimension: application of Theorem \ref{th:reducedimension}.}

Since $$\mathrm{dim}(\Lambda^2_0 K) =  2\,  \mathrm{dim}(K_1)- \mathrm{dim}(M) =  \mathrm{dim}(M) +  2 \, \mathrm{rk}(E_{-1})$$ 
and since the rank of the bi-vertical singular foliation $\mathcal{BV}_{K_1}$ on 
$K_1$ is lower or equal to $\mathrm{rk}(E_{-2}) $, Theorem \ref{th:reducedimension} and Remark \ref{rk:cut-dim} imply that we can replace $W_2'$ by another total relation $W_2$ which is a bi-submersion:
 $$   \left(\Lambda^2_{0}K, \mathcal{BV}_{\Lambda^2_{0}K}\right)\stackrel{p_0^2}{\longleftarrow} W_2^0\stackrel{d_0^2}{\longrightarrow}\left(K_1,\mathcal{BV}_{K_1}\right) $$
and a total relation:
\begin{equation}\label{diag:S_2^{0}1}
    \scalebox{0.7}{ \xymatrix{&& (M,\mathcal{F})&& \\ \Lambda^2_0 K\ar[urr]^{d_0^1\circ\mathrm{pr}_2}\ar[drr]_{d_0^1\circ\mathrm{pr}_1}&&\ar@{->>}[ll]_<<<<<<<<<<{p_0^2} W_2^0\ar@{..>}[u]\ar@{..>}[d]\ar@{->>}[rr]^{d_0^2} && K_1\ar[llu]_{d_1^1}\ar[dll]^{d_0^1}\\ &&(M,\mathcal{F})&& }}
 \end{equation} 
 but whose dimension is now:
 $$  \mathrm{dim}(W_2^0) =  \mathrm{dim}(M) +  2\mathrm{rk}(E_{-1}) + \mathrm{rk}(E_{-2}).$$\\

 \textbf{Step 2'}. Since the map
\[
\phi_0 \colon K_1 \longrightarrow \Lambda_0^2 K, \qquad 
x \longmapsto \bigl(x,\, s_0^0 \circ d_1^1(x)\bigr),
\]
is a morphism of bi-submersions, every point \(x \in K_1\) is equivalent to the point 
\(\bigl(x,\, s_0^0 \circ d_1^1(x)\bigr)\in \Lambda_0^2 K\). 
We may choose \(W_2^0\) so that it contains \(K_1\), which enables us to define an injection
\(
s_0^1 \colon K_1 \longrightarrow W_2^0
\)
compatible with the structure maps in the commutative diagram~\eqref{diag:S_2^{0}1}, i.e., such that the following diagram commutes \[\xymatrix{K_1&W_2^0\ar[l]_{d_0^2}\ar[rd]^{p_0^2}&\\K_1\ar[u]^{\mathrm{Id}}\ar[rr]^{\phi_0}\ar[ru]^{s_0^1}&&\Lambda^2_0K}\] 
This construction is detailed in Lemma~\ref{lem:Kan2} below.\\

 \textbf{Step 3: checking that this candidate admits natural face maps that satisfy the required
conditions (but no degeneracy maps yet)}.

Now, composing the submersion $ W^0_2\stackrel{p^2_0}{\longrightarrow}\Lambda^2_0 K= K_1 * K_1^{-1}$  with the projections on the first and second factor, one obtains two submersions $W_2^0\to K_1$ that we denote by $ d_1^2$ and $ d_2^2$. The maps $d_0^2, d_1^2, d_2^2$ satisfy the required simplicial relations by commutativity of Diagram \eqref{diag:S_2^{0}1}. Notice that the horn projection $p^2_0\colon W_2^0\to \Lambda_0^2K$ is a surjective submersion by construction.  We will show in \S\ref{sec:Ngeq3} that the Kan conditions are also satisfied; that is, the remaining horn projections 
\(p^2_1 = (d_0^2, d_2^2)\) and \(p^2_2 = (d_0^2, d_1^2)\) are surjective submersions.
Note that \(W_2^0\) is not yet our definitive \(K_2\). It will be enlarged in order to define all degeneracy maps.\\

\textbf{Step 4: construction of our candidate for $K_2$ as disjoint unions of equivalences,
so that degeneracies now exist.}

As indicated above, the construction of $W_2^0$ is not yet sufficient. Regarding the para-degeneracy maps, there is no straightforward way to define them solely based on an equivalence between $K_1$ and the horn $\Lambda^2_0K$. We will need to consider equivalences involving the other horns $\left(\Lambda^2_iK\right)_{i=1,2}$. Eventually, we take $K_2=\coprod_{i=0}^2 W_2^i$ to be a disjoint union of total relations  $\left(\Lambda^2_{i}K, \mathcal{BV}_{\Lambda^2_{i}K}\right)\stackrel{p_i^2}{\longleftarrow} {W}^i_2\stackrel{d_i^2}{\longrightarrow}\left(K_1,\mathcal{BV}_{K_1}\right)$.

We establish the following result, which is a special case of Proposition \ref{prop:up-to-degreeN}. 

\begin{remark}
    We could have considered equivalences with two horns spaces. However, we take all three because this ensures that the three projection  maps $\left(p_\ell^2\colon K_2\to \Lambda^2_\ell K\right)_{\ell=0,1,2}$ are systematically surjective.
\end{remark}

\begin{proposition}
    \label{prop:up-to-degree2}
Let $(M,\mathcal{F})$ be a singular foliation  admitting a geometric resolution $(E_\bullet, \dd, \rho)$. Any holonomy para-Lie $\infty$-groupoid up to order $1$ extends to a holonomy para-Lie $\infty$-groupoid up to order~$2$ \begin{equation}\label{eq:inftygroupoid-up-to-2}
    \mathcal{K}_\bullet(\mathcal{F})_{\leq 2}\colon \xymatrix{&K_2\ar@<-2pt>[r]\ar@<6pt>[r]\ar@<2pt>[r]&\ar@/^{0.8pc}/[l]\ar@/^{1.2pc}/[l] K_1 \ar@<-2pt>[r]\ar@<2pt>[r] & \ar@/^{0.8pc}/[l]M}
   \end{equation}
    that satisfies conditions 1 
    in Theorem \ref{int:theorem}.
\end{proposition}

{The proof involves  a delicate lemma, whose proof is a particular case of  Lemma \ref{lem:KanN}},\S \ref{sec:Ngeq3} and will be proven later in full generality:
\begin{lemma}\label{lem:Kan2}
\begin{enumerate}
\item {For $\ell=0,1,2$,  and every $x\in K_2$, the leaves of the bi-vertical singular foliations through $d_\ell^2(x)$ and  $p_\ell^2(x)\in\Lambda^2_\ell K$  have the same codimension.}\item For every $\ell=0,1,2$, $\left(\Lambda^2_{\ell}K, \mathcal{BV}_{\Lambda^2_{\ell}K}\right)\stackrel{p_\ell^2}{\longleftarrow} K_2\stackrel{d_\ell^2}{\longrightarrow}\left(K_1,\mathcal{BV}_{K_1}\right)$ is a total relation between the bi-submersions  $\Lambda^2_{\ell}K$ and $K_1$. {Moreover, there exist two smooths maps $s_0^1\colon K_1\to W^0_2$ and $s_1^1\colon K_1\to W^1_2$ such that for all $x\in K_1$,} \begin{enumerate}
    \item [$(a)$] $d_0^2(s_0^1(x))=x$ and $p_0^2(s_0^1(x))=(x,s_0^0d_1^1(x))$; 

    \item[$(b)$] $d_1^2(s_1^1(x))=x$ and $p_1^2(s_1^1(x))=(s_0^0d_0^1(x),x)$.
\end{enumerate}
\end{enumerate}
\end{lemma}
\begin{remark}
The maps \(s_0^1, s_1^1 \colon K_1 \to K_2\) of Lemma~\ref{lem:Kan2} take values in two distinct, disjoint components \(W_2^0, W_2^1 \subset K_2\). Consequently, the identity \(s_0^1 \circ s_0^0=s_1^1 \circ s_0^0\) does not hold.

\end{remark}

Let us now prove Proposition \ref{prop:up-to-degree2}.
\begin{proof}
The axioms of a para-Lie $\infty $-groupoid up to order $2$ in Definition \ref{def:para-Grp-upto} are satisfied by 
Lemma \ref{lem:Kan2}, since the latter Lemma gives the para-simplical maps and implies  the Kan conditions, namely that the projections on the horns are surjective submersions. Let us check that it is a holonomy para-Lie $\infty $-groupoids up to order $2$, i.e., that the three first  items in
Definition \ref{def:universal-Grpoid} are satisfied for $ k=1$ and $ k=2$ and that the fourth item holds for $ k=1$.
For $ k=1$, the two first items holds in view of Lemma \ref{lem:step1}. The third item holds because $ K_1$ is an atlas, and the fourth item holds by construction of $ K_2$. For $ k=2$,  the fourth item is void, and we have to check the first three items. The first one follows from Lemma \ref{lem:anchoredbundleforbivertical}, applied to $E_{-2}\stackrel{\dd^{(2)}}{\longrightarrow}E_{-1} \stackrel{\rho}{\longrightarrow} TM$ with $\ker\rho=\mathrm{im}(\dd^{(2)})$. Here, $\nu^\mathcal
F$ is as in Equation \eqref{eq:leftaction}.

The pullback vector bundle $(d^1_0)^*E_{-2}\to K_1$ together with the composition $\xymatrix{(d_0^1)^*E_{-2}\ar[rr]^{\nu^{\mathcal{F}}\circ \dd^{(2)}}&&TK_1}$ is an anchored bundle over 
 the bi-vertical singular foliation $\mathcal{BV}_{K_1}$ in the notations of Lemma \ref{lem:anchoredbundleforbivertical}. In addition, the pullback\begin{equation}\label{eq:geom-resol_2}
    \cdots \stackrel{}{\longrightarrow} (d^{1}_0)^*E_{-3}\stackrel{}{\longrightarrow}(d^{1}_0)^* E_{-2}\stackrel{\nu^{\mathcal{F}}\circ \dd^{(2)}}{\longrightarrow} TK_{1}\end{equation}
is a geometric resolution of the bi-vertical foliation $\mathcal{BV}_{K_1}$ on $K_1$.

The second item holds by Lemma \ref{lem:Kan2} which implies that
 $\Lambda_\ell^2 K \stackrel{p_\ell^{2}}{\longleftarrow} K_2  \stackrel{d_\ell^{2}}{\longrightarrow} K_{1}$ is a bi-submersion for the respective bi-vertical foliations for every $\ell=0,1,2$. 

The spans 

\begin{align*}
         \xymatrix{\left(\Lambda^2_{0}K, \mathcal{BV}_{\Lambda^2_{0}K}\right)&&& \Lambda^{3}_0 K\ar[lll]_<<<<<<<<<<<<<<{d_{0}^2\circ \mathrm{pr}_{2}\times d_{0}^2\circ \mathrm{pr}_{3}}\ar[rrr]^{d_{0}^2\circ \mathrm{pr}_{1}}&&& (K_{1},\mathcal{BV}_{K_{1}})}\\\xymatrix{\left(\Lambda^2_{1}K, \mathcal{BV}_{\Lambda^2_{1}K}\right)&&& \Lambda^{3}_1 K\ar[lll]_<<<<<<<<<<<<<<{d_{0}^2\circ \mathrm{pr}_{1}\times d_{1}^2\circ \mathrm{pr}_{3}}\ar[rrr]^{d_{1}^2\circ \mathrm{pr}_{2}}&&& (K_{1},\mathcal{BV}_{K_{1}})}\\  \xymatrix{\left(\Lambda^2_{2}K,\mathcal{BV}_{\Lambda^2_{2}K}\right)&&& \Lambda^{3}_2 K\ar[lll]_<<<<<<<<<<<<<<{d_{1}^2\circ \mathrm{pr}_{1}\times d_{2}^2\circ \mathrm{pr}_{2}}\ar[rrr]^{d_{2}^2\circ \mathrm{pr}_{3}}&&& (K_{1},\mathcal{BV}_{K_{1}})}\\\xymatrix{\left(\Lambda^2_{2}K, \mathcal{BV}_{\Lambda^2_{2}K}\right)&&&\Lambda^{3}_{3} K\ar[lll]_<<<<<<<<<<<<<<{d_{2}^2\circ \mathrm{pr}_{1}\times d_{2}^2\circ \mathrm{pr}_{2}}\ar[rrr]^{d_{2}^2\circ \mathrm{pr}_{3}}&&& (K_{1},\mathcal{BV}_{K_{1}}}).
    \end{align*}
 are bi-submersions since they can be written as the fiber product of a bi-submersion and a Morita-equivalence over $\mathcal{F}$ namely, $\Lambda^3_\ell K=K_2\times_{\Lambda^2_\ell K}\left(\Lambda^3_{\{\ell, \ell+1\}}K\right)$. A detailed computation is given in \S \ref{sec:Ngeq3}.

 Condition 1 in Theorem \ref{int:theorem} holds by construction. This completes the proof.
\end{proof}

\subsection{A holonomy para-Lie $\infty$-groupoid up to order $N$}\label{sec:Ngeq3}

Let us continue the recursive construction of the  holonomy para-Lie $\infty$-groupoid of a singular foliation.

The following proposition is the main result of this subsection, and its proof is deferred to the end of this section.
\begin{proposition}
    \label{prop:up-to-degreeN}
Let $(M,\mathcal{F})$ be a singular foliation admitting a geometric resolution $(E_\bullet, \dd, \rho)$. Any holonomy para-Lie $\infty$-groupoid up to order $N$ that satisfies condition 1 in Theorem \ref{int:theorem}
%
  extends to a holonomy para-Lie $\infty$-groupoid up to order $N+1$  that satisfies condition 1 
    in Theorem \ref{int:theorem}.
\end{proposition}


Assume that a holonomy para-Lie $\infty$-groupoid up to order $N$ $$\mathcal{K}_{\leq N}(\mathcal{F}):\xymatrix{&K_N\ar@<5pt>[r]\ar@<-5pt>@{->}[r]\ar@<1pt>[r]\ar@<-2pt>@{.}[r]&\ar@/^{0.8pc}/[l]\ar@/^{1.2pc}/[l]\ar@{.}@/^{1.4pc}/[l] \ar@/^{1.7pc}/[l]\cdots\; \ar@<5pt>[r]\ar@<-5pt>@{->}[r]\ar@<1pt>[r]\ar@<-2pt>@{.}[r]&\ar@/^{0.8pc}/[l]\ar@/^{1.1pc}/[l]\ar@{-}@/^{1.4pc}/[l] \ar@/^{1.7pc}/[l]K_3\ar@<5pt>[r]\ar@<-5pt>@{->}[r]\ar@<1pt>[r]\ar@<-2pt>[r]&\ar@/^{0.8pc}/[l]\ar@/^{1.2pc}/[l]\ar@/^{1.6pc}/[l]K_2\ar@<-2pt>[r]\ar@<6pt>[r]\ar@<2pt>[r]&\ar@/^{0.8pc}/[l]\ar@/^{1.2pc}/[l]K_1 \ar@<-2pt>[r]\ar@<2pt>[r]& \ar@/^{0.8pc}/[l]K_0}$$

\vspace{0.4cm}

that satisfies item 1 of Theorem \ref{int:theorem}  has been constructed.

{\textbf{Step 1: construction of a first candidate for $K_{N+1}$ using the recursion assumption, but whose dimension is not constant and there is no  degeneracies: application of Theorem \ref{thm:bi-sub-equiv}.}}

By the recursion assumption, $K_N$ and $\Lambda_\ell^{N+1}K$ for $\ell=0,\ldots, N+1$ are equivalent bisubmersions for all $\ell$. 

By Theorem \ref{thm:bi-sub-equiv}, there exists a bi-submersion $$\xymatrix{\left(\Lambda^{N+1}_\ell K, \mathcal{BV}_{\Lambda^{N+1}_\ell K}\right)&& {W'}_{N+1}^{\ell}\ar[ll]_<<<<<<<<<{p_\ell^{N+1}}\ar[rr]^{d_\ell^{N+1}}&& (K_{N},\mathcal{BV}_{K_{N}}})$$ which is an equivalence of bi-submersions, i.e., the  following diagrams commute:
 \begin{align}
     \label{diag:equivalence_{N+1}}
    \scalebox{0.7}{ \xymatrix{&& \left(K_{N},{{\mathcal{BV}}}_{K_{N}}\right)&& \\K_N\ar[urr]^{d^N_{\ell}}\ar[drr]_<<<<<<<<<<<<<<<<<<{p_{\ell}^N}&&\ar@{->>}[ll]_{d_\ell^{N+1}}{W'}_{N+1}^{\ell}\ar@{..>}[u]\ar@{..>}[d]\ar@{->>}[rr]^{p_\ell^{N+1}} && \Lambda_\ell^{N+1}K\ar[llu]_{d_{\ell}^N\circ\mathrm{pr}_{\ell+1}}\ar[dll]^<<<<<<<<<<<{d_{\ell-1}^N\circ \mathrm{pr}_{i\leq \ell}\times d_{\ell}^N\circ \mathrm{pr}_{i>\ell+1}}\\ && \left(\Lambda_{\ell}^NK,\, {\mathcal{BV}}_{\Lambda_{\ell}^NK}\right)&& }}\qquad  
    \scalebox{0.7}{ \xymatrix{&& \left(K_{N},{{\mathcal{BV}}}_{K_{N}}\right)&& \\K_N\ar[urr]^{d^N_{N}}\ar[drr]_<<<<<<<<<<<<<<<<<<{p_{N}^N}&&\ar@{->>}[ll]_{d_{N+1}^{N+1}}{W'}_{N+1}^{{N+1}}\ar@{..>}[u]\ar@{..>}[d]\ar@{->>}[rr]^{p_{N+1}^{N+1}} && \Lambda_{N+1}^{N+1}K\ar[llu]_{d_{N}^N\circ\mathrm{pr}_{N+1}}\ar[dll]^<<<<<<<<<<<{d_{N}^N\circ \mathrm{pr}_{i\leq N}}\\ && \left(\Lambda_{N}^NK,\, {\mathcal{BV}}_{\Lambda_{N}^NK}\right)&& }}
 \end{align}
  for every $0\leq \ell\leq N+1$.

   \begin{remark}\label{rmk:equiv-ell}
 In fact, ${W'}_{N+1}^{\ell}$ can be assumed to be  of the form $\mathcal T_{N+1}\times _{K_N} V_\ell$ with \begin{equation}
     \label{equiv-ell}V_\ell\subset \coprod_{\begin{array}{c}
          \phi\colon \Sigma\subset K_N\leftrightarrows \Lambda^{N+1}_\ell K \\
          \texttt{countably many}
     \end{array}}\mathrm{Graph}(\phi)\times \left(\mathcal B^{n_{N-1}}\right)^2\times  \mathcal B^{n_{N}},
 \end{equation} being a disjoint union of open neighborhoods of the $\mathrm{Graph}(\phi)$'s in $\mathrm{Graph}(\phi)\times \left(\mathcal B^{n_{N-1}}\right)^2\times  \mathcal B^{n_{N}}$, see \S \ref{sec:totalrelation}. Here 
 \begin{itemize}
     \item [-]$n_{N-1}$ is a common number of generators of the associate bi-vertical singular foliations on $K_{N-1}$ and on $\Lambda^{N}_\ell K$ for $0\leq \ell\leq N$; 
     \item [-]$n_N$ is a common number of generators of the associate bi-vertical singular foliations on $K_{N}$ and on $\Lambda^{N+1}_\ell K$ for $0\leq \ell\leq N+1$ 
     
     \item [-]and the $\phi$'s are the restrictions of local bi-submersion morphisms $K_N\leftrightarrows \Lambda^{N+1}_\ell K$ to a bi-transversal $\Sigma$ at a point of $K_N$ or $\Lambda^{N+1}_\ell K$. 
     \item [-]Also, $T_{N+1}\rightrightarrows K_N$ is global leaf preserving bi-submersion over $\mathcal{BV}_{K_N}$.
 \end{itemize}
  \end{remark}

{\textbf{Step 2: Replacing this candidate by another one of constant dimension (that we compute): application of Theorem \ref{th:reducedimension}.}}
  By Theorem \ref{thm:bi-sub-equiv},  one can assume that  ${W'}_{N+1}^{\ell}$ is a total relation, and by Theorem \ref{th:reducedimension} we can assume that ${W'}_{N+1}^{\ell}$ is a  disjoint union of manifolds all the same dimension, which can be chosen to be any integer greater or equal to the maximal codimension of the leaves of the bi-vertical vector fields on $ {W'}_{N+1}^{\ell}$. This construction leaves the preceding spaces unchanged. Let us compute an upper bound of this integer:
  \begin{lemma}
    The codimensions of the leaves of the bi-vertical singular foliation on ${W}_{N+1}^{\ell}$ are bounded by the sum of 
  \begin{enumerate}
      \item the dimension of $\Lambda^{N+1}_0 K$
      \item the rank of the bi-vertical foliation on $K_N$, which is itself bounded by  the rank of $E_{-N-1}$.
  \end{enumerate}
  \end{lemma}
  \begin{proof}
      By the recursion assumption, there exists a vector bundle isomorphism $\nu^{N}\colon \left(d_0^1\circ \cdots \circ d_0^N\right)^*E_{-N}\to \ker Td_0^N$ such that $Td_0^N\circ \nu^{N}=\left(d_0^1\circ \cdots \circ d_0^N\right)^*\dd^{(N)}$. The pullback  $(d_0^1\circ \cdots \circ d^N_0)^*E_{-N-1}\to K_N$ together with the composition \begin{equation}
          \label{eq:anchored_{N+1}}\nu^{N}\circ \dd^{(N+1)} \colon \xymatrix{\left(d_0^1\circ \cdots \circ d_0^N\right)^*E_{-N-1}\ar[r]^{}&TK_N}
      \end{equation} is an anchored bundle over 
 the bi-vertical singular foliation $\mathcal{BV}_{K_N}$,  by Lemma \ref{lem:anchoredbundleforbivertical}. In addition, the pullback \begin{equation}\label{eq:geom-resol_{N+1}}
    \cdots \stackrel{}{\longrightarrow} \left(d_0^1\circ \cdots \circ d^{N}_0\right)^*E_{-N-2}\stackrel{}{\longrightarrow}\left(d_0^1\circ \cdots \circ d^{N}_0\right)^* E_{-N-1}\stackrel{}{\longrightarrow} TK_{N}\end{equation}
is a geometric resolution of the bi-vertical foliation $\mathcal{BV}_{K_N}$ on $K_N$. {Since $W'^\ell_{N+1}$ is a bi-submersion between $K_N$ and $\Lambda^{N+1}_0K$, the codimension of the leaves of the bi-vertical singular foliation on $W'^\ell_{N+1}$ are bounded by $\mathrm{rk}(E_{-N-1})+\dim \Lambda^{N+1}_0K$ according to Remark \ref{rk:cut-dim}. }
\end{proof}

Since \eqref{eq:anchored_{N+1}} is an anchored bundle over $\mathcal{BV}_{K_{N}}$, we can take the leaf preserving bi-submersion \begin{equation}\label{eq:T_{N+1}}
    \mathcal{T}_{N+1}:=\mathcal{U}_{\left(d_0^1\circ \cdots \circ d_0^N\right)^*E_{-N-1}}
\end{equation} of Remark \ref{rmk:equiv-ell} be the bi-submersion over $\mathcal {BV}_{K_N}$, where $\mathcal{U}_{\left(d_0^1\circ \cdots \circ d_0^N\right)^*E_{-N-1}}$ is a neighborhood of the zero section of the pullback vector bundle $\left(d_0^1\circ \cdots \circ d_0^N\right)^*E_{-N-1}\to  K_N$ on which there is a bi-submersion over $\mathcal{BV}_{K_N}$ as in Example \ref{ex:holonomy-biss2}. By construction, we have $\left(d_0^1\circ \cdots \circ d_0^{N+1}\right)^*E_{-N-1}\subset \ker Tp_0^{N+1}$.

   By Theorem \ref{th:reducedimension} and Remark \ref{rk:cut-dim}, we can therefore replace ${W'}^\ell_{N+1}$ by another total relation $W_{N+1}^\ell$ whose dimension $\dim W_{N+1}^\ell=\mathrm{rk}(E_{-N-1})+ \dim (\Lambda^{N+1}_0K)$ so that $\left(d_0^1\circ \cdots \circ d_0^{N+1}\right)^*E_{-N-1}=\ker Tp_0^{N+1}\to {W}^\ell_{N+1}$.

{\textbf{Step 3: checking that this candidate admits natural face maps that satisfy the required conditions (but no degeneracy maps yet)}}
 For every $0\leq \ell \leq N+1$, let ${W}_{N+1}^{\ell}$ be as in Diagram \eqref{diag:equivalence_{N+1}}, we define the following maps on ${W}_{N+1}^{\ell}$ : 
$$d_i^{N+1}:=\begin{cases}
     \mathrm{pr}_{i+1}\circ p_\ell^{N+1}\colon {W}_{N+1}^{\ell}\to K_N, & \text{if}\;\; i=0,\ldots,\ell-1\\&\\d_\ell^{N+1}\colon {W}_{N+1}^{\ell}\to K_N, & \text{if\; $i=\ell$, is as in Diagram \eqref{diag:equivalence_{N+1}}}\\&\\\mathrm{pr}_i\circ p_0^{N+1}\colon {W}_{N+1}^{0}\to K_N, & \text{if}\,\; i=\ell+1,\ldots,N+1.
 \end{cases}$$
 \begin{lemma}\label{lem:face-face}
     For every $\ell=0,\ldots, N+1$, the maps $\left(d_i^{N+1}\colon {W}^\ell_{N+1}\to K_N\right)_{i=0,\ldots, N+1}$ satisfy the following identities $$d_i^{N}\circ d_j^{N+1}=d_{j-1}^N\circ d_i^{N+1}$$ for all $0\leq i<j\leq N+1$.
 \end{lemma}
\begin{proof}
    \underline{The cases involving the face map} $d_\ell^{N+1}$: the commutativity of Diagram \eqref{diag:equivalence_{N+1}}  reads \begin{itemize}
     
     \item[]   \begin{equation}\label{eq1}
               d^{N}_{\ell}\circ d^{N+1}_\ell=(d^{N}_\ell\circ \mathrm{pr}_{\ell+1})\circ p_\ell^{N+1}
           \end{equation}
     \item[] and \begin{align}
               \nonumber p_{\ell}^{N}\circ d^{N+1}_\ell&=\left(d_{\ell-1}^N\circ \mathrm{pr}_{i\leq \ell}\times d_{\ell}^N\circ \mathrm{pr}_{i>\ell+1}\right)\circ p_\ell^{N+1}\\&\label{eq2}=\left(d^{N}_{\ell-1}\circ d_0^{N+1},\ldots,{d^{N}_{\ell-1}\circ d_{\ell-1}^{N+1}}, d^{N}_{\ell}\circ d_{\ell+2}^{N+1},\ldots, d^{N}_{\ell}\circ d_{N+1}^{N+1}\right)
           \end{align}
        \end{itemize}
This implies that

\begin{enumerate}
       \item for $0\leq i\leq \ell-1$ and $j=\ell$ : $\mathrm{pr}_{i+1}\circ p^N_{\ell}\circ d^{N+1}_\ell=d^{N}_{\ell-1}\circ d_{i}^{N+1}$, using  Equation \eqref{eq2}. Therefore, $$d_{i}^N\circ d^{N+1}_\ell=d^{N}_{\ell-1}\circ d_{i}^{N+1}.$$

 \item $i=\ell$ and $j=\ell+1$ :  $d^{N}_{\ell}\circ d^{N+1}_\ell=(d^{N}_{\ell}\circ \mathrm{pr}_{\ell+1})\circ p_\ell^{N+1}=d^{N}_{\ell}\circ d_{\ell+1}^{N+1}$,\; using  Equation \eqref{eq1}.

\item  $i=\ell$ and $\ell+2 \leq j\leq N+1$  : $\mathrm{pr}_{j-1}\circ p^N_{\ell}\circ d^{N+1}_\ell=d^{N}_{\ell}\circ d_{j}^{N+1}$, using  Equation \eqref{eq2}. Therefore, $$d_{j-1}^N\circ d^{N+1}_\ell=d^{N}_{\ell}\circ d_{j}^{N+1}.$$

\end{enumerate}

\noindent
\underline{The cases without the face map} $d_\ell^{N+1}$:  for $i<j\neq\ell$ and by definition of $\Lambda^{N+1}_\ell K$ we have : 

$$d_i^N\circ d_j^{N+1}=\begin{cases}
     d_i^N\circ(\mathrm{pr}_{j+1}\circ p_\ell^{N+1})= d_{j-1}^N\circ(\mathrm{pr}_{i+1}\circ p_\ell^{N+1})=d_{j-1}^N\circ d_{i}^{N+1},\; \text{for}\; 0\leq j<\ell\\\\d_i^N\circ(\mathrm{pr}_{j}\circ p_\ell^{N+1})= \begin{cases}
         d_{j-1}^N\circ(\mathrm{pr}_{i+1}\circ p_\ell^{N+1})=d_{j-1}^N\circ d_{i}^{N+1},\;  \text{for}\,\; i<\ell+1\leq j\leq N+1\\ \\        d_{j-1}^N\circ(\mathrm{pr}_{i}\circ p_\ell^{N+1})=d_{j-1}^N\circ d_{i}^{N+1},\;  \text{for}\,\; \ell+1\leq i< j\leq N+1.
     \end{cases}
 \end{cases}$$
\end{proof}

{\textbf{Step 4: construction of our candidate for $K_{N+1}$ as disjoint unions of the previous ones, so that degeneracies now exist.}}

We define  $\displaystyle{K_{N+1}:= \coprod_{\ell=0}^{N+1} {W}_{N+1}^{\ell}}$ to be the disjoint union of these equivalences.


Let us check that all axioms are satisfied, that is

\begin{enumerate}
\item $\dim K_{N+1}=\mathrm{rk}(E_{-N-1})+\dim \Lambda^{N+1}_0K$: this hold my construction.
    \item the horn spaces in $K_{N+1}$, i.e., $\left(\Lambda^{N+2}_\ell K\right)_{\ell=0,\ldots,N+2}$ are bi-submersions between bi-vertical singular foliations $\mathcal{BV}_{K_{N+1}}$ and $\mathcal{BV}_{\Lambda_\ell^{N+1} K}$ for $\ell=0,\ldots,N+1$: this  holds by a direct computation.
\item the bi-vertical singular foliations on $K_{N+1}$ and $\Lambda^{N+2}_\ell K$ exist for $\ell=0,\ldots,N+2$: this will be a consequence of the existence of a geometric resolution $$ \cdots \stackrel{}{\longrightarrow} (d_0^1\circ \cdots \circ d_0^{N+1})^*E_{-N-2}\stackrel{}{\longrightarrow} TK_{N+1}$$of $\mathcal{BV}_{K_{N+1}}$.

    \item the horn space $\Lambda^{N+2}_\ell K$ is equivalent to $K_{N+2}$ for all $\ell=0,\ldots,N+2.$
\end{enumerate}

{$K_{N+1}$ can be chosen so that every point in the horn space $\Lambda^{N+2}_\ell K$ is related to a point in $K_{N+1}$ for all $\ell=0,\ldots,N+2$.} To see this,  we first need the following lemma.
\begin{lemma}
If for $k=0,\ldots, N+2$, $K_{N}\stackrel{\;\;d^{N+1}_{\ell}}{\longleftarrow}W^{k}\stackrel{p_{
\ell}^{N+1}}{\longrightarrow}\Lambda_\ell^{N+1}K$ is a bi-submersion between the bi-vertical singular foliations $\mathcal{BV}_{K_N}$ and $\mathcal{BV}_{\Lambda^{N+1}_\ell K}$ for $\ell=0,\ldots, N+1$, then $$\Lambda^{N+2}_\ell(W^0,\ldots, \widehat{W^\ell},\ldots,W^{N+3}):=\left\{ (x_0,\ldots, \widehat{x_\ell},\ldots,x_{N+2})\in W^1\times \cdots\times W^{N+3} \mid d^{N+1}_ix_j=d^{N+1}_{j-1}x_i,\; i<j\right\}$$ is a bi-submersion between $\mathcal{BV}_{K_N}$ and $\mathcal{BV}_{\Lambda^{N+1}_\ell K}$ for $\ell=0,\ldots, N+1$.
\end{lemma}

\begin{proof}
    It follows the same line as in the proof of  Lemma \ref{lem:hornN} below, when $W^k=K_{N+1}$ for $k=0,\ldots,N+2$.
\end{proof}

\begin{lemma}\label{lem:higher-mult}
    We can choose $K_{N+1}$ such that there exists for all $\ell =0,\dots, N $ a local morphism of bi-submersions $\Lambda^{N+2}_\ell K\hookrightarrow K_{N+1}$.
\end{lemma}
\begin{proof}
    Let $K_{N}\stackrel{\;\;d^{N+1}_{\ell}}{\longleftarrow} W_{N+1}\stackrel{p_{
\ell}^{N+1}}{\longrightarrow}\Lambda_\ell^{N+1}K$ be a bi-submersion between the bi-vertical singular foliations $\mathcal{BV}_{K_N}$ and $\mathcal{BV}_{\Lambda^{N+1}_\ell K}$ for $\ell=0,\ldots, N+1$. We denote $\Lambda^{N+2}_\ell W:=\Lambda^{N+2}_\ell(W,\ldots, W)$  for a bi-submersion $W$ between $\mathcal{BV}_{K_N}$ and $\mathcal{BV}_{\Lambda^{N+1}_\ell K}$ for $\ell=0,\ldots, N+1$.
       \begin{enumerate}
            \item Consider $K_{N+1}'$ the disjoint union of $W_{N+1}$ with $$\Lambda^{N+2}_\ell\left(\Lambda^{N+2}_{\ell_n} \Lambda^{N+2}_{\ell_{n-1}}  \cdots \Lambda^{N+2}_{\ell_1}W_{N+1},\ldots, \underbrace{\widehat{W_{N+1}}}_{\ell\text{th-place}},\ldots,\Lambda^{N+2}_{\ell'_n} \Lambda^{N+2}_{\ell'_{n-1}}  \cdots \Lambda^{N+2}_{\ell'_1}W_{N+1}\right)$$

            for $\ell=0, \ldots, N+2$.
        \item $K'_{N+1}$ is not a manifold, but each of its connected component is a bi-submersion between the bi-vertical singular foliations of $\Lambda^{N+1}_\ell K$ and $K_N$ for $\ell=0,\ldots,N+1$. By Corollary \ref{cor:reducedimension}, $K_{N+1}'$ contains a finite dimensional manifold that we  denote by $K_{N+1}$  which is a bi-submersion between the bi-vertical singular foliations of $\Lambda^{N+1}_\ell K$ and $K_N$ for $\ell=0,\ldots,N+1$. By construction, \begin{enumerate}
            \item $K_{N+1}$ an equivalence of bi-submersions between $\Lambda^{N+1}_\ell K$ and $K_N$ for  $\ell=0,\ldots,N+1$;
            \item and every point in $\Lambda^{N+1}_\ell K$ is related  to a point in  $K_{N+1}$ for all $\ell=0,\ldots,N+1$.
        \end{enumerate}
       \end{enumerate}
\end{proof}

The following lemma is crucial, as it establishes the Kan condition and provides the degeneracy maps.
\begin{lemma}\label{lem:KanN}
    \begin{enumerate}

     \item {For $\ell=0,\ldots,N+1$,  and every $x\in K_{N+1}$, the leaves of the bi-vertical singular foliations through $d_\ell^{N+1}(x)$ and  $p_\ell^{N+1}(x)\in\Lambda^{N+1}_\ell K$  have the same codimension.}
    
    \item For every $\ell=0,\ldots,N+1$, $\left(\Lambda^{N+1}_{\ell}K, \mathcal{BV}_{\Lambda^{N+1}_{\ell}K}\right)\stackrel{\;\;\;\;p_\ell^{N+1}}{\longleftarrow} K_{N+1}\stackrel{d_\ell^{N+1}}{\longrightarrow}\left(K_N,\mathcal{BV}_{K_N}\right)$ is a total relation between the bi-submersions $\Lambda^{N+1}_{\ell}K$ and $K_N$. 

    \item Moreover, for every $\ell\in \{0, \ldots, N\}$, the map $\phi_\ell\colon K_N\to \Lambda_\ell^{N+1}K$  defined as\begin{equation}
        \label{eq:phi_ell}x\mapsto \begin{cases}
         \left(x, s_0^{N-1}d_{1}^N(x),\ldots, s_{0}^{N-1}d_{N}^N(x)\right) & \text{for}\; \ell=0\\\\ \left( s_{\ell-1}^{N-1}d_0^N(x),\ldots, s_{\ell-1}^{N-1}d_{\ell-1}^N(x),x, s_\ell^{N-1}d_{\ell+1}^N(x),\ldots, s_{\ell}^{N-1}d_{N}^N(x)\right) & \text{for}\; 0<\ell< N\\\\ \left( s_{N-1}^{N-1}d_0^N(x),\ldots, s_{N-1}^{N-1} d_{N-1}^N(x),x\right) & \text{for}\; \ell=N
     \end{cases} \end{equation}is a morphism of bi-submersions and there exist $N+1$ smooth maps $\left(s_\ell^N\colon K_N\to K_{N+1}\right)_{\ell=0,\ldots,N}$ such that  the following diagram commutes \begin{equation}\label{eq:phi_ell2}
         \xymatrix{K_N&W_{N+1}^\ell \ar[l]_{d_\ell ^{N+1}}\ar[rd]^{p_\ell^{N+1}}&\\K_N\ar[u]^{\mathrm{Id}}\ar[rr]^{\phi_\ell}\ar[ru]^{s_\ell^N}&&\Lambda^{N+1}_\ell K}
     \end{equation} 
    
\end{enumerate}\end{lemma}
\begin{remark}
        Equation \eqref{eq:phi_ell2} in Lemma \ref{lem:KanN} implies that $d^{N+1}_\ell\circ s_\ell^N=\mathrm{id}=d^{N+1}_{\ell+1}\circ s_\ell^N$ and \begin{align*}
    d^{N+1}_i\circ s_\ell^N&=s_{\ell-1}^{N-1}\circ d_{i}^N\qquad \text{for}\; 0\leq i< \ell\\ d^{N+1}_i\circ s_\ell^N&=s_\ell^{N-1}\circ d_{i-1}^N \qquad \text{for}\; i>\ell+1. 
\end{align*}Indeed, 
\begin{enumerate}
    \item for $0\leq i< \ell$ :  
   by definition of $d_i^{N+1}$ we have \begin{align*} 
        d^{N+1}_i\circ s_\ell^N&=\mathrm{pr}_{i+1}\left(s_{\ell-1}^{N-1}\circ d_0^N,\ldots, s_{\ell-1}^{N-1}\circ d_{\ell-1}^N,\mathrm{id}, s_\ell^{N-1}\circ d_{\ell+1}^N,\ldots, s_{\ell}^{N-1}\circ d_{N}^N\right)\\&=s_{\ell-1}^{N-1}\circ d^N_i
    \end{align*}
    \item  for $i=\ell+1$ : \begin{align*} 
        d^{N+1}_{\ell+1}\circ s_\ell^N&=\mathrm{pr}_{\ell+1}\left( s_{\ell-1}^{N-1}\circ d_0^N,\ldots, s_{\ell-1}^{N-1}\circ d_{\ell-1}^N,\mathrm{id}, s_\ell^{N-1}\circ d_{\ell+1}^N(x),\ldots, s_{\ell}^{N-1}\circ d_{N}^N\right)\\&=\mathrm{id}
    \end{align*}
    \item for $\ell+1< i \leq N$ : \begin{align*} 
        d^{N+1}_i\circ s_\ell^N&=\mathrm{pr}_{i}\left( s_{\ell-1}^{N-1}\circ d_0^N,\ldots, s_{\ell-1}^{N-1}\circ d_{\ell-1}^N,\mathrm{id}, s_\ell^{N-1}\circ d_{\ell+1}^N,\ldots, s_{\ell}^{N-1}\circ d_{N}^N\right)\\&=s_{\ell}^{N-1}\circ d^N_{i-1}
    \end{align*}
    \end{enumerate}
    \end{remark}
    \begin{proof}
     Item 1 follows from the fact that the maps $d_\ell^{N+1}\colon K_{N+1}\to K_N$ are submersions for all $\ell=0,\ldots,N+1$ and that,  upon choosing a local section, for every $x\in K_{N+1}$, the points $d_\ell^{N+1}(x)$ and $p_\ell^{N+1}(x)$ are related.  Consequently, the codimension of the corresponding bi-vertical leaves are identical. Let us show item 2. By construction,  $$\xymatrix{\left(\Lambda^{N+1}_\ell K, \mathcal{BV}_{\Lambda^{N+1}_\ell K}\right)&& {W}_{N+1}^{\ell}\ar[ll]_<<<<<<<<<{p_\ell^{N+1}}\ar[rr]^{d_\ell^{N+1}}&&(K_{N},\mathcal{BV}_{K_{N}}})$$ is a bi-submersion for every $0\leq\ell\leq N+1 $. For $0\leq \ell'\neq \ell\leq N+1$ we need to show that $$\xymatrix{\left(\Lambda^{N+1}_{\ell'} K, \mathcal{BV}_{\Lambda^{N+1}_{\ell'} K}\right)&& {W}_{N+1}^{\ell}\ar[ll]_<<<<<<<<<{p_{\ell'}^{N+1}}\ar[rr]^{d_{\ell'}^{N+1}}&& (K_{N},\mathcal{BV}_{K_{N}}})$$ is a bi-submersion. The equality  $(d_{\ell'}^{N+1})^{-1}\left(\mathcal{BV}_{K_N}\right)=(p_{\ell'}^{N+1})^{-1}\left(\mathcal{BV}_{\Lambda^{N+1}_{\ell'}K}\right)$ holds because the diagrams in Equation \eqref{diag:equivalence_{N+1}} commute. This implies that $$\Gamma(\ker Td_{\ell'}^{N+1})+\Gamma(\ker Tp_{\ell'}^{N+1})\subseteq (d_{\ell'}^{N+1})^{-1}\left(\mathcal{BV}_{K_N}\right)=(p_{\ell'}^{N+1})^{-1}\left(\mathcal{BV}_{\Lambda^{N+1}_{\ell'}K}\right).$$ We need to show the other inclusion. Let $X\in \mathcal{BV}_{K_N}=\bigcap_{j=0}^N\Gamma\left(\ker Td^{N}_{j}\right)\subset \mathfrak X(K_N)$. The vector field $(0,\ldots,X,\ldots,0)\in \mathfrak X(K_{N}^{\times N+1 })$ is tangent to the horn space $\Lambda^{N+1}_{\ell'}K$ and belongs to the bi-vertical singular foliation $\mathcal{BV}_{\Lambda^{N+1}_{\ell'}K}=\left(\bigcap_{j=1}^{\ell'}\Gamma\left(\ker T(d_{\ell'-1}^N\circ \mathrm{pr}_j)\right)\right)\cap\left(\bigcap_{j=\ell'+1}^{N+1}\Gamma\left(\ker T(d_{\ell'}^N\circ \mathrm{pr}_j)\right)\right)\subset \mathfrak X\left(\Lambda^{N+1}_{\ell'}K\right)$.  By Proposition \ref{prop:lifting-property}, there exists a vector field $\widecheck{X}\in \mathfrak X({W}_{N+1}^\ell)$ such that $Td_\ell^{N+1}(\widecheck X)=0$ and  \begin{equation*}
     Tp_\ell^{N+1}(\widecheck X)=\left(Td_0^{N+1}(\widecheck X),\cdots,\widehat{d_\ell^{N+1}},\cdots, Td_{N+1}^{N+1}(\widecheck X)\right)=(0,\ldots, \underbrace{X\circ d_{\ell'}^{N+1}}_{\text{the place of }\, d_{\ell'}^{N+1}\, \text{in}\, p_\ell^{N+1}},\ldots, 0).
 \end{equation*}
The vector field $\widecheck X\in \mathfrak X(K_{N+1})$ is $p_{\ell'}^{N+1}$-related to zero  and $d_{\ell'}^{N+1}$-related to $X$. Therefore, we have $Td_{\ell'}^{N+1}\left(\Gamma(\ker Tp_{\ell'}^{N+1})\right)=(d_{\ell'}^{N+1})^*\mathcal{BV}_{K_N}$. This proves the other inclusion. 

To complete the proof, we need to check that the maps $p_{\ell'}^2\colon {K}_{N+1}\to \Lambda^{N+1}_{\ell'}K$ are submersions for all $\ell'=0,\ldots, N+1$. This
follows from these two items \begin{enumerate}
     \item[$(i)$] For $0\leq \ell'\neq \ell\leq N+1$, the equality $$\Gamma(\ker Td_{\ell'}^{N+1})+\Gamma(\ker Tp_{\ell'}^{N+1})=(d_{\ell'}^{N+1})^{-1}\left(\mathcal{BV}_{K_N}\right)=(p_{\ell'}^{N+1})^{-1}\left(\mathcal{BV}_{\Lambda^{N+1}_{\ell'}K}\right)$$ implies that $\mathcal{F}_{K_{N+1}}=(p_{\ell'}^{N+1})^{-1}\left(\mathcal{BV}_{\Lambda^{N+1}_{\ell'}K}\right)=\Gamma(\ker Td_{\ell'}^{N+1})+\Gamma(\ker Tp_{\ell'}^{N+1})$ is a singular foliation on ${W}^\ell_{N+1}$. In particular, for every $x\in {W}^\ell_{N+1}$,  $T_xp_{\ell'}^{N+1}\left(\ker T_xd_{\ell'}^{N+1}\right)=T_{p_{\ell'}^{N+1}(x)}\mathcal{BV}_{\Lambda^{N+1}_{\ell'}K}$.

     \item[$(ii)$] { For every $x\in \Sigma$, there exists $ \Sigma\subset {W}^\ell_{N+1}$ a submanifold that contains $x$ such that $T_x{W}^\ell_{N+1}=T_x\Sigma\oplus T_x\mathcal{F}_{{W}^\ell_{N+1}}$. This implies in particular that the codimension of the leaf of 
     $\mathcal{BV}_{K_N}$ through $d_{\ell'}^{N+1}(x)$ is $\dim \Sigma$. The restriction of $p_{\ell'}^{N+1}\colon{W}_{N+1}^\ell\to K_N$ to $\Sigma\subset {W}_{N+1}^\ell$ is an immersion (we can assume it is injective). Its image $\Sigma'=p_{\ell'}^{N+1}(\Sigma)\subset K_N$ is a submanifold that intersects cleanly the bi-vertical leaf that passes though $p_{\ell'}^{N+1}(x)$ at a single point. By item $1$ of Lemma \ref{lem:KanN}, we have that  $T_xp_{\ell'}^{N+1}(T_x{W}^\ell_{N+1})=T_{p_{\ell'}^{N+1}(x)}\Sigma'\oplus T_{p_{\ell'}^{N+1}(x)}\mathcal{BV}_{\Lambda^{N+1}_{\ell'}K}=T_{p_{\ell'}^{N+1}(x)}\Lambda^{N+1}_{\ell'}K$.}
 \end{enumerate}
 Since the disjoint union of equivalences of bi-submersions is an equivalence of bi-submersions, $K_{N+1}$ is a bi-submersion between $K_N$ and $\Lambda^{N+1}_{\ell'} K$ for $\ell'=0,\ldots,N+1$.\\

 \noindent
 Item 3 is obtained as follows. To define the  maps $s_\ell^N\colon K_N\to K_{N+1}$ for $0\leq \ell\leq N$, we  use the explicit formula
in Equation \eqref{equiv-ell}. The latter also contains the graph of the restriction of the bi-submersion morphism  $\phi_\ell\colon K_{N}\to \Lambda^{N+1}_{\ell}K$ of the form of Equation \eqref{eq:phi_ell} 
to a bi-transversal $\Sigma\subset K_N$. This  allows us to define the $s_\ell^N$-map as
follows: consider $\iota_{T_{N+1}}\colon K_N\to T_{N+1}$ the unit of the path holonomy bi-submersion atlas $T_{N+1}\rightrightarrows ~K_N$. For $0\leq \ell\leq N$, we define $s_\ell^N$ to be the well-defined map  $$s_{\ell}^{N}\colon K_N\to {W}_{N+1}^{\ell}\subset K_{N+1},\; x\mapsto \left(\iota_{T_{N+1}}(x), (x,\phi(x))\times 0^{2n_{N-1}+n_N}\right)$$ with $\phi$ is as in Equation \eqref{eq:phi_ell}, so that $d^{N+1}_\ell(s_\ell^N(x))=x$ and $p_\ell^{N+1}(s_\ell^N(x))=\phi(x)$ for all $x\in K_N$. 
    \end{proof}


The following lemma is needed
\begin{lemma}\label{lem:hornN} \begin{enumerate}
           \item The spaces $\left(\Lambda^{N+2}_\ell K\right)_{0\leq \ell \leq N+2}$ are smooth and 
       $\Lambda^{N+2}_\ell K$ is a total relation of  bi-submersions for every $0\leq\ell\leq N+2$: 
\begin{align*}
     \ell=0\colon \xymatrix{\left(\Lambda^{N+1}_{0}K, \mathcal{BV}_{\Lambda^{N+1}_{0}K}\right)&&& \Lambda^{N+2}_0 K\ar[lll]_<<<<<<<<<<<<<<{\bigtimes^{i>1}d_{0}^{N+1}\circ \mathrm{pr}_{i}}\ar[rrr]^{d_{0}^{N+1}\circ \mathrm{pr}_{1}}&&& \left(K_{N},\mathcal{BV}_{K_{N}}\right)}\\ 0< \ell < N+1\colon \xymatrix{\left(\Lambda^{N+1}_{\ell}K, \mathcal{BV}_{\Lambda^{N+1}_{\ell}K}\right)&&& \Lambda^{N+2}_\ell K\ar[lll]_<<<<<<<<<<<<<<{d_{\ell-1}^{N+1}\circ \mathrm{pr}_{i\leq \ell}\times d_{\ell}^{N+1}\circ \mathrm{pr}_{i>\ell+1}}\ar[rrr]^{d_{\ell}^{N+1}\circ \mathrm{pr}_{\ell+1}}&&& \left(K_{N},\mathcal{BV}_{K_{N}}\right)}\\ \ell=N+1\colon \xymatrix{\left(\Lambda^{N+1}_{N+1}K,\mathcal{BV}_{\Lambda^{N+1}_{N+1}K}\right)&&&\Lambda^{N+2}_{N+1} K\ar[lll]_<<<<<<<<<<<<<<{\bigtimes^{i\leq N+1}d_{N}^{N+1}\circ \mathrm{pr}_{i}}\ar[rrr]^{d_{N+1}^{N+1}\circ \mathrm{pr}_{N+2}}&&& \left(K_{N},\mathcal{BV}_{K_{N}}\right)}\\ \ell=N+2\colon \xymatrix{\left(\Lambda^{N+1}_{N+1}K, \mathcal{BV}_{\Lambda^{N+1}_{N+1}K}\right)&&&\Lambda^{N+2}_{N+2} K\ar[lll]_<<<<<<<<<<<<<<{\bigtimes^{i\leq N+1}d_{N+1}^{N+1}\circ \mathrm{pr}_i}\ar[rrr]^{d_{N+1}^{N+1}\circ \mathrm{pr}_{N+2}}&&& \left(K_{N},\mathcal{BV}_{K_{N}}\right)}.
    \end{align*}

           \item the bi-submersions  $\left(\Lambda_\ell^{N+2}K\right)_{0\leq \ell \leq N+2}$ are \underline{equivalent} to $K_{N+1}$.
       \end{enumerate}
 \end{lemma}

 \begin{proof}
 \begin{enumerate}
\item We have the following decompositions 
   {\begin{equation}\label{diag:(N+2,l)-horn1}
          {\hbox{$\xymatrix{ &&\Lambda^{N+2}_\ell K\ar[dll]_{\mathrm{pr}_{\ell+1}}\ar[drr]^{\times_{i\neq\ell+1 }\mathrm{pr}_{i}}&&\\K_{N+1}\ar[d]_{d_{\ell}^{N+1}}\ar[drr]^{p_{\ell}^{N+1}}&& &&\Lambda_{\{\ell,\ell+1\}}^{N+2}K\ar[dll]_>>>>>>>>>{{{\bigtimes^{\ell}d_{\ell}^{N+1}  \bigtimes^{N+1-\ell}d_{\ell+1}^{N+1}}}}\ar[d]^>>>>>{\bigtimes^{\ell}d_{\ell-1}^{N+1}  \bigtimes^{N+1-\ell}d_{\ell}^{N+1}}\\ K_{N}& &\Lambda^{N+1}_{\ell}K && \Lambda^{N+1}_{\ell}K.}$}}
        \end{equation}}
 for $0\leq \ell\leq N+1$ and \begin{equation}\label{diag:(N+2,N+2)-horn1}
          {\hbox{$\xymatrix{ &&\Lambda^{N+2}_{N+2} K\ar[dll]_{\mathrm{pr}_{N+2}}\ar[drr]^{\times_{i\neq N }\mathrm{pr}_{i}}&&\\K_{N+1}\ar[d]_{d_{N+1}^{N+1}}\ar[drr]^{p_{N+1}^{N+1}}&& &&\Lambda_{\{N+1,N+2\}}^{N+2}K\ar[dll]_>>>>>>>>>{{{\bigtimes^{N+1}d_{N}^{N+1} }}}\ar[d]^>>>>>{\bigtimes^{N+1}d_{N+1}^{N+1}}\\ K_{N}& &\Lambda^{N+1}_{N+1}K && \Lambda^{N+1}_{N+1}K}$}}
        \end{equation} for $\ell=N+2$. In Diagram \eqref{diag:(N+2,l)-horn1}, we should erase $\times^0$. We need to check that all the maps of Diagram \eqref{diag:(N+2,l)-horn1} and Diagram \eqref{diag:(N+2,N+2)-horn1} are well-defined and are surjective submersions: notice that the projections of $\Lambda^{N+2}_\ell K$ to $K_{N+1}$ and to $\Lambda^{N+2}_{\{\ell, \ell+1\}}K$ are well-defined, since $p^{N+1}_\ell\colon K_{N+1}\to \Lambda^{N+1}_\ell K$ is a surjective submersion for $0\leq \ell\leq N+1$. Let us check that the maps $$\xymatrix{ \Lambda^{N+1}_{\ell}K  &&& \ar[lll]_>>>>>>>>>>>>>>>>{\bigtimes^{\ell}d_{\ell}^{N+1}  \bigtimes^{N+1-\ell}d_{\ell+1}^{N+1}}\Lambda_{\{\ell,\ell+1\}}^{N+2}K \ar[rrr]^>>>>>>>>>>>>>>>{\bigtimes^{\ell}d_{\ell-1}^{N+1}  \bigtimes^{N+1-\ell}d_{\ell}^{N+1}}  &&&  \Lambda^{N+1}_{\ell}K}$$ are surjective submersions. Surjectivity is obvious, since for $0\leq \ell\leq N+1$ we have the well-defined identities\begin{align}
     \left(\bigtimes^{\ell}d_{\ell}^{N+1}  \bigtimes^{N+1-\ell}d_{\ell+1}^{N+1} \right)\left(\underbrace{s_{\ell-1}^{N},\ldots, s_{\ell-1}^{N}}_{\ell\text{-times}},\underbrace{s_{\ell}^{N},\ldots, s_{\ell}^{N}}_{(N+1-\ell)\text{-times}}\right)&=(\mathrm{id},\ldots,\mathrm{id})\\\left(\bigtimes^{\ell}d_{\ell-1}^{N+1}  \bigtimes^{N+1-\ell}d_{\ell}^{N+1}\right)\left(\underbrace{s_{\ell-1}^{N},\ldots, s_{\ell-1}^{N}}_{\ell\text{-times}},\underbrace{s_{\ell}^{N},\ldots, s_{\ell}^{N}}_{(N+1-\ell)\text{-times}}\right)&=(\mathrm{id},\ldots,\mathrm{id}).
\end{align} Now we check the surjectivity at the level of tangent spaces at every point $x\in \Lambda_{\{\ell,\ell+1\}}^{N+2}K$. To simplify the notations, we perform the computations at the level of the space and apply the tangent functor at a given point throughout.  Let  $(a_0,\ldots,\widehat{a_\ell},\ldots, a_{N+1})\in \Lambda_{\ell}^{N+1}K$. We construct recursively an element $(\xi_0,\ldots, \widehat{\xi_{\ell}}, \widehat{\xi_{\ell+1}}, \ldots, \xi_{N+2})\in\Lambda_{\{\ell, \ell+1\}}^{N+2}K$ such that  $$\bigtimes^{\ell-1}_{i=0}d_{\ell-1}^{N+1}(\xi_{i})  \bigtimes^{N+2}_{i=\ell+2}d_{\ell}^{N+1}(\xi_i)=(a_0,\ldots,\widehat{a_\ell},\ldots, a_{N+1}).$$ We proceed as follows
\begin{align*}
\xi_{N+2}&\colon \text{pick}\; \xi_{N+2}\in K_{N}\;\text{such that}\; d_\ell^{N+1}\xi_{N+2}=a_{N+1}\\
\xi_{N+1}&\colon \text{pick}\; \xi_{N}\in K_{N}\;\text{such that}\; p_{\{0,\ldots, \widehat{\ell},\ldots, N\}}^{N+1}\xi_{N}=(d_\ell^{N+1}\xi_{N}, d_{N+1}^{N+1}\xi_{N+1})=\underbrace{(a_{N}, d_{N+1}^{N+1}\xi_{N+2})}_{\in \Lambda^{N+1}_{\{0,\ldots, \widehat{\ell},\ldots, N\}}K}\\
\xi_{N} &\colon \text{pick}\; \xi_{N}\in K_{N}\;\text{such that}\; p_{\{0,\ldots, \widehat{\ell},\ldots, N-1\}}^{N+1}\xi_{N}=(d_\ell^{N+1}\xi_{N},d_{N}^{N+1}\xi_{N}, d_{N+1}^{N+1}\xi_{N})\\&\hspace{7.2cm}=\underbrace{(a_{N-1},d_{N}^{N+1}\xi_{N}, d_{N}^{N+1}\xi_{N+2})}_{\in \Lambda^{N+1}_{\{0,\ldots, \widehat{\ell},\ldots, N-1\}}K}\\
\vdots & \\
\xi_{\ell+2} & \colon \text{pick}\; \xi_{\ell+2}\in K_{N}\;\text{such that}\; p_{\{0,\ldots, \ell-1,\ell+1\}}^{N+1}\xi_{\ell+1}=(d_\ell^{N+1}\xi_{\ell+2},d_{\ell+2}^{N+1}\xi_{\ell+2},\ldots, d_{N+1}^{N+1}\xi_{\ell+2})\\&\hspace{6.9cm}=\underbrace{(a_{\ell+1},d_{\ell+2}^{N+1}\xi_{\ell+3},\ldots, d_{\ell+2}^{N+1}\xi_{N+2})}_{\in \Lambda^{N+1}_{\{0,\ldots, \ell-1,\ell+1\}}K}\\
\xi_{\ell-1}& \colon \text{pick}\; \xi_{\ell-1}\in K_{N}\;\text{such that}\; p_{\{0,\ldots, \ell-2,\ell\}}^{N+1}\xi_{\ell-1}=(d_{\ell-1}^{N+1}\xi_{\ell-1},d_{\ell+1}^{N+1}\xi_{\ell-1},\ldots, d_{N+1}^{N+1}\xi_{\ell-1})\\&\hspace{6.5cm}=\underbrace{(a_{\ell-1},d_{\ell-1}^{N+1}\xi_{\ell+2},\ldots, d_{\ell-1}^{N+1}\xi_{N+2})}_{\in \Lambda^{N+1}_{\{0,\ldots, \ell-2,\ell\}}K}\\
\vdots & \\
\xi_{0} &\colon \text{pick}\; \xi_{0}\in K_{N}\;\text{such that}\; p_\ell^{N+1}\xi_{0}=(d_0^{N+1}\xi_0,\ldots, d_{\ell-1}^{N+1}\xi_{0},\widehat{d_{\ell}^{N+1}\xi_{0}},\ldots, d_{N+1}^{N+1}\xi_{0})\\&\hspace{4.6cm}=\underbrace{(d_0^{N+1}\xi_1, \ldots,a_{0}, \widehat{d_{0}^{N+1}\xi_{\ell+1}},\ldots, d_{0}^{N+1}\xi_{N+2})}_{\in \Lambda^{N+1}_\ell K}
\end{align*}

We have used the fact that all the generalised horn projections are surjective submersions. We show that the other maps $\bigtimes^{\ell}d_{\ell}^{N+1}  \bigtimes^{N+1-\ell}d_{\ell+1}^{N+1}\colon \Lambda^{N+2}_{\{\ell,\ell+1\}} K\stackrel{}{\rightarrow}\Lambda^{N+1}_\ell K$ are surjective submersions in a similar way.

Now, since for $0\leq\ell\leq N+1$ the maps
\begin{align}
    \label{Morita_{ell-1,ell}}
     \xymatrix{ &\ar[dr]^>>>>>>>{\bigtimes^{\ell}d_{\ell-1}^{N+1}  \bigtimes^{N+1-\ell}d_{\ell}^{N+1}} \Lambda_{\{\ell,\ell+1\}}^{N+2}K\ar[dl]_>>>>>>>{\bigtimes^{\ell}d_{\ell}^{N+1}  \bigtimes^{N+1-\ell}d_{\ell+1}^{N+1}}  & \\ \Lambda^{N+1}_{\ell}K& & \Lambda^{N+1}_{\ell}K} 
  \end{align}

makes the respective diagrams 

\begin{align}
    \label{}
   \scalebox{0.8}{\hbox{$ {\xymatrix{&& \Lambda_{\ell}^{N}K&&\\\Lambda^{N+1}_{\ell} K\ar[urr]^{d_{\ell-1}^{N}\circ \mathrm{pr}_{i\leq \ell}\times  d_{\ell}^{N}\circ \mathrm{pr}_{i> \ell+1} }\ar[drr]_{d_\ell^{N}\circ\mathrm{pr}_{\ell+1}}&&\ar@{->>}[ll]_<<<<<<<<<<{} \Lambda_{\{\ell,\ell+1\}}^{N+2}K\ar@{..>}[u]\ar@{..>}[d]\ar@{->>}[rr]^{} && \Lambda^{N+1}_{\ell} K\ar[llu]_{d_{\ell-1}^{N}\circ \mathrm{pr}_{i\leq \ell}\times d_{\ell}^{N}\circ \mathrm{pr}_{i>\ell+1} }\ar[dll]^{d_{\ell}^{N}\circ\mathrm{pr}_{\ell+1}}\\&& K_{N-1}&& }}  \qquad     \xymatrix{&& \Lambda_{N}^{N}K&&\\\Lambda^{N+1}_{N} K\ar[urr]^{d_{N}^{N}\circ \mathrm{pr}_{i\leq N}}\ar[drr]_{d_{N}^{N}\circ\mathrm{pr}_{N+1}}&&\ar@{->>}[ll]_<<<<<<<<<<{}\Lambda_{\{N+1,N+2\}}^{N+2}K\ar@{..>}[u]\ar@{..>}[d]\ar@{->>}[rr]^{} && \Lambda^{N+1}_{N} K\ar[llu]_{d_{N}^{N}\circ \mathrm{pr}_{i\leq N+1}}\ar[dll]^{d_{N}^{N}\circ\mathrm{pr}_{N+1}}\\&& K_{N-1}&&}$}}\end{align}

   commute, the right-hand side of the decomposition \eqref{diag:(N+2,l)-horn1} or $\eqref{diag:(N+2,N+2)-horn1}$ is a Morita equivalence of singular foliations between ${\mathcal{BV}}_{\Lambda^{N+1}_\ell K}$ and itself. By Proposition \ref{prop: inv-composition},  Diagram \eqref{diag:(N+2,l)-horn1} or \eqref{diag:(N+2,N+2)-horn1}  is a bi-submersion between the bi-vertical singular foliations ${\mathcal{BV}}_{K_{N+1}}$ and ${\mathcal{BV}}_{\Lambda^{N+1}_{\ell}K}$ for all  $0\leq \ell\leq N+1$ as  the fiber product of a bi-submersion and a Morita equivalence of singular foliations.\\

   \noindent
The bi-submersions $\left(\Lambda^{N+2}_\ell K\right)_{\ell=0,\ldots, N+2}$ are total relations between $\left(\Lambda^{N+1}_\ell K\right)_{0\leq  \ell \leq N+2}$ and $K_{N}$ for $\ell=0,\ldots, N+2$: for $\ell=0,\ldots, N+2$, the following diagrams commute
\begin{align}
   \scalebox{0.6}{ \xymatrix{&& (K_{N-1},\mathcal{F}_{K_{N-1}})&&\\ &&\ar[dll]^{}\Lambda^{N+2}_\ell K\ar@{->}[drr]_{} &&\\ ({\Lambda^{N+1}_\ell K}, \mathcal{F}_{\Lambda_\ell^{N+1}K})\ar[uurr]^{} \ar[ddrr]_{}&& &&  (K_{N}, \mathcal{F}_{K_{N}}) \ar[ddll]^{}\ar[uull]_{} \\&& \ar@{->}[ull]_{}K_{N+1}\ar@{->}[urr]^{}&&\\&&(\Lambda^{N}_\ell K,\mathcal{F}_{\Lambda^{N}_\ell K}) &&}} \quad \scalebox{0.6}{\xymatrix{&&(K_{N-1},\mathcal{F}_{K_{N-1}})&&\\ &&\ar[dll]^{}\Lambda^{N+2}_{N+2} K\ar@{->}[drr]_{} &&\\ ({\Lambda^{N+1}_{N+1} K}, \mathcal{F}_{\Lambda_{N+1}^{N+1}K})\ar[uurr]^{} \ar[ddrr]_{}&& && (K_{N}, \mathcal{F}_{K_{N}}) \ar[ddll]^{}\ar[uull]_{}\\&& \ar@{->}[ull]_{}K_{N+1}\ar@{->}[urr]^{}&&\\&&(\Lambda^{N}_{N} K,\mathcal{F}_{\Lambda^{N}_{N} K}) && }}\end{align}

    \item  We need to show that the bi-submersions  $\left(\Lambda_\ell^{N+2}K\right)_{0\leq \ell \leq N+2}$ are \underline{equivalent} to $K_{N+1}$. By Lemma \ref{lem:higher-mult}, every point of $\Lambda^{N+2}_\ell K$ is related to a point of $K_{N+1}$ for all $0\leq \ell \leq N+2$. On the other side, all the points of $K_{N+1}$ are related to points of $\Lambda^{N+2}_\ell K$ via the bi-submersion morphism \begin{equation}
     \label{eq:phi_elll}x\mapsto \begin{cases}
         \left(x, s_0^{N}\circ d_{1}^{N+1}(x),\ldots, s_{0}^{N}\circ d_{N+1}^{N+1}(x)\right) & \text{for}\; \ell=0\\\\ \left( s_{\ell-1}^{N}\circ d_0^{N+1}(x),\ldots, s_{\ell-1}^{N}\circ d_{\ell-1}^{N+1}(x),x, s_\ell^{N}\circ d_{\ell+1}^{N+1}(x),\ldots, s_{\ell}^{N}\circ d_{N+1}^{N+1}(x)\right) & \text{for}\; 0<\ell< N+1\\\\ \left( s_{N}^{N}\circ d_0^{N+1}(x),\ldots, s_{N}^{N}\circ d_{N}^{N+1}(x),x\right) & \text{for}\; \ell=N+1, N+2
     \end{cases}
 \end{equation}
\end{enumerate}
\end{proof}

We are now in a position to write the proof of Proposition \ref{prop:up-to-degreeN}.

\begin{proof}[Proof (of Proposition \ref{prop:up-to-degreeN})]
Let us check that the axioms of a para-Lie $\infty $-groupoid up to order $N+1$ in Definition \ref{def:para-Grp-upto} are satisfied. For $k=0,\dots,N$, they hold in view of the recursion assumption. 
Lemma \ref{lem:KanN} implies that the axioms of para-Lie $\infty$-groupoid up to order $N+1$ hold since the latter Lemma gives the para-simplical maps and the Kan conditions: it states that the projections on the horns are surjective submersions.

Let us now check that it is a holonomy para-Lie $\infty $-groupoids up to order $N+1$, i.e., that the four items in
Definition \ref{def:universal-Grpoid} are satisfied for $ k=1,\ldots, N$ and the three first ones hold for $k=N+1$.

These four items hold by the recursion assumption for $k=1,\ldots, N-1$ and the three first items hold for $k=N$. The fourth one holds by definition of $K_{N+1}$.
We need to check the three first ones for $ k=N+1$. Itme 1 is obtained as follows: since the pullback \begin{equation}\label{eq:geom-resol_{N}}
    \cdots \stackrel{}{\longrightarrow} \left(d_0^1\circ \cdots \circ d^{N}_0\right)^*E_{-N-2}\stackrel{}{\longrightarrow}\left(d_0^1\circ \cdots \circ d^{N}_0\right)^* E_{-N-1}\stackrel{}{\longrightarrow} TK_{N}\end{equation}is a geometric resolution of the bi-vertical foliation $\mathcal{BV}_{K_{N}}$ on $K_{N}$ and $\dim K_{N+1}=\mathrm{rk}(E_{-N-1})+ \dim \Lambda ^{N+1}_0 K$, there exists a vector bundle isomorphism $\nu^{N+1}\colon \left(d_0^1\circ \cdots \circ d_0^{N+1}\right)^*E_{-N-1}\to\ker Td_0^{N+1}$ such that $Td_0^{N+1}\circ \nu^{N+1}=\left(d_0^1\circ \cdots \circ d_0^{N+1}\right)^*\dd^{(N+1)}$, by Lemma \ref{lem:biss_anchored} and Lemma \ref{lem:anchoredbundleforbivertical}. The exactness of the complex \eqref{eq:geom-resol_{N}} at level of sections implies that  the intersection $\Gamma(\ker Td_0^{N+1}
    )\cap\Gamma(\ker Tp_0^{N+1}
    )$ is a singular foliation on $K_{N+1}$ and 
 the pullback vector bundle $(d_0^1\circ \cdots \circ d^{N+1}_0)^*E_{-N-2}\to K_{N+1}$ together with the composition $$\nu^{N+1}\circ \dd^{(N+2)} \colon \xymatrix{\left(d_0^1\circ \cdots \circ d_0^{N+1}\right)^*E_{-N-2}\ar[r]^{}&TK_{N+1}}$$ is an anchored bundle over 
 the bi-vertical singular foliation $\mathcal{BV}_{K_{N+1}}=\bigcap_{j=0}^{N+1}\Gamma\left(\ker Td^{N+1}_{j}\right)$, by Lemma \ref{lem:anchoredbundleforbivertical}.  In particular, $\mathcal{BV}_{\Lambda_\ell^{k+1}K}=\left(\bigcap_{j=1}^{\ell}\Gamma\left(\ker T(d_{\ell-1}^k\circ \mathrm{pr}_j)\right)\right)\cap\left(\bigcap_{j=\ell+1}^{k+1}\Gamma\left(\ker T(d_\ell^k\circ \mathrm{pr}_j)\right)\right)$ is a singular foliation on $\Lambda_\ell^{N+2}K$ for all $\ell=0,\ldots, N+2$.


The second item holds by Lemma \ref{lem:KanN} and Lemma \ref{lem:hornN} which imply that
 $\Lambda_\ell^{N+1} K \stackrel{\;\;\;\;\;p_\ell^{N+1}}{\longleftarrow} K_{N+1}\stackrel{d_\ell^{N+1}}{\longrightarrow} K_{N}$ and $\Lambda_\ell^{N+1} K \stackrel{}{\longleftarrow} \Lambda^{N+2}_\ell K\stackrel{}{\longrightarrow} K_{N}$ are bi-submersions for the respective bi-vertical foliations for every $\ell=0,\ldots,N+2$. 
 The third item holds by Lemma \ref{lem:hornN} since it states that they are equivalent.
 
 Condition 1 in Theorem \ref{int:theorem} holds by construction. This completes the proof.
\end{proof}


\subsection{Conclusion}

Theorem \ref{int:theorem} now follows from Proposition \ref{prop:up-to-degreeN} and, except for the two last items, which has now to be dealt with. 

\begin{enumerate}
    \item The formula involving the dimensions in condition 1 of Theorem \ref{int:theorem} is obtained recursively by using the fact that  $\dim K_k=\mathrm{rk}(E_{-k})+\dim \Lambda_0^kK$ for $k\geq 1$ and  formula \eqref{eq:horndimension}. This proves the first one.

\item In Proposition \ref{prop:up-to-degreeN}, Condition 1 is satisfied, i.e., $(d_0^1\circ \cdots \circ d^k_0)^*E_{-k}\simeq\ker Tp_0^k$ for all $k\geq 1$. Therefore, if the length of the geometric resolution $(E,\dd, \rho)$ of $\mathcal{F}$ in Equation \eqref{eq:geom-resol_{N+1}} is $n$, then for all $k\geq n+1$, $\ker Tp_0^k\equiv 0$. For dimension reasons, we have $\ker Tp_i^k\equiv 0$ for all $0\leq i\leq k$. Since for each $0\leq i\leq k$, the projection map $p_i^k\colon K_k\to \Lambda^k_i K$ is a surjective submersion,  it follows that $p_i^k\colon K_k\to \Lambda ^k_iK$ is a surjective local diffeomorphism for all $k> n$ and $0\leq i\leq k$.\\

{In addition, assume  that $ E_{-k}$ is trivial for all $ k \geq n+1$. Equivalently, this means that there are no bivertical vector fields on $K_{k}$ for $ k \geq n+1$. 
Consider, for $k\geq n+1$ and  $i=0,\dots, k+1$, inside the direct product manifold,
  $$\Lambda_i^{k+1}K\times K_k,$$
  the subset $W_i$ of pairs $(y,x)$ which are equivalent. We claim that this subset is a submanifold, locally given by the graph of  a smooth map from an open subset of $\Lambda_i^{k+1} K$, and that $K_{k+1} =\coprod_{i=0}^{k+1}W_i$ satisfies the recursion assumption, it is a total relation between $\Lambda_i^{k+1}K$ and $K_{k}$: indeed, let $(y,x)$ be an element in this set. There is a neighborhood $ \mathcal U$ of $y$ in $\Lambda_i^{k+1} K$ and a smooth morphism of bisubmersions $\phi$ mapping $\mathcal{U}$ to $K_k$ and $y$ to $x$. There is also a neighborhood $\mathcal{V}$ of $x$ in $K_k$ where two different points can not be equivalent by Theorem \ref{th:whereoneusesAnalyse}. We can assume $ \phi(\mathcal U)\subset \mathcal V$. Now, the restriction of $W_i$ to $\mathcal U \times \mathcal V$ is the graph of $\phi$. In particular, it is a sub-manifold. By construction, it is a total relation.}

\end{enumerate}

 This concludes the proof of Theorem \ref{int:theorem}.

\section{Examples}

\begin{enumerate}
    \item For a singular foliation $(M,\mathcal F)$ which is Debord, i.e., projective as a module over the algebra of functions so that there is a geometric resolution of length $1$, the holonomy groupoid is a Lie groupoid \cite{DEBORD2013613}. 
    This Lie groupoid, being constructed as a groupoid of germs, satisfies by construction the property that it is not possible that two different elements are equivalent. As a consequence, the nerve (see Example \ref{example:groupoidnerve}) of this groupoid is a Lie $\infty$-groupoid: it is a holonomy Lie $\infty$-groupoid for this singular foliation. 

\item The next case is the case of a singular foliation $\mathcal F$ that admits a resolution of length $2$ that we denote by 
$$\cdots 0\stackrel{}{\longrightarrow}E_{-2}\stackrel{}{\longrightarrow} E_{-1}\stackrel{}{\longrightarrow}TM.$$
Let us assume that $E_{-1}$ admits a Lie algebroid structure. Assume that this Lie algebroid integrates to a connected Lie groupoid $\Gamma$. Then one can choose $K_1 = \Gamma$ together with source and target map and its unit map, since it is a bisubmersion over $(M,\mathcal F)$. Also, under these circumstances, there is a Lie algebra bundle structure on $E_{-2}$ such that the map
$\Gamma(E_{-2})\stackrel{}{\longrightarrow} \Gamma (E_{-1})$ is a Lie algebra morphism, i.e., $E_{-2}|_m$ maps to the isotropy Lie algebra of the Lie algebroid $E_{-1}$ at $m$ for every $m\in M$. The Lie algebra bundle structure is unique because this Lie algebra morphism is injective on a dense open subset.

Moreover, the bivertical singular foliation on $\Gamma$ is generated by right-invariant vector fields coming from sections of $E_{-2}$. It is also generated by left-invariant vector fields coming from sections of $E_{-2}$. 
For $ \Gamma \times_M \Gamma$, bivertical vector fields are generated by those of the form $ (u+\overrightarrow{a},\overleftarrow{a}+v)$ where $ a \in \Gamma(E_{-1}) $, $u,v $ are bivertical vector fields on $ \Gamma$. 

Let $\mathcal B \to M $ be a Lie group bundle that integrates the Lie algebra bundle  $E_{-2}\to M$, such that each fiber is simply connected. There is a groupoid morphism $\phi : \mathcal B \to \Gamma $.
A first attempt is  to choose $K_2'=\Gamma \times_M \Gamma \times_M \mathcal B  $ with the three face maps $d_0^2,d_2^2$ and  $d_1^2$ being the projections on the first and second components, then the map 
 $$ (\gamma_1,\gamma_2,b) \mapsto \gamma_1 \gamma_2 \phi(b)$$ 
together with the degeneracy maps $s_0^1(\gamma)=(\gamma,1,1)$ and $s_1^1(\gamma)=(1,\gamma,1)$. The latter is indeed an equivalence between the bi-submersions $ \Gamma \times_M \Gamma $ and $\Gamma $, and a relation between their respective bi-vertical vector fields. Unfortunately, $K_2'$ is not a total relation in general. However, we can replace $\mathcal B$ by a larger group bundle, now non-connected, such that any element in the isotropy group of $\Gamma$ equivalent to the identity lies in the image of $\mathcal B$. We then define $K_2= \Gamma \times_M \Gamma \times_M \mathcal B$. It is a total relation.

Since there is no non-trivial bi-vertical vector fields on $ K_2$.
The rest of the construction is then automatic. 
To be more abstract: $\mathcal B \longrightarrow \Gamma $ is a crossed module of Lie groupoids. 
A simplicial manifold can be associated to it \cite{Mathieu-Ginot}.

\item Now, if there is no Lie algebroid structure on $ E_{-1}$, then the construction goes through as follows.
We make the simplifying assumption that $E_{-1}$ and $ E_{-2}$ are trivial bundles. Assume that the  singular foliation $(M, \mathcal{F})$ is generated by the vector fields $X=(X_1,\ldots, X_n)\in \mathfrak X(M)^n$.
We let \begin{equation}
    K_0:=M,\; K_1\subset M\times \mathbb R^n,\;\text{and}\; K_2\subset M\times (\mathbb R^n)^2\times \mathbb R^k
\end{equation}
$K_1\subset M\times \mathbb R^n$ is a neighborhood of $M$ such that  $s, t\colon K_1\subset M\times \mathbb R^n\to M$ (as in Example \ref{ex:holonomy-biss}) is a path holonomy bi-submersion over $\mathcal{F}$ associated to $X_1,\ldots, X_n$. Also, $K_2$ is an open neighborhood of $M$ in $M\times (\mathbb R^n)^2\times \mathbb R^k$ where the exponential maps in the formulas below \eqref{eq:structuremaps1} make sense. Last, $k$ is the number generators of the bi-vertical singular foliation on $K_1$ when it exists.
 
 The structure maps are
\begin{align}
\nonumber&d^1_1:=s,\; d^1_0:=t\colon K_1\rightarrow K_0,\, s_0^0(m):=(m,0);\\\label{eq:structuremaps1}&d^2_0(m,\alpha, \beta, \lambda):=(\exp(\beta X)(m),\alpha);\\\nonumber&d^2_1(m,\alpha, \beta, \lambda):=\exp(\lambda Z)\circ\exp(\alpha \overrightarrow{X})(m,\beta); \\\nonumber&d^2_2(m,\alpha, \beta, \lambda):= (m, \beta);\, s_0^1(m, \alpha):= (m,\alpha, 0,0),\, \text{and}\, s_1^1(m, \alpha):= (m,0, \alpha,0).
\end{align}

$$\xymatrix
   {& 2 \ar@{->}[dl]_{\exp(\lambda Z)\circ\exp(\alpha \overrightarrow{X})(m,\beta)} \ar@{<-}[dr]^{\left(\exp(\beta X)(m),\alpha\right)}\ar@{}[d]|-{\circlearrowleft}\\
     0 \ar[rr]_{(m,\beta)} & & 1}
$$
where $\overrightarrow{X}=(\overrightarrow{X_1},\ldots, \overrightarrow{X}_n, Z_1, \ldots, Z_k)$ are generators of $\Gamma(\ker Td_1^1)$ such that for all $i=1,\ldots,n$ $\overrightarrow{X}_i$ is $d_0^1$-related to $X_i$  and $Z=(Z_1, \ldots, Z_k)$ are generators of the bi-vertical singular foliation on $K_1$. Here $\exp(\alpha Y)$ is a shorthand for the composition of flows $\exp(\alpha_1 Y_1)\circ \cdots \circ \exp(\alpha_n Y_n)$ with $\alpha=(\alpha_1,\ldots,\alpha_n)\in \mathbb R^n$ and $Y=(Y_1,\ldots, Y_n)$ a family of vector fields.

It is not hard to check that the maps in \eqref{eq:structuremaps1} satisfy the simplicial identities. Moreover, they satisfy the $1$-Kan and $2$-Kan conditions in a neighborhood of $M$.
\end{enumerate}

\appendix

\section{Exact sequences of vector bundles over different bases}\label{app:exact-sequence-of-vb}

 Let $ (E_i \to M_i)_{i\geq 1}$ be a sequence of vector bundles, and $ d_i: E_i \to E_{i-1}$ a sequence of vector bundles morphisms over smooth maps $ \phi_i\colon M_i\to M_{i-1}$ (defined for $ i\geq 1$):
\begin{equation}\label{eq:complex}
  \xymatrix{\cdots\ar[r]&E_3\ar[r]^{d_3}\ar[d]&E_2
  \ar[r]^{d_2}\ar[d]&E_1\ar[r]^{d_1}\ar[d]&E_0 \ar[d]\\\cdots\ar[r]&M_3\ar[r]^{\phi_3}&M_2\ar[r]^{\phi_2}&M_1\ar[r]^{\phi_1}&M_0}
\end{equation}
Such a data is called a \emph{complex} if $ d_i \circ d_{i+1}=0$ for $i\geq 1$. Now, it is interesting to ask: what is the correct definition of exactness at the level of sections of such a complex? For the present article, the definition will be as follows. For all  $k\geq i $, define $\phi_{k,i}:M_k\to M_i$ to be the composition:
 $$  \phi_{k,i}: =\phi_{i+1} \circ \dots \circ \phi_k $$
 and $ \phi_{i,i}=id$.
 We say that a complex as in \eqref{eq:complex} is \emph{exact} at the level of sections if for all $k\geq i$, the pull-back complex:
 \begin{equation}\label{eq:shortcomplex}
  \xymatrix{\phi_{k,i}^* E_i\ar[r]^{\phi_{k,i}^*d_i}\ar[d]&\phi_{k,i-1}^*E_{i-1}
  \ar[r]^{\phi_{k,i-1}^* d_{i-1}}\ar[d]&\phi_{k,i-2}^*E_{i-2}\ar[d]\\ M_k\ar[r]^{\mathrm{id}}&M_k\ar[r]^{\mathrm{id}}&M_k}
\end{equation}
  is exact at the level of sections for all $i\geq i$, i.e., if the middle term in the short sequence
   \begin{equation} \label{eq:exactnesscontent} \xymatrix{\Gamma(\phi_{k,i}^*E_i) \ar[r]^{\phi_{k,i}^*d_i}&\Gamma(\phi_{k,i-1}^*E_{i-1})
  \ar[r]^{\phi_{k,i-1}^*d_{i-1}}&\Gamma(\phi_{k,i-2}^*E_{i-2}) }\end{equation}
  is exact. Here “$\Gamma$” stands for  the sheaf of sections.
  In this article, all $ \phi_i$ will be surjective submersions: in that case, a complex is exact at the level of sections if and only if \eqref{eq:exactnesscontent} holds for $ k=i$. Note that, in this case, this implies that the complex:
   $$  \xymatrix{\cdots\stackrel{}{\longrightarrow}\Gamma_{proj}(E_i)\ar[r]^{d_i}&\Gamma_{proj}(E_{i-1})
  \ar[r]^{d_{i-1}}&\Gamma_{proj}(E_{i-2})\stackrel{}{\longrightarrow}\cdots}$$
   is exact.  Here $\Gamma_{proj}(E_i) $ stands for the subspace of all sections $a$ of $ E_i$ over $ M_i$ which are \emph{projectable}, i.e., such that there exists a section $ b$ of $ E_{i-1}$ with $ d_i (a)=b$. (Notice that the converse is not true in general, even if all $ \phi_i$ are surjective submersions: projectable sections might form an exact complex, but the complex might not be exact at the level of sections.)

\section{Lie $\infty$-groupoids}\label{Appendix:higher-groupoids}
We refer the reader, e.g., to \cite{Duskin.John.W, duskin2006higher} or \cite{Henriques,CCZ, Getzler, Mehta, Siran-Severa, del2024cohomology} for more details on the notion of simplicial sets and $\infty$-groupoids.

\subsection{Simplicial manifolds}
\begin{definition}\label{def:sMfd} We denote by $\Delta$ the category whose objects are finite sets $[n]:=\{0,\ldots,n\}$ and whose morphisms are monotonously increasing maps. A \emph{simplicial manifold} $S_\bullet$ is a contravariant functor from $\Delta$ to $\mathrm{\textbf{Man}}$ the category of manifolds. 
\end{definition}

\begin{remark}
Any monotonously increasing map can be written as the composition of maps $[k]\to [k+1]$ omitting one element and maps $[k]\to[k-1]$ hitting one element twice. Hence, the information of a simplicial manifold $S_\bullet$ is completely encoded by:
\begin{itemize}
    \item $S_k:=S([k])$, $k\geq 0$
    \item $d^k_i=S(\delta^k_i):S_{k}\to S_{k-1}$ for $k\geq 1$ $i\in\{0,\ldots,k\}$, where $\delta^k_i:[k-1]\to [k]$ is the injective monotonously increasing map which whose image is $[k]\backslash \{i\}$.
    \item $s^k_i=S(\sigma^k_i)\colon S_{k}\to S_{k+1}$ for $k\geq 0$ $i\in\{0,\ldots,k\}$, where $\sigma^k_i:[k+1]\to [k]$ is the surjective monotonously increasing map with $\sigma^k_i(i)=\sigma^k_i(i+1)=i$.
\end{itemize}
The $d^k_i$'s are called \emph{face maps} and the $s^k_i$'s are called \emph{degeneracy maps}. Notice that a countable collection $(S_\bullet,d^\bullet_\bullet,s^\bullet_\bullet)$ of objects and maps of  $\mathrm{\textbf{Man}}$ form a simplicial manifold if and only if the maps  $d^k_i$ and $s^k_i$ with $0\leq i\leq k$  satisfy certain identities commonly referred to as \emph{simplicial identities}, that is 
\begin{align}\label{eq:faces-faces}
    d^{k-1}_i d^{k}_j &= d^{k-1}_{j-1} d^k_i\;\; \text{if}\;\; 0\leq i<j\leq k\\\nonumber\\\label{eq:degen-degen} s^{k}_i s^{k-1}_j &= s^{k}_{j+1} s^{k-1}_i \;\; \text{if}\;\; 0\leq i\leq j\leq k-1
\end{align}
and 
\begin{equation}\label{eq:faces-deneg}
 d^k_i s^{k-1}_j = s^{k-2}_{j-1} d^{k-1}_i \; \text{if}\; i<j, \quad
 d^k_j s^{k-1}_j=\mathrm{id}_{S_{k}}=d^k_{j+1}s^{k-1}_j,\quad
 d^k_i s^{k-1}_j = s^{k-2}_j d^{k-1}_{i-1} \; \text{if}\; i> j+1.
\end{equation}

In that case, we shall denote the collection $S_\bullet=(S_\bullet,d^\bullet_\bullet,s^\bullet_\bullet)$ by

$$S_\bullet\colon \xymatrix{&\cdots\; \ar@<5pt>[r]\ar@<-5pt>@{->}[r]\ar@<1pt>[r]\ar@<-2pt>@{.}[r]&\ar@/^{0.8pc}/[l]\ar@/^{1.1pc}/[l]\ar@{-}@/^{1.4pc}/[l] \ar@/^{1.7pc}/[l]S_3\ar@<5pt>[r]\ar@<-5pt>@{->}[r]\ar@<1pt>[r]\ar@<-2pt>[r]&\ar@/^{0.8pc}/[l]\ar@/^{1.2pc}/[l]\ar@/^{1.6pc}/[l]S_2\ar@<-2pt>[r]\ar@<6pt>[r]\ar@<2pt>[r]&\ar@/^{0.8pc}/[l]\ar@/^{1.2pc}/[l]S_1 \ar@<-2pt>[r]\ar@<2pt>[r]& \ar@/^{0.8pc}/[l]S_0}.$$

\vspace{0.3cm}
Notice that by the simplicial identities the face maps are surjective  and the degeneracies are embeddings. Also, for every $k\geq 1$,  $S_0$ is embedded in $S_k$ via the successive compositions of  degeneracy maps $M\stackrel{s_0^0}{\longrightarrow}S_1\stackrel{s_0^1}{\longrightarrow}S_2\stackrel{}{\longrightarrow}\cdots \stackrel{s_0^{k-1}}{\longrightarrow} S_k$.
\end{remark}


\begin{example}
\label{example:groupoidnerve}
Let $\displaystyle{\mathcal G\rightrightarrows M}$ be a Lie groupoid with source and target  $s, t$ and a unit map $\mathfrak i\colon M\hookrightarrow  \mathcal{G}$. Then the following construction yields a simplicial manifold, known as the simplicial nerve of $\mathcal G\rightrightarrows M$:
\begin{itemize}
    \item $S_0=M$, for $k>0$, $S_k=\underbrace{\mathcal G\times_{s,M,t} \mathcal G\times_{s,M,t}\mathcal G\times_{s,M,t}\cdots \times \mathcal G\times_{s,M,t}\mathcal G}_{k-\text{times}}$ is the space of $k$-composable arrows for $k>0$.
    \item $d^1_1=s$, $d^1_0=t$, $d^k_i(g_k,\ldots,g_{1})=\begin{cases} (g_k,\ldots,g_{2}) & \text{for}\; i=0\\ 
        (g_k,\ldots, g_{i+1}g_i,\ldots,g_{1}) & \text{for}\; i=1, \ldots,k-1\\ (g_{k-1},\ldots,g_{1}) & \text{for}\; i=k
    \end{cases}$
    \item $s^0_0=\mathfrak{i} \colon M\hookrightarrow \mathcal G$ and  for $0\leq i\leq k>1$ we have $$s^k_i(g_k,\ldots,g_{1})=\begin{cases}\left(g_k,\ldots, g_{1}, \mathfrak i\circ s(g_1)\right) & \text{for}\; i=0\\
        \left(g_k,\ldots,g_{i+1},\mathfrak{i}\circ t(g_{i+1}), g_i,\ldots, g_{1}\right) & \text{for}\; i=1, \ldots,k
    \end{cases}$$ 
\end{itemize}
\end{example}

The analogue of the above example still works for local Lie groupoids. One just has to restrict all $S_k$ to  subspaces such that the $d_i^k$ and $s_i^k$ are defined.

\subsubsection{(Generalized) horn spaces and horn projections}
Let us first introduce some notations and vocabularies.\\

It is convenient for us to use notations from \cite{Dorsch} to consider for any ordered subset $I\subset \{0, \ldots, k\}$ the so-called \emph{(generalized) $(k, I)$-horn space $\Lambda_I^kS$} of a simplicial manifold consists of a $k+1-|I|$ tuple $(x_0,\ldots,\widehat{x_I},\ldots, x_k)$, where $x_i \in  S_{k-1}$, such that $d^{k-1}_ix_i=d^{k-1}_{j-1}x_i$ for $i<j$. The hat $\widehat{x_I}$ means the elements indexed on $I$ are missing. We shall denote the projection map of $S_k$ to $\Lambda_I^kS$ by \begin{equation}\label{eq:Horn-projections}
    p^k_I\colon S_k\to \Lambda_I^kS,\; x\mapsto ( d^k_0x,\ldots,\widehat{ d^k_Ix},\ldots,  d^k_kx).
\end{equation} As an example for $\ell<  \ell'\in \{0, \ldots, k\}$
\begin{align}
\Lambda_{\ell}^kS=\left\{(x_0, \ldots, \widehat{x}_\ell,\ldots, x_k)\in S^{\times k}_{k-1}\,  \middle| \; d^{k-1}_ix_i=d^{k-1}_{j-1}x_i\;\;\text{for}\;\; i, j\in [k]\setminus \{\ell\}\;\; \text{with}\;\; i<j \right\}\\\nonumber
\\
    \Lambda_{\{\ell, \ell'\}}^kS=\left\{(x_0, \ldots, \widehat{x}_\ell, \ldots, \widehat{x}_{\ell'
},\ldots, x_k)\in S^{\times k-1}_{k-1}\,  \middle| \; d^{k-1}_ix_i=d^{k-1}_{j-1}x_i\;\;\text{for}\;\; i, j\in [k]\setminus \{\ell,\ell'\}\;\; \text{with}\;\; i<j\right\}.
\end{align}

\begin{definition}[\cite{Henriques,Getzler,CCZ,Siran-Severa}] Let $n\in \mathbb N\cup \{\infty\}$. A \emph{Lie $n$-groupoid} is a simplicial manifold $K_\bullet$ such that the horn projections $p^k_\ell\colon K_k\to \Lambda^k_\ell K$ of Equation \eqref{eq:Horn-projections} are (surjective) submersions for every $0\leq \ell \leq k\leq n$ and are diffeomorphism for $k>n $ and $0\leq \ell \leq k$.  For $n=\infty$, we speak of \emph{Lie $\infty$-groupoid}.

In all cases, we shall say that $K_\bullet$ satisfies the \emph{Kan condition}.

\end{definition}
\begin{remark}\label{prop:horns}
    Let $K_\bullet$ be  Lie $\infty$-groupoid over a manifold $K_0=M$.

    \begin{enumerate}
        \item The horn spaces $\left(\Lambda_\ell^kK\right)_{0\leq \ell\leq k}$ or $\Lambda^k_I K$ for some ordered subset $I\subset [k]$ 
        are smooth manifolds and the generalized projection $p_I^k\colon K_k\to \Lambda^k_IK$ is a submersion. For instance, 
         we have the fiber product decompositions 
    
    \begin{align}\label{eq:horn-decomposition}
        \Lambda^k_{\ell}K=\begin{cases}
            K_{k-1}\bigtimes_{p^{k-1}_\ell,\, \Lambda^{k-1}_{\ell}K,\, \bigtimes^\ell d^{k-1}_\ell\bigtimes^{k-\ell-1}d^{k-1}_{\ell+1} } \Lambda^k_{\{\ell, \ell+1\}}K &\text{for}\;\; \ell\neq k\\ &\\K_{k-1}\bigtimes_{p^{k-1}_{k-1},\, \Lambda^{k-1}_{k-1}K,\, \bigtimes^k d^{k-1}_{k-2} } \Lambda^k_{\{k-1, k\}}K &\text{for}\;\; \ell= k
        \end{cases}
    \end{align}
    
We also have similar decompositions for the $\Lambda^k_{\{-, -\}}K$'s.  Hence, the smoothness of the aforementioned horn spaces are obtained recursively, cf., \cite[Lemma 2.4]{Henriques} and \cite[Remark 3.4]{Dorsch} for more details.
\item The dimension $\dim \left(\Lambda^k_\ell K\right)$ of the horn space $\Lambda^k_\ell K$ is constant for $0\leq \ell\leq k$. Moreover, for every $0\leq\ell\leq k\geq 1$, \begin{equation}\label{eq:horndimension}
       \dim (\Lambda^k_\ell K)=\sum_{i=0}^{k-1}(-1)^i\begin{pmatrix}
           k\\i+1
       \end{pmatrix}\dim K_{k-1-i}.
   \end{equation}
    {This can be seen by recursion  using the fiber product decomposition of $\Lambda^k_{\ell}K$ in Equation \eqref{eq:horn-decomposition}}, see \cite[Proposition 2.10]{CUECA2023108829}.
       \end{enumerate}
\end{remark}

\subsection{The tangent complex of a Lie $\infty$-groupoid}

There are several chain complexes of vector bundles that can be associated to a simplicial manifold or a Lie $\infty$-groupoid. Let us recall their definitions and the notion of tangent complex of a Lie $\infty$-groupoid. For more details on these notions see, e.g., \cite{goerss2009simplicial,Weibel} or \cite{Mehta, CUECA2023108829}.\\

\noindent

\textbf{Moore complex}. For $k\geq 1$, let $M\hookrightarrow K_k$ be the canonical embedding of $M$ in $K_k$ through the successive compositions degeneracy maps $s_0^{k-1}\circ \cdots \circ s_0^0 \colon M\to K_k$. For every Lie $\infty$-groupoid (also for an arbitrary simplicial manifold)

\begin{equation}
   K_\bullet\colon \xymatrix{&\cdots\; \ar@<5pt>[r]\ar@<-5pt>@{->}[r]\ar@<1pt>[r]\ar@<-2pt>@{.}[r]& \ar@/^{0.8pc}/[l]\ar@/^{1.1pc}/[l]\ar@{-}@/^{1.4pc}/[l] \ar@/^{1.7pc}/[l] K_3\ar@<5pt>[r]\ar@<-5pt>@{->}[r]\ar@<1pt>[r]\ar@<-2pt>[r]&\ar@/^{0.8pc}/[l]\ar@/^{1.1pc}/[l]\ar@/^{1.4pc}/[l] K_2\ar@<-2pt>[r]\ar@<6pt>[r]\ar@<2pt>[r]&\ar@/^{0.8pc}/[l]\ar@/^{1.2pc}/[l]K_1 \ar@<-2pt>[r]\ar@<2pt>[r] & \ar@/^{0.8pc}/[l]K_0=M}
   \end{equation}
\vspace{0.4cm}

we have a simplicial vector bundle over $M$ 
\begin{equation}
   TK_\bullet|_M\colon \xymatrix{&\cdots\; \ar@<5pt>[r]\ar@<-5pt>@{->}[r]\ar@<1pt>[r]\ar@<-2pt>@{.}[r]& \ar@/^{0.8pc}/[l]\ar@/^{1.1pc}/[l]\ar@{-}@/^{1.4pc}/[l] \ar@/^{1.7pc}/[l] TK_3|_M\ar@<5pt>[r]\ar@<-5pt>@{->}[r]\ar@<1pt>[r]\ar@<-2pt>[r]&\ar@/^{0.8pc}/[l]\ar@/^{1.1pc}/[l]\ar@/^{1.4pc}/[l]TK_2|_M\ar@<-2pt>[r]\ar@<6pt>[r]\ar@<2pt>[r]&\ar@/^{0.8pc}/[l]\ar@/^{1.2pc}/[l]TK_1|_M \ar@<-2pt>[r]\ar@<2pt>[r] & \ar@/^{0.8pc}/[l]TM}
   \end{equation}
\vspace{0.4cm}

here, for $k\geq 1$, $TK_k|_M$ stands for the restriction of the tangent bundle $TK_k\to K_k$ to the image of $M$ through the canonical embedding of $M$ in $K_k$. The face and degeneracy maps are the $Td_i^k$'s  and the $Ts_i^k$'s.  The simplicial vector bundle $TK_\bullet|_M$ induces a $\mathbb N$-graded complex of vector bundles 
    \begin{equation}\label{simplicial-ch}
        \xymatrix{\cdots\ar[r]^>>>>>{\partial_3}&TK_2|_M\ar[r]^{\partial_2}&TK_1|_M\ar[r]^{\partial_1}&TM}
    \end{equation}
    whose differential map is given by alternating the face maps \begin{equation}
        \partial_k=\sum_{i=0}^k(-1)^iTd^k_i
    \end{equation} for all $k\geq 1$. The identity  $\partial^2=0$ follows from the simplicial identities  \eqref{eq:faces-faces}. The latter complex is known as the \emph{Moore complex} associated to $K_\bullet$. This complex of vector bundles can be constructed for arbitrary simplicial manifolds, cf., e.g., \cite {goerss2009simplicial,Mehta}. 
\\

\noindent
\textbf{The normalized tangent space}. It is easily checked by using the simplicial identities \eqref{eq:faces-deneg} that the differential map $\partial$ of the Moore complex associated to $K_\bullet$ goes to quotient to induce a 
 morphism of vector bundle
\begin{equation}\label{eq:normlized_TS}
   \partial_k\colon  \frac{TK_{k+1}|_M}{\bigplus_{i=0}^{k}\mathrm{im}(Ts_i^{k})}\longrightarrow \frac{TK_{k}|_M}{\bigplus_{i=0}^{k-1}\mathrm{im}(Ts_i^{k-1})},\quad k\geq 1.
\end{equation}
The latter is a complex of vector bundles over $M$ and is  called the \emph{normalized tangent space} of $K_\bullet$, see  \cite[chap. III.2]{goerss2009simplicial} and \cite{Mehta,CUECA2023108829}. There are a posteriori two other natural chain complexes of vector bundles that can only be defined for a Kan simplicial manifold, i.e., Lie $\infty$-groupoid. One of them is as follows:  define 

\begin{equation}\label{eq:Normalised}
     N_\bullet K_\bullet\colon \xymatrix{\cdots\ar[r]&N_kK_\bullet \ar[r]^{Td^k_0}&N_{k-1}K_\bullet\ar[r]&\cdots \ar[r]&N_2 K_\bullet\ar[r]^{Td^2_0}&N_1 K_\bullet\ar[r]^{Td^1_0}&TM}
\end{equation}
 with
 $$N_k K_\bullet:=\bigcap_{i=1}^{k}\ker(Td^k_i|_M)=\ker( Tp^k_0|_M)$$

This is indeed a chain complex thanks to the simplicial identities \eqref{eq:faces-faces}. The $N_kK_\bullet's$ are vector bundles of over $M$ since the horn projections are submersions. Similarly, one may define another complex of vector bundles by taking the intersection of the  kernels of all face maps except the last $k$-th face map and take the remaining face map, $Td^k_k$, as the differential map.

The following theorem unifies these three complexes of vector bundles.  
\begin{theorem}[Dold-Kan correspondence\cite{goerss2009simplicial,Weibel}]
    The chain complexes constructed in \eqref{eq:normlized_TS} and \eqref{eq:Normalised} are isomorphic.
\end{theorem}


\begin{definition}\label{def:tangent-complex}
    The \emph{tangent complex} of a Lie $\infty$-groupoid  $K_\bullet$ is the complex of vector bundles  \eqref{eq:normlized_TS}.
\end{definition}
\begin{remark}
    Definition \ref{def:tangent-complex} is valid for arbitrary simplicial manifolds. However, by Dold-Kan theorem, the tangent complex of a Lie $\infty$-groupoid $K_\bullet$ can be alternatively defined as the complex  of vector bundles
    \begin{equation}\label{eq:complex-kernels2}
        \xymatrix{\cdots\ar[r]^>>>>>{Td_0^4}&\ker Tp_0^3|_M\ar[r]^{Td_0^3}&\ker Tp_0^2|_M\ar[r]^{Td_0^2}&\ker Td_1^1|_M\ar[r]^{Td_0^1}&TM}.
    \end{equation}
If $K_\bullet$ is just a simplicial manifold, then by \cite[Theorem 3.1]{Dorsch}, for all $n\geq 1$  the projection maps $(p_i^n)_{0\leq i\leq n}$ are submersions in a neighborhood of $M$. Therefore, Equation \eqref{eq:complex-kernels2} makes sense for simplicial manifolds.

In  particular, the restriction of  the tangent complex of $K_\bullet$  \eqref{eq:complex-kernels2} to a point $m\in M$ yields a complex of vector spaces of the form  \begin{equation}\label{}
        \xymatrix{\cdots\ar[r]^>>>>>{Td_0^4|_m}&\ker T_{s_0^0(s_0^1(s_0^2(m)}p_0^3\ar[r]^>>>>>>{Td_0^3|_m}&\ker T_{s_0^0(s_0^1(m)}p_0^2\ar[r]^>>>>>{Td_0^2|_m}&\ker T_{s_0^0(m)}d_1^1\ar[r]^>>>>>>{Td_0^1|_m}&T_mM}.
    \end{equation}
\end{remark}

\subsection{Para-simplicial manifolds}
We now introduce a relaxed version of simplicial manifolds. Throughout this paper, the simplicial identities \eqref{eq:degen-degen}
 are not essential to present and prove our results. All the definitions and properties on simplicial manifolds presented in this paper, including the tangent complex, Kan condition, remain valid without any modifications for the notion  para-simplicial manifolds we now introduce.
\begin{definition}\label{def:partial-SM}
    A \emph{para-simplicial manifold} is a sequence $S_\bullet=(S_k)_{k \in\mathbb N}$ of manifold together with smooth functions $\left(d^k_i\colon S_k\longrightarrow S_{k-1}\right)_{i=0}^k$ called \emph{face maps} and $\left(s^k_i\colon S_k\longrightarrow S_{k+1}\right)_{i=0}^k$ called \emph{degeneracy maps} satisfying the simplicial identities \eqref{eq:faces-faces}  and  \eqref{eq:faces-deneg}. 
\end{definition}

\begin{remark}
    \begin{enumerate}
        \item A para-simplicial manifold that satisfies \eqref{eq:degen-degen} is a simplicial manifold.  Para-simplicial manifolds are stronger than semi-simplicial manifolds\cite{MKAN1970170}, as they possess degeneracy maps that satisfy the simplicial relations with the face maps.
        \item {Let $S_\bullet$ be a para-simplicial manifold. For $k\geq 2$, consider the equivalence relation on $S_k$ defined by identifying the points $s^{k}_i s^{k-1}_j(x)\sim s^{k}_{j+1} s^{k-1}_i(x)$ for $x\in S_{k-1}$ and   $0\leq i\leq j\leq k-1$. The face and degeneracy maps goes to quotient and $S_\bullet/_{\sim}$ is a simplicial space, i.e., the axioms \eqref{eq:faces-faces}, \eqref{eq:degen-degen} and \eqref{eq:faces-deneg} are fulfilled.}
    \end{enumerate}    
\end{remark}

We need the following technical definition,
\begin{definition}
 A \emph{para-simplicial manifold up to order $N\in \mathbb N$}  
 $$K_{\leq N}:\xymatrix{&K_N\ar@<5pt>[r]\ar@<-5pt>@{->}[r]\ar@<1pt>[r]\ar@<-2pt>@{.}[r]&\ar@/^{0.8pc}/[l]\ar@/^{1.2pc}/[l]\ar@{.}@/^{1.4pc}/[l] \ar@/^{1.7pc}/[l]\cdots\; \ar@<5pt>[r]\ar@<-5pt>@{->}[r]\ar@<1pt>[r]\ar@<-2pt>@{.}[r]&\ar@/^{0.8pc}/[l]\ar@/^{1.1pc}/[l]\ar@{-}@/^{1.4pc}/[l] \ar@/^{1.7pc}/[l]K_3\ar@<5pt>[r]\ar@<-5pt>@{->}[r]\ar@<1pt>[r]\ar@<-2pt>[r]&\ar@/^{0.8pc}/[l]\ar@/^{1.2pc}/[l]\ar@/^{1.6pc}/[l]K_2\ar@<-2pt>[r]\ar@<6pt>[r]\ar@<2pt>[r]&\ar@/^{0.8pc}/[l]\ar@/^{1.2pc}/[l]K_1 \ar@<-2pt>[r]\ar@<2pt>[r]& \ar@/^{0.8pc}/[l]K_0}$$
    
\vspace{0.2cm}
is  a para-simplicial manifold that is considered to be defined up to and including the $N$-simplices $K_N$.
 \end{definition}

\section{Bi-submersions between two singular foliations}
\label{sec:bi-submersion}

\subsection{Definition}

\vspace{0.2cm}

We extend the definition of bi-submersions over \emph{one} singular foliation given in \cite{AS} to \emph{two} singular foliations.

\begin{definition}\label{def:bi-submersion}
    Let $(M, \mathcal{F}_M)$ and $(N,\mathcal{F}_N)$ be foliated manifolds. A  \emph{bi-submersion  between $(M, \mathcal{F}_M)$ and $(N,\mathcal{F}_N)$} is a manifold\footnote{$W$ may be not connected, and we do \textbf{not} assume the connected components to be all of the same dimension.} $W$ equipped with a pair of submersions\footnote{We do not assume $p$ and $q$ to be surjective.} $p\colon W\longrightarrow M$ and  $q\colon W\longrightarrow N$
   fulfilling the following properties\footnote{Item 2 of Definition \ref{def:bi-submersion} simply means that both singular pull-back foliations $ p^{-1}( \mathcal F_M)$ and $q^{-1} (\mathcal F_N) $ are equal to the space of vector fields of the form $\xi+\zeta $ with $\xi\in \Gamma(\ker (T p)) $ and $\zeta \in \Gamma(\ker (T q))$.}:
    
    \begin{enumerate}
        \item $p^{-1}(\mathcal{F}_M)=q^{-1}(\mathcal{F}_N)$,
        \item $p^{-1}(\mathcal{F}_M)=\Gamma(\ker Tp)+ \Gamma(\ker Tq)$.
    \end{enumerate}
    Bi-submersions shall be denoted either by  $$(M,\mathcal{F}_{M}) \stackrel{p}{\leftarrow} W \stackrel{q}{\rightarrow} (N,\mathcal{F}_N),$$ or, when there is no risk of confusion, simply by $W$: the two submersions will be implicitly denoted by $p$ and $q$, or variations thereof (e.g., $p'$,$q'$). 
    Also, we denote by $\mathcal F_W$ the singular foliation  on $W$ given by $$\mathcal F_W:= p^{-1}(\mathcal{F}_M)=q^{-1}(\mathcal{F}_N)=\Gamma(\ker Tp)+ \Gamma(\ker Tq).$$
    \end{definition}

    Let us give more vocabulary.
    \begin{itemize}
        \item 
    A bi-submersion between $(M, \mathcal{F})$ and $(M,\mathcal{F})$ is simply called \emph{bi-submersion over $(M,\mathcal{F})$}. These objects were introduced by Androulidakis and Skandalis \cite{AS}.
\item Vector fields in $\mathcal {BV}_W:=\Gamma(\ker Tp)\cap\Gamma(\ker Tq)\subset \mathfrak X(W)$ are called \emph{bi-vertical vector fields}. They form an involutive submodule, i.e.,  stable under Lie bracket.
    When the  submodule $\mathcal{BV}_W$ is locally finitely generated, it defines a singular foliation on $W$, which we refer to as the \emph{bi-vertical singular foliation of $(M,\mathcal{F}_{M}) \stackrel{p}{\leftarrow} W \stackrel{q}{\rightarrow} (N,\mathcal{F}_N)$}.
    \item  When only item 1 in Definition \ref{def:bi-submersion} holds, we will speak of a \emph{Morita equivalence of singular foliations}\footnote{The notion was introduced in Garmendia and Zambon \cite{Garmadia-Zambon} with a difference: they assume that $p,q$ are surjective maps with connected fibers, while we do not.}
    \end{itemize}

\begin{remark} 
       Since any bi-submersion as in Definition \ref{def:bi-submersion} is a Morita equivalence of singular foliations as introduced by Garmendia and Zambon \cite{Garmadia-Zambon} upon restricted to appropriate neighborhoods of $w \in W, p(w)\in M, q(w) \in W$, one can use several results of \cite{Garmadia-Zambon}.
       To be more precise, for every $w\in W$, if  $\Sigma_{p(w)}\subset M$ is a transversal for $\mathcal{F}_M$ at $p(w)$; $\Sigma_{q(w)}\subset N$ a transversal for $\mathcal{F}_N$ at $q(w)$, then $\left(\Sigma_{p(w)}, \iota^{-1}_{\Sigma_{p(w)}}\mathcal{F}_M\right)$ and $\left(\Sigma_{q(w)}, \iota^{-1}_{\Sigma_{q(w)}}\mathcal{F}_N\right)$ are isomorphic singular foliations. We will say more about this point. In particular, bi-submersions preserve the codimension of leaves\cite[Proposition 2.5]{Garmadia-Zambon}.       
\end{remark}

For morphisms of bi-submersions, we also follow the terminology of Androulidakis-Skandalis.

 \begin{definition}
    A \emph{morphism} from a bi-submersion $(M,\mathcal{F}_{M}) \stackrel{p}{\leftarrow} W \stackrel{q}{\rightarrow} (N,\mathcal{F}_N)$ to a bi-submersion $(\mathcal{F}_{M}, M) \stackrel{p'}{\leftarrow} W' \stackrel{q'}{\rightarrow} (N,\mathcal{F}_N)$ is a smooth map $\varphi\colon W\longrightarrow W'$ that makes the following diagram commutes
    \begin{equation}
        \xymatrix{&W\ar[dd]^\varphi\ar[rd]^q\ar[ld]_p&\\M&&N\\&W'\ar[ur]_{q'}\ar[ul]^{p'}&}
    \end{equation}
    When $\varphi$ is only defined on an open subset of $W$, we speak of a \emph{local morphism}.
\end{definition}

We will see that all morphisms are “locally invertible”, and the notion of morphism will be replaced by the notion of “relation”, see \S \ref{sec:Equivalentbi-submersions}.

{Here is an important lifting property of bi-submersions that generalizes \cite[Proposition 5.15]{RubenSymetries}.}

\begin{proposition}\label{prop:lifting-property}
    {Let $(M,\mathcal{F}_{M}) \stackrel{p}{\leftarrow} W \stackrel{q}{\rightarrow} (N,\mathcal{F}_N)$ be a bi-submersion. For every pair of vector fields $(X,Y)\in \mathcal{F}_M\times \mathcal{F}_N$ there exists  a vector field $Z\in p^{-1}(\mathcal{F}_M)=q^{-1}(\mathcal F_N)$ such that $Z$ is $p$-related to $X$ and $q$-related to $Y$. }
\end{proposition}
\begin{proof}
    Let $(X,Y)\in \mathcal{F}_M\times \mathcal{F}_N$. Since $p^{-1}(\mathcal{F}_M)=q^{-1}(\mathcal F_N)=\Gamma(\ker Tp)+ \Gamma(\ker Tq)$, there exists a vector field $\underline{X}=X^p+X^q\in p^{-1}(\mathcal{F}_M)$ that is $p$-related to $X$ and $X^p\in \Gamma(\ker Tp)$ and  $X^q\in \Gamma(\ker Tq)$. In particular, $X^q$ is $p$-related to $X$. Likewise, there exists a vector field $\underline{Y}=Y^p+Y^q\in q^{-1}(\mathcal{F}_N)$ that is $q$-related to $Y$ and $Y^p\in \Gamma(\ker Tp)$ and  $Y^q\in \Gamma(\ker Tq)$. In particular, $Y^p$ is $q$-related to $Y$. The vector field $Z=X^q+Y^p$ is $p$-related to $X$ and $q$-related to $Y$.
\end{proof}

\subsection{Examples of bi-submersions}

Let us first list several examples of bi-submersions.

\begin{example}\label{ex:bi-submersions-projectable}
    Let $W,M,N$ be manifolds, and let $p\colon W\longrightarrow M$ and $q\colon W\longrightarrow N$ be submersions such that $\Gamma(\ker(T\pi_M))$ is generated by $p$-projectable vector fields and $\Gamma(\ker(T\pi_N))$ is generated by $q$-projectable vector fields. Define $\mathcal{F}_M$, respectively, $\mathcal{F}_N$ to be the images of these  projections: those are easily seen to be singular foliations. Then $$(M, \mathcal{F}_M) \stackrel{p}{\longleftarrow} W \stackrel{q}{\longrightarrow} (N, \mathcal{F}_N) $$ is easily seen to be a bi-submersion between $(M, \mathcal{F}_M)$ and  $(N, \mathcal{F}_N)$. It is clear by definition of $\mathcal{F}_M$ that $ \Gamma(\ker Tp)+ \Gamma(\ker Tq)\subseteq p^{-1}(\mathcal{F}_M)$. Now, for every $\xi\in (\pi_M)^{-1}(\mathcal{F}_M)$, there exists a vector field $\xi_1\in \Gamma(\ker Tq)$ such that $Tp(\xi-\xi_1)=0$, which proves the opposite inclusion.  The other equality is obtained similarly. This proves the result.
\end{example}
The following fundamental example shows that there exists always a bi-submersion over any given singular foliation.
\begin{example}\label{ex:holonomy-biss}\cite{AS}
Let $\mathcal{F}$ be a singular foliation on a manifold $M$. Let $m\in M$. Let $X_1,\ldots, X_n$ be generators of $\mathcal{F}$ near $m$. 
There is an open neighborhood $\mathcal{W}$ of $(m, 0)\in M \times \mathbb R^n$ such that $(M,\mathcal{F})\stackrel{p}{\leftarrow} \mathcal{W}\stackrel{q}{\rightarrow}(M,\mathcal{F})$ is a bi-submersion over $\mathcal{F}$ with \begin{equation}\label{eq:holbi-submersion}
    p(m,\lambda_1,\ldots,\lambda_n)=m\quad\text{and}\quad \displaystyle{q(m, \lambda_1,\ldots,\lambda_n)=\exp\left(\sum_{i=1}^n\lambda_iX_i \right)(m)}
\end{equation}
where for $X\in\mathfrak X(M)$,\;$\exp(X)_m:=\varphi^{X}_1(m)$ denotes the time-$1$ flow of $X$ at $m\in M$. 
Those bi-submersions are called \emph{path holonomy} in \cite{AS}. 
 When the generators $X_1,\ldots, X_n$  form a minimal family of generators of $ \mathcal F$ near $m$ (i.e., when their classes in $\mathcal{F}_m:=\mathcal{F}/\mathcal{I}_m\mathcal{F}$ form a basis), we will speak of a \emph{minimal path holonomy bi-submersion}.
 \end{example}

{
There is a more abstract description of the path holonomy bi-submersion over a singular foliation, which will be soon of a great importance.}

 \begin{example}\label{ex:holonomy-biss2}\cite[Proposition 4.35]{LLL2} Let $\mathcal{F}$ be a singular foliation on a manifold $M$ and $\mathcal{U}\subseteq M$ an open subset such that there exists an anchored bundle $A\stackrel{\rho}{\to} T\mathcal{U}=TM|_{\mathcal{U}}$ over $\mathcal{F}_\mathcal{U}$ the restriction of $ \mathcal{F}$ to $\mathcal{U}$, i.e., $A\stackrel{p}{\to} \mathcal{U}$ is a vector bundle such that $\rho(\Gamma(A))=\mathcal{F}_\mathcal{U}$.

Given a $TM$-connection $\nabla$ on $A$. Let  $\xi\in \mathfrak X(A)$ be the $\nabla$-geodesic vector field, i.e., the unique linear vector field on $A$ (that satisfies $Tp(\xi(a))=\rho(a)$ for $a\in A$) whose integral curves are the $\nabla$-geodesics.  There is an open neighborhood $\mathcal U_A$ of the zero section in $A$ on which $ M \stackrel{p}{\leftarrow} \mathcal U_A \stackrel{q}{\rightarrow} M $ is a bi-submersion over $\mathcal F $ with $q=p\circ \phi^\xi_1$. Here, $\phi^\xi_1$ is the time-$1$ flow of $\xi$.
     
 \end{example}

 \begin{example}
     Every Lie groupoid $\mathcal{G}\rightrightarrows M$ is a bi-submersion over the singular foliation defined by the image of the anchor map of its Lie algebroid $A\mathcal{G}$. It is even true for a local (i.e., defined on a neighborhood of identity) Lie groupoid.
 \end{example}

 Later, in Proposition \ref{prop:localstructure}, we will see that any bi-submersion is of the form given in Example \ref{ex:exampleofbi-submersion} below in a neighborhood of a point.

 \begin{example} \label{ex:exampleofbi-submersion}
Let $(M,\mathcal F_M)$ and $(N,\mathcal F_N)$ be foliated manifolds. Let $ m\in M$ and $ n \in N$ be two points such that the transverse singular foliations of the respective leaves through $m,n$ are isomorphic. This spells out by saying that  there exists  $ (\Sigma, \mathcal F_\Sigma)$ a foliated manifold, equipped with a point $O$ where all vector fields on $\mathcal F_\Sigma$ vanish, and two injective immersions $r:\Sigma \hookrightarrow M  $ and $l:\Sigma \hookrightarrow N  $ mapping $O$ to $m$ and $n$, respectively, transverse to $ \mathcal F_M$ and $ \mathcal F_N$, respectively, such that $r$, $l$ and the diffeomorphism \begin{equation}\label{eq:varphi}\varphi:= l^{-1} \circ r : l(\Sigma) \simeq r(\Sigma) \end{equation} are isomorphisms of foliated manifolds ($r(\Sigma)$ and $ l(\Sigma)$ being equipped with their respective induced singular foliations.  
Let $ d_M,d_N$ be the dimensions of the leaves through $m$ and $n$. By the local splitting theorems (see, e.g.,  \cite[\S 7.3]{LLL1}), there exists  isomorphisms of foliated manifolds
\begin{equation}\label{eq:} \Phi_M: (M^{m},\mathcal F_M) \simeq (\mathcal B^{d_M},\mathfrak X(\mathcal B^{d_M} ))  \times (\Sigma, \mathcal F_{\Sigma}) {\hbox{ and }}  \Phi_N: (N^{n},\mathcal F_N)  \simeq (\Sigma, \mathcal F_{\Sigma}) \times (\mathcal B^{d_N},\mathfrak X(\mathcal B^{d_N} ))  .
\end{equation}
Here $ M^m, N^n$ are neighborhoods of $m$ and $n$ in $M$ and $N$, respectively. For any bi-submersion 
    $$ (\Sigma, \mathcal{F}_\Sigma)\stackrel{p_\Sigma}{\leftarrow}W_\Sigma\stackrel{q_\Sigma}{\rightarrow}(\Sigma, \mathcal{F}_\Sigma) $$
    over $(\Sigma, \mathcal{F}_\Sigma)$,
the following pair of maps  
     $$ (M^m, \mathcal{F}_M)\stackrel{p}{\leftarrow}  \mathcal B^{d_M} \times W_\Sigma \times \mathcal B^{d_N}   \stackrel{q}{\rightarrow}(N^n, \mathcal F_N) $$ defines a bi-submersion between $\mathcal F_M$ and $\mathcal{F}_N$. Above $p(u,w,v)= \Phi_M(u,p_\Sigma(w))$ and $q(u,w,v)= \Phi_N(q_\Sigma(w),v)$.
\end{example}

\begin{definition}\label{def:typicaltype}
     Bi-submersions as in  Example \ref{ex:exampleofbi-submersion} will be from now on called bi-submersions of \emph{typical type}. The bi-submersion $ W_\Sigma$ over $ (\Sigma,\mathcal F_\Sigma)$ used in its construction will be called its \emph{model}, and the isomorphism of foliated manifolds $ \varphi$ in Equation \eqref{eq:varphi} is called its \emph{transverse diffeomorphism}.
\end{definition}

\subsection{New bi-submersions from old ones}

In this section, we present several constructions of bi-submersions derived from given bi-submersions.

The proof of the following proposition follows the same lines as in  \cite[Proposition 2.4]{AS}, we include a proof for the sake of completeness.
\begin{proposition}[Inverse and composition]\label{prop: inv-composition} Let $(M, \mathcal{F}_M)\stackrel{p}\leftarrow V\stackrel{q}{\rightarrow}(P, \mathcal{F}_P)$ be a Morita equivalence\footnote{Here we do not need to assume the fibers of $p$ and $q$ to be connected nor we assume $p$ and $q$ to be surjective.
} and $(P,\mathcal{F}_P)\stackrel{r}\leftarrow W\stackrel{\tau}{\rightarrow}(N, \mathcal{F}_N)$ a bi-submersion.
\begin{enumerate}
    \item $(N, \mathcal{F}_N)\stackrel{\tau}\leftarrow  W\stackrel{r}{\rightarrow}(P, \mathcal{F}_P)$ is a bi-submersion that is referred to as the \emph{inverse} of $(P, \mathcal{F}_P)\stackrel{r}\leftarrow W\stackrel{\tau}{\rightarrow}(N, \mathcal{F}_N)$.
It shall be denoted by $
    W^{-1}$.
    \item  $(M,\mathcal{F}_{M}) \stackrel{p\circ\mathrm{pr}_V}{\longleftarrow} V\times_{q,P,r}W\stackrel{\tau\circ\mathrm{pr}_W}{\longrightarrow} (N,\mathcal{F}_N)$ is a bi-submersion referred to as the \emph{composition} of $V$ and  $W$. It shall be denoted by $V*W$. Here,  $\mathrm{pr}_V\colon V\times W\to V$ is the projection onto $V$ and $\mathrm{pr}_W\colon V\times W\to W$ the projection onto $W$.
\end{enumerate}
Item 2 still holds when $(M, \mathcal{F}_M)\stackrel{r}\leftarrow W\stackrel{\tau}{\rightarrow}(P, \mathcal{F}_P)$ is a Morita equivalence and $(M, \mathcal{F}_M)\stackrel{p}\leftarrow V\stackrel{q}{\rightarrow}(P, \mathcal{F}_P)$ a bi-submersion. In particular, it holds when both are bi-submersions.
\end{proposition}

\begin{proof}
Item 1 is immediate. Let us prove item 2. The fibered product 
$$\xymatrix{&&\ar[dl]_{\mathrm{pr}_V}V\times_{q,P,r}W\ar[dr]^{\mathrm{pr}_W}&& \\&\ar[dl]_{p}V\ar[dr]^{q} & &\ar[dl]_{r}W\ar[dr]^{\tau}\\(M,\mathcal{F}_M)&&(P,\mathcal{F}_P)&&(N,\mathcal{F}_N)}$$

$$(M,\mathcal{F}_{M}) \stackrel{\alpha}{\longleftarrow} V\times_{q,P,r}W\stackrel{\beta}{\longrightarrow} (N,\mathcal{F}_N)$$ with $\alpha=p\circ\mathrm{pr}_V$ and $\beta=\tau\circ\mathrm{pr}_W$, is a Morita equivalence between $\mathcal{F}_M$ and $\mathcal{F}_{N}$ (as the composition of Morita equivalences of singular foliations \cite{Garmadia-Zambon}), that is $\alpha^{-1}(\mathcal{F}_M)=\beta^{-1}(\mathcal{F}_{N})$. In particular, we have that $\Gamma(\ker T\alpha)+\Gamma(\ker T\beta)\subseteq \beta^{-1}(\mathcal{F}_N)$. We need to prove the other inclusion. By Lemma \ref{lemma:pull-back} $$(P, \mathcal F_P)\stackrel{\lambda}{\leftarrow}V\times_{q,P,r}W\stackrel{\beta}{\rightarrow} (N, \mathcal F_N)$$ is a bi-submersion, where $\lambda\colon V\times_{q,P,r}W \stackrel{}{\longrightarrow}  P,\; (v,w)\longmapsto q(v)=r(w)$. Using the fact that $\ker(T\lambda) \subseteq \ker(T\mathrm{\mathrm{pr}_V} ) \oplus \ker(T\mathrm{pr}_W)$ it follows that
\begin{align*}
    \beta^{-1}(\mathcal{F}_N)&=\Gamma(\ker T\beta)+\Gamma(\ker T\lambda),\qquad\text{(bi-submersion property)}\\&\subseteq \Gamma(\ker T\beta)+\Gamma(\ker T\mathrm{pr}_V) + \Gamma(\ker T\mathrm{pr}_W)\\&\subseteq\Gamma(\ker T\beta)+\Gamma(\ker T\mathrm{pr}_V),\quad\text{(since $\Gamma(\ker T\mathrm{pr}_W)\subset \Gamma(\ker T\beta)$)}\\&\subseteq \Gamma(\ker T\alpha)+\Gamma(\ker T\beta) ,\qquad\text{(since $\Gamma(\ker T\mathrm{pr}_V)\subset \Gamma(\ker T\alpha)$)}.
\end{align*}Hence, $\alpha^{-1}(\mathcal{F}_{M})=\beta^{-1}(\mathcal{F}_N)=\Gamma(\ker T\alpha)+\Gamma(\ker T\beta)$.
\end{proof}

We continue by describing two methods for making bi-submersions “bigger” and “smaller”, respectively. The proof of the first lemma is left to the reader.

\begin{lemma}\label{lemma:pull-back}Let $(M, \mathcal{F}_M)\stackrel{p}\leftarrow W\stackrel{q}{\rightarrow}(N, \mathcal{F}_N)$ be a bi-submersion. Then, for every submersion $\pi\colon  V \rightarrow W$ the composition  $(M, \mathcal{F}_M)\stackrel{p\circ\pi}\longleftarrow V\stackrel{q\circ\pi}{\longrightarrow}(N, \mathcal{F}_N)$
is a bi-submersion. In particular, the restriction of a bi-submersion to an open subset is a bi-submersion.
\end{lemma}

The next proposition is less straightforward and can be used to reduce the dimension of a bi-submersion. It uses the bi-vertical singular foliation: more on the latter is given in \S\ref{sec:Bi-verticalFoliation}.

\begin{proposition}
\label{prop:reduced-bisub}
    Let $(M, \mathcal{F}_M)\stackrel{p}\leftarrow W\stackrel{q}{\rightarrow}(N, \mathcal{F}_N)$ be a bi-submersion.  For every sub-manifold $W'\subset W$ that intersects $\mathcal{BV}_W$ cleanly\footnote{i.e., such that for every $w' \in W'$
    $$ T_{w'} W = T_{w'} W' + T_{w'} \mathcal{BV}_W $$
where $T_{w'} \mathcal{BV}_W$ is the subspace of tangent vectors through which a bi-vertical vector field exists. Recall that  $\mathcal{BV}_W$ stand for bi-vertical vector fields on $W$.}:
    \begin{enumerate}
        \item the restrictions $p_{W'}\colon W'\to M$ and  $q_{W'}\colon W'\to N$ of the maps $p\colon W\to M$ and $q\colon W\to N$ respectively, are submersions;

        \item and $(M, \mathcal{F}_M)\stackrel{p_{W'}}\longleftarrow W'\stackrel{q_{W'}}{\longrightarrow}(N, \mathcal{F}_N)$ is a bi-submersion.
    \end{enumerate}
    \end{proposition}

    \begin{proof} Let us show
      that the differential  $T_{x}p_{W'}\colon W'\to T_{p(x)}M$ is surjective for all $x\in W'$. Since $p$ is a submersion, for any $v\in T_{p(x)}M$, there exists $\widetilde{v}\in T_xW$ such that $T_xp(\widetilde v)=v$. 
       Since  $\iota_{W'}\colon W'\hookrightarrow  W$ is transverse submanifold to $\mathcal{BV}_W$, there is by definition a decomposition $$\widetilde v=\underbrace{\widetilde v_1}_{\in T_xW'}+ \underbrace{\widetilde v_2}_{\in T_x\mathcal{BV}_W}.$$ Since  $T_xp(\widetilde v_2)=0$, we have $T_xp(\widetilde v_1)=v$. The same argument applies to $q_{W'}$. This proves item 1.

Since $\iota_{W'}\colon W'\hookrightarrow  W$ intersects $\mathcal{BV}_W\subset \mathcal{F}_W$ cleanly, it also intersects the pull-back foliation $\mathcal{F}_W=p^{-1}(\mathcal{F}_M)=q^{-1}(\mathcal{F}_M)$ cleanly. In particular, \begin{align*}
    p_{W'}^{-1}(\mathcal{F}_M)&=(p\circ \iota_{W'})^{-1}(\mathcal{F}_M)=\iota_{W'}^{-1}(p^{-1}(\mathcal{F}_M)=\iota_{W'}^{-1}(q^{-1}(\mathcal{F}_M)=(q\circ \iota_{W'})^{-1}(\mathcal{F}_M)\\&=q_{W'}^{-1}(\mathcal{F}_N)\supseteq \Gamma(\ker Tp_{W'})+\Gamma(\ker Tq_{W'}).
\end{align*}
We only need to show that $p_{W'}^{-1}(\mathcal{F}_M)\subseteq\Gamma(\ker Tp_{W'})+\Gamma(\ker Tq_{W'})$. 

Let $X \in \mathcal F_M$. 
By the bi-submersion property, there is a vector field $\widetilde X\in \Gamma(\ker Tq)$ such that $Tp(\widetilde X)=X\circ p$, as sections of $p^* TW $. Since $W'$ intersects $ \mathcal{BV}_W$ cleanly, we have a decomposition:
 $$ \widetilde X = \underbrace{\widetilde X^{W'}}_{\in \mathfrak X_{W'}}+ \underbrace{\widetilde Y}_{\in \mathcal{BV}_W}, $$
where $ \mathfrak X_{W'}$ stands for vector fields on $W$ tangent to $W'$. 
 The vector field $\widetilde X^{W'}$ belongs to $ p^{-1}(\mathcal F_M) $ and to $\Gamma(\ker Tq)$. As a consequence, its restriction $ \bar{X} $ to $W'$ belongs to $ p^{-1}_{W'}(\mathcal F_M) $ and to  $\Gamma(\ker Tq_{W'})$. By construction, $ Tp_{W'}(\bar{X}) = X \circ p_{W'} $.

For $Z\in p_{W'}^{-1}(\mathcal F_M)$, there exist a finite number of smooth functions $(g_i)_{i\in I}$ on $W'$ and vector fields $(X_{i})_{i\in I}$ of $\mathcal{F}_M$ such that, as sections of $ p_{W'}^* TM \to W'$:
$$Tp_{W'}(Z)= \sum_{i\in I}g_i \, \, X_i\circ p_{W'}.$$
Let $ \bar{X}_i$ be a lift of $ X_i$, as above, to a section of $\Gamma(\ker Tq_{W'})$.  By construction, $Z-\sum_{i\in I}g_i\bar{X}_i \in \Gamma(\ker Tp_{W'})$. This proves item 2.
%
\end{proof}

We finish with a last lemma.

\begin{lemma}
Let $ (M,\mathcal F_M) \stackrel{p}{\leftarrow} W \stackrel{q}{\rightarrow} (N,\mathcal F_N) $ be a bi-submersion. For any submanifolds $ M' \subset M$ and $N' \subset N$ that intersect $\mathcal F_M $ and $ \mathcal F_N$ cleanly, then
 $$    W_{M'}^{N'}:=p^{-1}(M') \cap q^{-1} (N')  $$
 is a\footnote{Or course, it may be the empty set, but the empty set is a bi-submersion.} bi-submersion for the restricted singular foliations
  $ \mathcal F_{M'}$ on $M'$ and $  \mathcal F_{N'}$ on $N'$.
\end{lemma}
\begin{proof}
Without any loss of generality, one can assume $ N'=N$. Let us explain why: assume that we have shown that for any bi-submersion as in the Lemma, and any $M'\subset M$ that intersect $ \mathcal F_{M}$ cleanly,  $ p^{-1}(M') \subset W $ is a bi-submersion between the induced foliated manifold $ (M',\mathcal F_{M'})$ and $ (N,\mathcal F_N)$. It suffices then to apply the same partial result a second time, but this time to the submanifold $N' \subset N $, in order to derive the desired result.

Let us therefore assume $ N'=N$. Since $ p$ is a submersion, $ p^{-1}(M')$ is a submanifold of $W$, and $p$ restricts to a submersion $ p: p^{-1}(M') \to M'$.
Let us first show that the restriction of $q$ to $ p^{-1}(M')$ is still a submersion valued in $ N$ (not necessarily onto). Consider an arbitrary point $w \in  p^{-1}(M')$ and an arbitrary tangent vector $ v \in T_{p(w)} N$.  
Since $q: W \to N$ is a submersion, there exists $\tilde{v}  \in T_w W$ which projects on $ w$.  In view of decomposition
 \begin{equation} \label{eq:sommedirectepardef} T_{p(w)} M' + T_{p(w)}\mathcal F_M  = T_{p(w)} M \end{equation}
there exists $ a \in T_w p^{-1}(M') \subset T_w W $ and $ b \in \mathrm{ker}(T_ w q) \subset T_w W $ such that 
 $$ T_w p \, (a) + T_w p \, (b)  = T_w p \, (   \tilde{v}) .$$
 This implies that the difference $ \tilde{v}-b \in T_w W $  satisfies $T_w q( \tilde{v}-b) = v $  and 
  $ T_w p( \tilde{v}-b ) \in T_{p(w)}  M'$ by construction.
As a consequence, $ \tilde{v}-b$ belongs to $ p^{-1}(M')$ and projects on $v $ through $ T_w  q$. This proves that the restriction of $q$ to $p^{-1}(W') $  is a submersion. 

It also follows from Equation \eqref{eq:sommedirectepardef} that $ T_w p^{-1}(W') + T_w p^{-1}(\mathcal F_{M} ) = T_w W $, i.e., $ p^{-1}(W')$ intersects   
 $ p^{-1}(\mathcal F_{M} )= q^{-1}(\mathcal F_{N} ) $ cleanly. 
Taking the pull-back map of singular foliations by submersions or inclusions of transverse submanifolds is contravariant. This implies that
 $$  p^{-1}_{|_{p^{-1}(W')}} (\mathcal F_{M'} ) = q^{-1}_{|_{p^{-1}(W')}}(\mathcal F_{N} ).$$
Now, let $ X \in \mathcal F_{M'}$. Let $ \tilde{X} \in \mathcal F_M$ whose restriction to  $M'$ is $X$. There exists a section of $ u \in {\mathrm{ker}}(Tq) $ such that $ Tp( u) = \tilde{X}$. The restriction of $u$ to $ p^{-1}(M') $ is a section of $ {\mathrm{ker}}(Tq)$ that $ Tp$ projects to $X$. 
For any vector field  $Z \in   p^{-1}(\mathcal F_{M'})$, there exists near every point local generators $X_1, \dots, X_k$ of $\mathcal F_{M'}$  such that $ Tp (Z) - \sum_{i=1}^n \varphi_i \circ p  \, X_i =0 $ as sections of $ p^! TM'$. Let $ \bar{X}_1, \dots, \bar{X}_k$ be sections of $ {\mathrm{ker}}(Tq)$ such that  $ Tp(\bar{X_i})=X_i$ as before, then $ Z - \sum_{i=1}^k \varphi_i \bar{X}_i $ belongs to  $ {\mathrm{ker}}(Tp_{|_{p^{-1}(W')}})$. This completes the proof.
\end{proof}

 \subsection{A generalization of bisections: bi-transversals}\label{sec:bitransversals}
 The concept of bisection, as introduced in \cite{AS} for a bi-submersion over singular foliation $(M, \mathcal{F})$ turns out to be inadequate for the general case of bi-submersions between foliated manifolds $(M, \mathcal{F}_M)$ and $(N, \mathcal{F}_N)$, especially when $M$ and $N$ have different dimensions. In this section, we propose a new concept to remedy this deficiency. 

\begin{definition}\label{def:bi-transversal}
Let $(M, \mathcal{F}_M) \stackrel{p}{\longleftarrow} W \stackrel{q}{\longrightarrow} (N, \mathcal{F}_N) $
be a bi-submersion.  A \emph{bi-transversal at $ w \in W$} is a  pointed submanifold $ (\Sigma, w) \subset W $  that is transverse\footnote{I.e., it intersects $ \mathcal F_W$ cleanly and $T_w W= T_w \mathcal F_w \oplus  T_w \Sigma $.} to  $\mathcal{F}_W=p^{-1}(\mathcal F_M)=q^{-1}(\mathcal F_N)$ at $w$. When there is no possible confusion, we shall denote $(\Sigma, w)$ simply by $\Sigma$.
\end{definition}

\begin{example}\label{ex:typicaltype}
For a bi-submersion of typical type as in Definition \ref{def:typicaltype} with model a path holonomy bi-submersion $ \mathcal W\subset \Sigma \times \mathbb R^d $, the subset $$\{(0,(\sigma,0),0)  {\hbox{ with }} \sigma \in \Sigma \} \subset \mathcal B^{d_M} \times \mathcal W \times \mathcal B^{d_N} $$ 
is a bi-transversal at $ (0,(O,0),0)$. 
\end{example}

\begin{remark}
A bisection of a bi-submersion  $(M, \mathcal{F}) \stackrel{p}{\longleftarrow} W \stackrel{q}{\longrightarrow} (M, \mathcal{F}) $ over a singular foliation $(M, \mathcal{F})$ in the sense of \cite{AS} is a submanifold $\Sigma\subset W$ such that  $\forall \sigma \in \Sigma$
 $$T_\sigma \Sigma \oplus {\mathrm{ker}}(T_\sigma p)= T_\sigma \Sigma \oplus {\mathrm{ker}}(T_\sigma q)= T_\sigma W.
 $$ 
 Such a $\Sigma$ intersects $p^{-1}(\mathcal F_M)=q^{-1}(\mathcal F_N)\subset \mathfrak X(W)$ cleanly, but is not, in general, a bi-transversal. A bisection through a point $ w \in W$ such that all vector fields in $\mathcal F_M $ vanish at $p(w)$ and $q(w)$ is, however, a bi-transversal at $w$.
\end{remark}
\begin{remark}
    Notice that the notion of bi-transversal or bisection in Definition \ref{def:bi-transversal} does not use the second axiom in Definition \ref{def:bi-submersion}. Hence, it makes sense for Morita equivalence of singular foliations.
\end{remark}
The following lemma provides a justification for the terminology ‘‘bi-transversal of a bi-submersion” as used in Definition \ref{def:bi-transversal}.
\begin{lemma}\label{lemma:bi-transversal}
Let $(M, \mathcal{F}_M) \stackrel{p}{\longleftarrow} W \stackrel{q}{\longrightarrow} (N, \mathcal{F}_N) $
be a bi-submersion. For any $w \in W$, there exists at least a bi-transversal $(\Sigma,w)$ through $w$.
Moreover, for any bi-transversal   $(\Sigma,w) $, there exists a neighborhood $ \Sigma'$ of $ w$ in $ \Sigma$ such that
\begin{enumerate}
\item The restrictions of $p$ and $q$ to $ \Sigma'$ are injective immersions in a neighborhood of $w$.
\item The images $p( \Sigma')$, $q(\Sigma')$ are  submanifolds of $M$ and $N$ intersecting $\mathcal F_M,\mathcal F_N $ cleanly and transverse to the leaves $L_{p(w)}\subset M$ at $p(w)$ and $L_{q(w)}\subset N$ at $q(w)$, respectively.
\end{enumerate}
\end{lemma}
\begin{proof}
    There exist vector subspaces $V \subset T_w W$  that do not intersect  $ \ker(T_w p)+ \ker(T_w q)$. In particular, both $T_w p\colon V \to T_{p(w)}M$ and $ T_w q\colon V \to T_{q(w)}M$ are linear invertible maps on their images. Any submanifold through $w$ admitting $V$ as its tangent space admits a restriction to a neighborhood $\Sigma$ of $w$ which is a bi-transversal.

    \begin{enumerate}
        \item By assumption, the restriction of $p,q$ on $(\Sigma, w)$ are immersions. They are injective immersions in a neighborhood of $w$ since immersions are locally injective. In particular, their images $p( \Sigma,w)$, $q(\Sigma,w)$ are submanifolds of $M$ and $N$.
        \item This is a direct consequence of item 2 of  Proposition \ref{prop:operations-on-transversals}.
    \end{enumerate}
    
\end{proof}
\begin{definition}A bi-transversal that satisfies both items of Lemma
    \ref{lemma:bi-transversal} will be said to be \emph{good}.
\end{definition}

Lemma
    \ref{lemma:bi-transversal} means that every bi-transversal $ (\Sigma,w)$ can be assumed to be good, by shrinking to a smaller open neighborhood of $w$ if necessary.  To a good bi-transversal $(\Sigma,w) $, one associates a diffeomophism
 $$\underline{\Sigma} \colon p(\Sigma) \stackrel{\sim}{\longrightarrow} q(\Sigma)$$ 
 such that the following diagram of diffeomorphisms commutes:
\begin{equation}\label{diag:trans-diffeo}
    \xymatrix{&\Sigma\subset W\ar[dr]^{q|_\Sigma}\ar[ld]_{p|_\Sigma} & &\\ p(\Sigma)\ar[rr]_{\underline{\Sigma}}& &q(\Sigma)&}
\end{equation}

\begin{proposition}[Bi-transversals induce isomorphism of transverse foliations]\label{prop:inducedsymetry}
Let $(M, \mathcal{F}_M) \stackrel{p}{\longleftarrow} W \stackrel{q}{\longrightarrow} (N, \mathcal{F}_N) $
be a bi-submersion. For every good bi-transversal $\Sigma\subset W$, the induced diffeomorphism $$\underline{\Sigma} \colon p(\Sigma) \longrightarrow q(\Sigma) $$ is an isomorphism of transverse singular foliations from $ \left(p(\Sigma), \mathcal F_M|_{p(\Sigma)}\right)$ to $\left(q(\Sigma), \mathcal F_N|_{q(\Sigma)}\right) $. 
\end{proposition}

In particular, denote by $ \mathcal F_\Sigma$ the induced singular foliation on $W$. The commutative diagram \eqref{diag:trans-diffeo2} is made of isomorphisms of singular foliations
\begin{equation}\label{diag:trans-diffeo2}
    \xymatrix{&(\Sigma, \mathcal F_\Sigma) \ar[dr]^{q|_\Sigma}\ar[ld]_{p|_\Sigma} & &\\  \left(p(\Sigma), \mathcal F_M|_{p(\Sigma)}\right)\ar[rr]_{\underline{\Sigma}}& & \left(q(\Sigma), \mathcal F_N|_{q(\Sigma)}\right).&}
\end{equation}

\begin{proof}
    It is enough to show that the restrictions  $p|_\Sigma,\; q|_\Sigma$ of $p,\;q$ to $\Sigma$ are  diffeomorphisms of singular foliations from 
$ \left(\Sigma, \mathcal F_W|_\Sigma \right)$ to $\left(p(\Sigma), \mathcal F_M|_{p(\Sigma)}\right) $ and $ (\Sigma,\mathcal F_W|_\Sigma) $ to $\left(q(\Sigma), \mathcal F_N|_{q(\Sigma)}\right)$ respectively. Then, the result will follow from Diagram \eqref{diag:trans-diffeo}. To begin, observe that according to Proposition \ref{prop:transverse-foliation}(1), the restriction $\left(\mathcal F_W\right)_{\Sigma}$ of $\mathcal F_W$ to $\Sigma\subset W$ is a singular foliation on $\Sigma$ due to the transversality of $\Sigma$ to $\mathcal F_W$. Likewise, by Lemma \ref{lemma:bi-transversal}, the restrictions $(\mathcal{F}_M)_{p(\Sigma)}$ and $(\mathcal{F}_N)_{q(\Sigma)}$ are singular foliations. Since the restriction $p_{|_\Sigma} $ of $p$ to $ \Sigma$ is a diffeomorphism, it is possible to consider the push-forward singular foliation $$(p_{|_\Sigma})_* \left(\left(\mathcal F_W\right)_{\Sigma} \right).$$ It is by construction a singular foliation on $p(\Sigma)$.
We first show the inclusion $(\mathcal F_N)_{p(\Sigma)}\subseteq ({p_{\Sigma}})_{*}(\left(\mathcal F_W\right)_\Sigma)$. Let $u \in (\mathcal F_N)_{p(\Sigma)}$. There exists an unique vector field $u^{\Sigma}\in\mathfrak X(\Sigma)$ such that $ Tp_\Sigma(u^\Sigma)=u$. 
 Let $ v$ be a vector field in $p^{-1}(\mathcal F_M) $ that $p$-related to $u$. 
For every $\sigma\in \Sigma$, the difference $  u_{\sigma}^\Sigma-v_\sigma  $  is valued in ${\mathrm{ker}}(T_\sigma p)$. 
In view of the decomposition
 $$  T_\sigma W = T_\sigma \Sigma \oplus {\mathrm{ker}}(T_\sigma p) ,$$
there exists a vector field in $Z \in \Gamma( {\mathrm{ker}}(T p))$ such that $u_{\sigma}^\Sigma -v_\sigma= Z_\sigma$ for all $ \sigma \in \Sigma$. Consider $u^! := v +Z $. The vector field $ u^!$ belongs to $p^{-1}(\mathcal F_M)$ by definition. Also, by construction, $u^!$ is tangent to $\Sigma$. Last, it coincides with $u^\Sigma$ on $ \Sigma$. This proves that $u^{\Sigma} \in \left(\mathcal{F}_W\right)_\Sigma $. Therefore,  the desired inclusion holds.

\vspace{1mm}
\noindent
Let us show the opposite inclusion $(p_{|_{\Sigma}})_* \left( \mathcal{F}_W\right)_\Sigma  \subseteq (\mathcal F_M)_{p(\Sigma)} $ : let $ v$ be a vector field in $\mathcal{F}_W$ that happens to be tangent to $\Sigma$. We need to show that there exists a vector field $ \tilde{v}$ which coincides with $v$ on $ \Sigma$ and is $p$-related to a vector field in $\mathcal F_M$. By construction, $v= \sum_i g_i v_i$ where the $g_i$ are local smooth functions on 
$W$ and $v_i\in \mathcal{F}_W $ are $p$-related to elements  $u_i \in \mathcal F$. Let $ \tilde{g}_i$ be functions on $W$ that coincide with $g_i$ on $\Sigma$ and are constant along $p$-fibers, i.e., $\tilde{g}_i=s^*f_i$ for some smooth function $f_i$ on $M$. The vector field $ \tilde{v} = \sum_i \tilde{g}_i v_i$ is tangent to $\Sigma $ (since it coincides with $v$ on $ \Sigma$) and is $p$-related with $\sum_i f_i u_i \in \mathcal F_M $. Hence $ p_*(v|_{\Sigma}) = \sum_i f_i u_i |_{p(\Sigma)}$
and in particular $(p_\Sigma)_*(v|_{\Sigma})\in \mathcal{F}_{p(\Sigma)}$.   
In turn, this implies $(p_{\Sigma})_* \left( \mathcal{F}_W \right)_\Sigma \subset (\mathcal F_M)_{p(\Sigma)} $. This same conclusion  is valid for $$(q_{|_\Sigma})_* \left(\left(\mathcal F_W\right)_{\Sigma} \right)=(\mathcal{F}_N)_{q({\Sigma})}.$$  This completes the proof.
\end{proof}

\begin{example}
For the bi-submersions of typical type and the bi-transversal in Example \ref{ex:typicaltype}, 
$ \underline{\Sigma}$ coincides with what we called in Definition \ref{def:typicaltype} the transverse diffeomorphism.
\end{example}

The following proposition is immediate.
\begin{lemma}\label{lemma:bisection}
    {Let $\Phi\colon W\to W'$ be a bi-submersions morphism  and $\Sigma\subseteq W$ a bi-transversal of $W$. The image  $\Phi(\Sigma)\subset W'$ is a bi-transversal of $W'$ and $\underline{\Phi(\Sigma)}=\underline \Sigma$.} 
    
    The same properties hold for good  bi-transversals.
\end{lemma}

Lemma \ref{lemma:bi-transversal} and Proposition \ref{prop:inducedsymetry} imply the following result (already obtained for Morita equivalences by Garmendia and Zambon's \cite{Garmadia-Zambon}).

\begin{corollary}\label{cor:bitrans}
Given a bi-submersion $(M, \mathcal{F}_M)\stackrel{p}{\leftarrow}W\stackrel{q}{\rightarrow}(N, \mathcal F_N) $, for every  $w\in W$, the leaves of $\mathcal{F}_M$ through $ p(w)$ and of $\mathcal{F}_N$ through $ q(w)$ have isomorphic transverse singular foliations.
\end{corollary}

\subsection{Local structure of a bi-submersion}\label{Local-structureofbis}

 Corollary \ref{cor:bitrans} allows associating to any point $w$ in a bi-submersion $W$ a typical type as in Definition \ref{def:typicaltype} using the diffeomorphism $ \underline{\Sigma}$ associated to a good bi-transversal $ (\Sigma,w)$. It is particularly interesting if we choose its model to be a path holonomy bi-submersion for $ (\Sigma,\mathcal F_\Sigma)$, as we now see.
 
\begin{proposition}
\label{prop:localstructure}
Consider a bi-submersion $(M, \mathcal{F}_M)\stackrel{p}{\leftarrow}W\stackrel{q}{\rightarrow}(N, \mathcal F_N) $.
 Let $w$ be a point in $W$, and $ \Sigma$ be a good bi-transversal through $w$.
We use notations and conventions of Equation \eqref{diag:trans-diffeo2}, in particular we
assume that $ \mathcal F_\Sigma$ is finitely generated.
Then:
\begin{enumerate}
 \item For any  path holonomy bi-submersion $ \mathcal A $ on $ (\Sigma,\mathcal F_\Sigma)$, there is a local morphism $ \Xi$ of bi-submersions from a neighborhood of $w$ in $W$ and a bi-submersion of typical type of model $\mathcal A$ and transverse diffeomorphism $\underline{\Sigma}$.
    \item There exists a path holonomy bi-submersion for which this morphism $\Xi$ can be chosen to be an isomorphism.
    \item There exists a minimal  path holonomy bi-submersion for which this morphism $\Xi$ can be chosen to be a submersion.
\end{enumerate}
\end{proposition}

Let us spell this statement differently: the first item claims that there exist 
\begin{enumerate}
\item neighborhoods $W^w$, $M^{p(w)} $,  $M^{q(w)}$ of the points $ w,p(w),q(w) $ in $W,M,N$ respectively, 
\item and a map  $\Xi : W^w \longrightarrow \mathcal B^{d_M}  \times  \mathcal A  \times \mathcal B^{d_N}$
\end{enumerate}
such that the following diagram commutes:
  \begin{equation}
      \label{diag:local-structures}
 \xymatrix{    &  W^w  \ar@{->>}[ddd]^{\Xi} \ar[dl]_p \ar[dr]^q &  \\ M^{p(w)} \ar[d]^{\simeq \Phi_M } & &  N^{q(w)} \ar[d]^{\simeq \Phi_N }  \\   \mathcal B^{d_M} \times \Sigma & &  \Sigma  \times  \mathcal B^{d_N} \\  &  \mathcal B^{d_M} \times  \mathcal A  \times  \mathcal B^{d_N} \ar[ur]^{(q_\mathcal A, \mathrm{pr}_3 ) }  \ar[ul]_{(\mathrm{pr}_1, p_\mathcal A ) }& }
  \end{equation}
 where $ d_M,d_N$ are the dimensions of the leaves through $ p(w)$ and $q(w)$, respectively, and $ \Phi_M,\Phi_N$ are local splittings. Here, we denote by $ \mathcal B^d$ the unit ball in $ \mathbb R^d$. The second and third item claim that $ \Xi$ can be chosen to be a diffeomorphism or a submersion for a well-chosen path holonomy bi-submersion on the transverse singular foliation.
\begin{proof}We give a sketch of the proof. A direct computation shows that the composition of the maps $$\Sigma \stackrel{\mathrm{pr}_{\Sigma}\circ \Phi_M}{\longleftarrow}M^{p(w)}\stackrel{p}{\longleftarrow}W^w\stackrel{q}{\longrightarrow} N^{q(w)}\stackrel{\mathrm{pr}_{\Sigma}\circ \Phi_N}{\longrightarrow}\Sigma$$ define a bi-submersion over $\mathcal{F}_\Sigma$ that induces identity at $w$. Since $\mathcal{A}$ is a path holonomy bi-submersion there exists a morphism of bi-submersions $\varphi\colon W^w\to \mathcal{A}$, see \cite[Proposition 2.10]{AS}. The map $\Xi=(\mathrm{pr}_1\circ \Phi_M\circ p,\varphi,\mathrm{pr}_2\circ\Phi_N\circ q)$ makes Diagram \eqref{diag:local-structures} commutes by construction.
The second and third items are generalizations of equivalent statements for bi-submersion over one singular foliation: any bi-submersion is locally of that form  see Exercise 2.4.47 and the fact that for a path holonomy bi-submersion made of a anchored bundle of minimal rank at a point $m$,  any morphism of bi-submersions  valued into it has to be a submersion. 
\end{proof}

\begin{remark}\label{rmk:Xifibers}
    Each vector field tangent to the fiber of $ \Xi$
in the third item of Proposition
    \ref{prop:localstructure} is a bi-vertical vector field.
\end{remark}

The next lemma is technical, but is of crucial importance.

\begin{lemma}\label{lem:projectable-generators}
    Let $(M,\mathcal{F}_{M}) \stackrel{p}{\leftarrow} W \stackrel{q}{\rightarrow} (N,\mathcal{F}_N)$ be a bi-submersion. The pull-back singular foliation $\mathcal{F}_W:=q^{-1}(\mathcal{F}_M)=p^{-1}(\mathcal{F}_N)$ is generated by $p$-projectable vector fields in  $\Gamma(\ker Tq)$ and by $q$-projectable vector fields in $\Gamma(\ker Tp)$. More precisely, for $w \in W$, let $ {X}_1, \dots, {X}_{n}$ be local generators of $ \mathcal F_M$ in a neighborhood of $p(w)$
and ${Y}_1, \dots, {Y}_{r}$ be local generators of $ \mathcal F_N$ in a neighborhood of $q(w)$.  Let $k=\mathrm{rk}(\ker Tp),\; \ell=\mathrm{rk}(\ker Tq)$ and $\kappa:=k+\ell$. Then $w$ admits a neighborhood $\mathcal U_w $ on which $\mathcal{F}_W$ is generated by vector fields $$ \overrightarrow {X_1}, \dots, \overrightarrow {X_{n}},\overleftarrow {Y_1}, \dots,\overleftarrow {Y_{r}},Z_1, \dots,Z_{\kappa}$$ satisfying the following properties:
\begin{enumerate}
\item  $\overrightarrow{X_1}, \dots, \overrightarrow{X_{n}} \in \Gamma({\mathrm{ker}}(Tq))$ are $p$-related to ${X}_1, \dots, {X}_{n}$ respectively, 
\item $ \overleftarrow{Y_1}, \dots, \overleftarrow{Y_{r}} \in \Gamma({\mathrm{ker}}(Tp))$ are $q$-related to $ {Y}_1, \dots, {Y}_{r}$ respectively,
\item  $Z_1, \dots,Z_{\kappa}\in\Gamma({\mathrm{ker}}(Tp)\cap \Gamma({\mathrm{ker}}(Tq)) $ are bi-vertical vector fields.
\item The vector fields $\overrightarrow{X_{1}}, \dots, \overrightarrow{X_{n}}, Z_1, \dots, Z_\kappa$ generate ${\mathrm{ker}}(Tq) $ at every point of $\mathcal U$ and the vector fields $\overleftarrow{Y_{1}}, \dots, \overleftarrow{Y_{r}}, Z_1, \dots, Z_\kappa$ span ${\mathrm{ker}}(Tp) $ at every point of $ \mathcal U_w$.
 \end{enumerate}
\end{lemma}
\begin{proof}
    The proof follows the same patterns for $(M, \mathcal{F}_M)=(N,\mathcal
    F_N)$ as in \cite[Lemma 4.6]{LLL2}. 
\end{proof}


As a consequence, we have the following corollary.

\begin{corollary}Let $(M,\mathcal{F}_{M}) \stackrel{p}{\leftarrow} W \stackrel{q}{\rightarrow} (N,\mathcal{F}_N)$ be a bi-submersion. The kernel
    $\Gamma(\ker Tp)\subset \mathfrak X(W)$ (resp. $\Gamma(\ker Tq)\subset \mathfrak X(W)$ ) is generated by $q$-projectable vector fields (resp. $p$-projectable vector fields). In particular, $\Gamma_{proj}(\ker Tp)$ (resp. $\Gamma_{proj}(\ker Tq)$) is finitely generated, and  $Tq\left(\Gamma_{proj}(\ker Tp)\right)=q^*\mathcal{F}_N$ and $Tp\left(\Gamma_{proj}(\ker Tq)\right)=p^*\mathcal{F}_M$.
\end{corollary}

{\subsection{Anchored bundles, bi-submersions, and bi-vertical vector fields}

Recall that an anchored vector bundle over $\mathcal{F}_M$ is by definition
a vector bundle $A\to M$ equipped with a vector bundle morphism  $A\stackrel{\rho}{\longrightarrow} TM$ over the identity of $M$ such that $\rho(\Gamma(A))=\mathcal{F}_M$. 
We call \emph{left action} a vector bundle morphism $\nu^{\mathcal{F}_M}\colon p^*A\to \ker Tq$ such that $Tp\circ \nu^{\mathcal{F}_M}=p^*\rho$.

Let us spell out the axiom of a left action: for all $ w \in W$, there has to exists a linear map 
\begin{equation}
\label{eq:leftaction}
\nu^{\mathcal{F}_M}_{w}  \colon p^*A|_w =  A_{p(w)} \longrightarrow  \ker T_w q \subset T_w W, \end{equation} depending smoothly on $w$, such that
    $$  \xymatrix{A_{p(w)} \ar[d]^{\rho} \ar[rr]^{\nu^{\mathcal{F}_M}_{w}}&  &  \ker T_w q \subset T_w W  \ar[dll]^{T_w p } \\ T_{p(w)} M &  & } .$$
Right actions are defined in the same manner.

\begin{lemma}\label{lem:biss_anchored}
Let $(M,\mathcal{F}_{M}) \stackrel{p}{\leftarrow} W \stackrel{q}{\rightarrow} (N,\mathcal{F}_N)$ be a bi-submersion. For any anchored bundle $A\stackrel{\rho}{\longrightarrow} TM$ over $\mathcal{F}_M$,
\begin{enumerate}
\item  a left action exists, 
    \item and, moreover, if
     $\mathrm{rk}(A)= \dim W- \dim N$ (resp. $\mathrm{rk}(A)\geq  \dim W-\dim N $), then, for any given $w \in W$, the vector bundle morphism \eqref{eq:leftaction} can be chosen to be an isomorphism of vector bundles in a neighborhood of $w$ (resp. a vector bundle morphism surjective in a neighborhood of $w$).
     \end{enumerate}
actions of $W$.
\end{lemma}

Symmetrically, a vector bundle $\nu^{\mathcal{F}_N}$ can be defined with the same properties for any anchored bundle over $ \mathcal F_N$: it is then called a \emph{right action}.

\begin{proof}[Proof (of Lemma \ref{lem:biss_anchored})]
 \begin{enumerate}
     \item Let $w\in W$ and $\mathcal{U}_w$ be an open neighborhood of $w$. Choose $(e_1, \ldots, e_{\mathrm{rk}(A)})$ to be a local trivialization of $A$ on $\mathcal{U}=p(\mathcal{U}_w)$. The vector fields $\rho(e_i)\in \mathfrak X(M)$ with $i=1,\ldots , r=\mathrm{rk}(A)$ generate $\mathcal{F}_M$ on $\mathcal{U}$. Let  $\overrightarrow{\rho(e_1)}, \ldots,\overrightarrow{\rho(e_{\mathrm{rk}(A)})}\in \Gamma(\ker Tq)$ and $Z_1, \ldots, Z_\kappa\in \Gamma(\ker Tp)\cap\Gamma(\ker Tq)$  be as in Lemma \ref{lem:projectable-generators}, i.e., there form a set of local generators of $\Gamma(\ker Tq)$ and $\overrightarrow{\rho(e_i)}\sim_p {\rho(e_i)}$ for every $i=1,\ldots, \mathrm{rk}(A)$. We define the map $\nu^{\mathcal{F}_M}$ on local generators by \begin{equation}
     \nu^{\mathcal{F}_M}_{\mathcal{U}_w}\colon \Gamma(p^*A)|_{\mathcal{U}_w}\longrightarrow \Gamma(\ker Tq)|_{\mathcal{U}_w},\; p^*e_i\longmapsto \overrightarrow{\rho(e_i)}
 \end{equation}
 and extend by $C^\infty(\mathcal U_w)$-linearity. These maps can be glued to a global one using partitions of unity. Indeed, there exists a countable family of points $(w_\alpha)_{\alpha\in I}$ of $W$ such that the open neighborhood of $\mathcal{U}_{w_\alpha}$ of $w_\alpha$ that form an open cover of $W$ and such that the $p(\mathcal{U}_{w_\alpha})$'s trivialize the vector bundle $A$. Let $(\chi_\alpha)_{\alpha\in I}$ be a partition
of unity subordinate to the open cover $\left(\mathcal{U}_{w_\alpha}\right)_{\alpha\in I}$. We define the map $\nu^{\mathcal{F}_M}$ on the sections of $p^*A$
as $\sum_{\alpha\in I}\chi_\alpha \nu^{\mathcal{F}_M}_{\mathcal{U}_\alpha}$. By construction, the latter induces a vector bundle morphism $\nu^{\mathcal{F}_M}\colon p^*A\to \ker Tq$ such that $Tp\circ \nu^{\mathcal{F}_M}=p^*\rho$. This proves item $1$.

\item For $w\in W$, let $n=\dim\left(\frac{\mathcal{F}_M}{\mathcal{I}_{p(w)}\mathcal{F}_M}\right)$ be the minimal number of generators of $\mathcal{F}_M$ in a neighborhood of $p(w)$. There exists a local trivialization $e_1, \ldots, e_n, e_{n+1}, \ldots, e_{\mathrm{rk}(A)}$ of $A$ on an open neighborhood of  $p(w)$ so that $\rho(e_i)=0$ for $i=n+1, \ldots, \mathrm{rk}(E)$. By Lemma \ref{lem:projectable-generators}(4), by definition of $n$ and since $\mathrm{rk}(A)=\mathrm{rk}(\ker Tq)$, there exist bi-vertical vector fields $\xi_{n+1}, \ldots, \xi_{\mathrm{rk}(E)}\in \Gamma(\ker Tp)\cap\Gamma(\ker Tq)$ such that $\overrightarrow{e_1}=\overrightarrow{\rho(e_1}), \ldots, \overrightarrow{e_n}=\overrightarrow{\rho (e_n)}, \xi_{n+1}, \ldots, \xi_{\mathrm{rk}(A)}$ form a local trivialization of $\ker Tq$ on an open neighborhood $\mathcal{U}_w$ of $w$. Hence, we define \begin{align*}
         \nu^{\mathcal{F}_M}|_{\mathcal{U}_w} \colon  \Gamma(p^*A)|_{\mathcal{U}_w}&\stackrel{\sim}{\longrightarrow} \Gamma(\ker Tq)|_{\mathcal{U}_w}\\  p^*e_i&\longmapsto \left\{\begin{array}{ll}
              \overrightarrow{e_i}& \text{for}\; i=1,\ldots,n \\
              \xi_i  & \text{for}\; i=n+1,\ldots,\mathrm{rk}(A).
         \end{array}\right.
         \end{align*}

The identity $Tp\left(\nu^{\mathcal{F}_M}(p^*e)\right)=\rho(e)\circ p$\,  implies  that $\nu^{\mathcal{F}_M}\colon p^*A\to \ker Tq$ restricts to a map  $$p^*(\ker\rho)\to \Gamma(\ker Tp)\cap\Gamma(\ker Tq).$$
 \end{enumerate}
This completes the proof.

\end{proof}

\subsection{The global rank of a singular foliation and bi-submersions}

Recall that the  minimal number of local generators of a singular foliation $\mathcal{F}$ near a point coincides with the rank of the holonomy Lie algebroid through that point. We denote it by $ {\mathrm{rk}}_x(\mathcal F)$, and we denote 
$${\mathrm{rk}(\mathcal F)}=\mathrm{max}_{x\in M} {\mathrm{rk}}_x{(\mathcal F)}$$
the upper bound of these ranks. By construction, ${\mathrm{rk}(\mathcal F)}$ belongs either to $ \mathbb N \cup \{+\infty\}$. It is always a finite integer provided that one of the following two equivalent conditions are satisfied:
\begin{enumerate}
    \item[(i)] $ \mathcal F$ is globally finitely generated, 
    \item[(ii)] $ \mathcal F$ admits a globally defined anchored bundle. 
\end{enumerate}
We refer to the discussion about ranks in \cite{zbMATH03423310}.
If one of these two equivalent conditions is satisfied, then the anchored bundle  over $ \mathcal F$ can be chosen to be a trivial vector bundle, and $ {\mathrm{rk}}({\mathcal F})$ is lower or equal to the rank of any anchored bundle over $ \mathcal F$. Such a singular foliation is said to be of \emph{global finite rank}.

\begin{lemma}
Let $W$ be a bi-submersion between two singular foliations  $$(M, \mathcal{F}_M)\stackrel{p}{\leftarrow}W\stackrel{q}{\rightarrow}(N, \mathcal F_N) $$ of global finite ranks. 
Assume $p$ and $q$ are surjective. Then:
\begin{equation} \label{eq:ranks}     {\mathrm{rk}} (\mathcal{F}_M )  - {\mathrm{dim}}(M)=   {\mathrm{rk}} (\mathcal{F}_N)    -  {\mathrm{dim}}(N) .\end{equation}
\end{lemma}
\begin{proof}
It suffices of course that this relation holds for the pull-back of a singular foliation $\mathcal F_M $  on $M$ through a surjective submersion $ \pi: P \to M$. In this case, it is an immediate consequence of the following equality, valid for all $p\in P$:
$$     {\mathrm{rk}}_p (\pi^{-1}\mathcal{F}_M) =  {\mathrm{rk}}_p (\mathcal{F}_M )+{\mathrm{dim}}(P) -{\mathrm{dim}}(M) $$
which expresses that a minimal set of local generators for $\pi^{-1}(\mathcal F_M) $ near $p\in P$ is obtained as the concatenation of  $ {\mathrm{dim}}(P) -{\mathrm{dim}}(M)$ vertical vector fields that form a local trivialization of the tangent spaces of the fibers of $\pi$ with a family  $ Y_1, \dots, Y_r$ of vector fields, with $ Y_i$ being $ \pi$-related to some $X_i \in  \mathcal F_M$ for all $i=1, \dots, r$, such that  $ X_1, \dots, X_r$ are a minimal set of generators of $\mathcal F_M $.
\end{proof}

\begin{remark}\label{rmq:rank-invariant}
\normalfont
For Equation \eqref{eq:ranks}, we only need $ W$ to be a Morita equivalence in the sense of Garmendia and Zambon's \cite{Garmadia-Zambon} and could be proven in discussion around Remark 2.4 in \cite{Garmadia-Zambon}. 
Equation \eqref{eq:ranks} is an analogous to the classical statement that the dimension of a differentiable stack, i.e., the rank of the Lie algebroid minus the dimension of the base, is Morita invariant, see, e.g., \cite{jsg/1310388899}.
This equation implies that 
\begin{eqnarray*} {\mathrm{rk}} (\mathcal{F}_M )  +  {\mathrm{rk}} (\mathcal{F}_N )  + {\mathrm{dim}}(M) +  {\mathrm{dim}}(N)  \\ 
 = 2 \left(    {\mathrm{rk}} (\mathcal{F}_M )  +  {\mathrm{dim}}(N) \right) \\
 =2  \left( {\mathrm{rk}} (\mathcal{F}_N )  + {\mathrm{dim}}(M) \right)
\end{eqnarray*}
is an even integer. In particular,
\begin{equation}\label{eq:dim:min} \frac{1}{2}\left( {\mathrm{rk}} (\mathcal{F}_M )  +  {\mathrm{rk}} (\mathcal{F}_N )  + {\mathrm{dim}}(M) +  {\mathrm{dim}}(N)  \right) 
\end{equation}
is an integer.
 \end{remark}

\subsection{Bi-vertical vector fields as a singular foliation}\label{sec:Bi-verticalFoliation}

Let us start by collecting several elementary results about bi-vertical vector fields on a bi-submersion.
Recall that for any bi-submersion $(M, \mathcal{F}_M)\stackrel{p}{\leftarrow}W\stackrel{q}{\rightarrow}(N, \mathcal F_N) $,  the sheaf 
$\mathcal{BV}_W $  of bi-vertical vector fields on $W$ forms a $ C^\infty(W)$-module, and is stable under the Lie bracket. It may not be a singular foliation, since it may not be finitely generated. %

Let us be more precise. Consider a left action $\nu^{\mathcal{F}_M}$ as in \eqref{eq:leftaction}. We still denote by $\nu^{\mathcal{F}_M}$ the induced linear map 
 \begin{equation}
 \label{eq:levelsections}
    \nu^{\mathcal{F}_M }\colon \Gamma(p^* A) \to \Gamma( \ker T q) 
 \end{equation}
  at the level of sections.
   The following holds:
  \begin{enumerate}
    \item The restriction to  $p^*(\ker\rho)$ of the linear map  \eqref{eq:levelsections}  is valued in bi-vertical vector fields:
 \begin{equation}
 \label{eq:levelsections2}
    \nu^{\mathcal{F}_M }\colon p^*(\ker\rho) \to \mathcal{BV}_W
 \end{equation}
    Here $p^*(\ker\rho)\subset \Gamma(p^*A)$ is the $\mathcal{O}_W$ sub-module generated by sections $\eta\in \Gamma(p^*E)$ such that $\rho(\eta(w))=0$ for all $w\in W$.

    \item If the left action $\nu^{\mathcal{F}_M}$ \eqref{eq:levelsections} is surjective (resp. bijective)  at every point, then the linear map \eqref{eq:levelsections2} is surjective (resp. bijective).
    \end{enumerate}

Since any singular foliation admits, near a point, an anchored bundle  $(A,\rho) $ whose rank satisfies $\dim W=\mathrm{rk}(A)+ \dim N$, a bijective left action $\nu^{\mathcal{F}_M}$ exists near every point of $W$.
In this case, the linear map \eqref{eq:levelsections2} is invertible. 
Bi-vertical vector fields form therefore a singular foliation on $W$ if and only if  $\ker\rho\subset \Gamma(A)$ is locally finitely generated. 
This happens if and only if  there exists in a neighborhood of every point of $M$ a short exact sequence sequence of vector bundles of the form 
 \begin{equation}
  \label{eq:shortexact}
  B\stackrel{\dd}{\longrightarrow} A\stackrel{\rho}{\longrightarrow} TM
  \end{equation} 
  such that $\rho(\Gamma(A))=\mathcal{F}_M$ and $\dd(\Gamma(B))=\ker\rho$, 
  and $$ \nu^{\mathcal F_M} \circ p^* \dd : \Gamma(p^* B)  \longrightarrow \mathcal{BV}_W $$
  is surjective, i.e.,  
  \begin{equation}
      \label{eq:anchoredforbivertical}
  \xymatrix{p^*B\ar[rr]^{\nu^{\mathcal{F}_M}\circ p^*\dd}&&TW}
  \end{equation}
   is an anchored bundle over $\mathcal{BV}_W$.
Let us recapitulate this discussion.

    \begin{lemma}\label{lem:anchoredbundleforbivertical}
    Let $(M,\mathcal{F}_{M}) \stackrel{p}{\leftarrow} W \stackrel{q}{\rightarrow} (N,\mathcal{F}_N)$ be a bi-submersion. The following items are equivalent: 
    \begin{enumerate}
       \item There exists near every point of $M$ a local sequence of vector bundles $B\stackrel{\dd}{\longrightarrow} A\stackrel{\rho}{\longrightarrow} TM$  near $ p(w)$  such that $\rho(\Gamma(A))=\mathcal{F}_M$, $\dd(\Gamma(B))=\ker\rho$, and  $\mathrm{rk}(A) = \dim W - \dim N$.
       \item Bi-vertical vector fields on $W$ form a singular foliation on $W$. 
\end{enumerate}
     In this case, an anchored bundle over $\mathcal{BV}_W$ near any $w\in W$ can be obtained as follow: choose a left action  $ \nu^{\mathcal{F}_M} $ as in Equation \eqref{eq:leftaction}. Then  Equation \eqref{eq:anchoredforbivertical}  describes an anchored bundle over $\mathcal{BV}_W$.
    \end{lemma}


Now, when does a local sequence as in item 1) in Lemma \ref{lem:anchoredbundleforbivertical} exist near every point of $M$? The answer is in fact quite classical in sheaf theory. It holds if and only if $ \mathcal{F}_M $ is what is called a coherent sheaf, i.e. if for one (equivalently all) anchored bundle $(A,\rho) $ over $ \mathcal F_M$ there exists a vector bundle $ B \to M$ and a vector bundle morphism $ d : B \to A$ such that every section of $A$ in the kernel of $ \rho$ belongs to the image of $d$.
Let us be more precise.

\begin{lemma}\label{lem:isasingularfoliation}
Consider a bi-submersion $(M, \mathcal{F}_M)\stackrel{p}{\leftarrow}W\stackrel{q}{\rightarrow}(N, \mathcal F_N) $
with surjective $p,q$.
The following items are equivalent:
\begin{enumerate}
    \item[(i)] 
$\mathcal{BV}_W $  is a singular foliation.
     \item[(ii)] $ (M, \mathcal{F}_M) $ is a coherent sheaf,
      \item[(iii)] $ (N, \mathcal{F}_N) $ is a coherent sheaf.
\end{enumerate}
\end{lemma}
\begin{proof}
 The equivalence $ (i)  \equiv (ii)$ is the content of Lemma
    \ref{lem:anchoredbundleforbivertical}.
    The equivalence $(ii) \equiv (iii) $
    is then obvious by symmetry.
\end{proof}


Here is a last lemma on the bi-vertical singular foliation $\mathcal{BV}_W $ of a bi-submersion $W$. 

\begin{lemma}
\label{lemma:codim}

Let $W$ be a bi-submersion between two singular foliations  $$(M, \mathcal{F}_M)\stackrel{p}{\leftarrow}W\stackrel{q}{\rightarrow}(N, \mathcal F_N) $$ of global finite ranks such that $\mathcal{BV}_W $ is a singular foliation.
Then the codimension of its leaves is lower or equal to the integer in \eqref{eq:dim:min}. 
\end{lemma}
\begin{proof}
 This lemma is a trivial consequence of item 3 in Proposition \ref{prop:localstructure}, since each connected component of each fiber of the submersion $\Xi$ in that statement is contained in 
 a leaf of the bi-vertical vector fields by Remark \ref{rmk:Xifibers}.
There is also a direct proof, without using this result. Since $p$ is a submersion:
$$  T_w \Sigma  \oplus {\mathrm{ker}}(T_w p)   = T_w W $$
Item 4 in Lemma \ref{lem:projectable-generators} implies that
the codimension of $T_w \mathcal{BV}_W $ in $   {\mathrm{ker}}(T_w p) $ is lower or equal to the number of local generators of $\mathcal F_N$. It is therefore lower or equal to $  {\mathrm{rk}}_{q(w)}(\mathcal F_N)  $.
Hence, the codimensions of the leaves of  $\mathcal{BV}_W $  are lower or equal to $   {\mathrm{dim}}(M)+{\mathrm{rk}}(\mathcal F_N)$, which coincides with the integer in \eqref{eq:dim:min} by Remark \ref{rmk:Xifibers}.
\end{proof}

\begin{remark}
It is a routine to check that (provided $p$ and $q$ are onto), there is one leaf of  $\mathcal{BV}_W $ whose codimensions in exactly the integer in \eqref{eq:dim:min}: it suffices to take $w$ such that $ {\mathrm{rk}}_{q(w)} (\mathcal F_N)$ is maximal.
\end{remark}

\begin{remark}
The bi-vertical singular foliation is a very special type of singular foliation. 
For instance, it is always, locally, defined by an action of the Abelian Lie group $\mathbb R^N$ for some integer $N$. This follows from the local model which shows that any point has a neighborhood where the leaves of the bi-vertical singular foliation are given by the obvious action of $ E_{-2} $ on $ E_{-1}$, see \S \ref{Local-structureofbis}. Also, the local model proves that the restriction to  a neighborhood of a point of the leaves are closed in that neighborhood. But there is more: its leaves are closed. This is an immediate consequence of the fact, proven below, that two points in the same leaf of the bi-vertical vector fields are equivalent,  and that every point has a neighborhood where equivalent points are in the same leaves as the restriction of that neighborhood of the bi-vertical singular foliation. 
\end{remark}

ˆ\section{Equivalent bi-submersions}
\label{sec:Equivalentbi-submersions}
\subsection{Definitions and first properties}


A notion of equivalence between bi-submersions was first introduced in \cite{AS} and later reformulated in \cite[\S 4.5]{LLL2}. In this work, we provide a natural generalization for bi-submersions between two different singular foliations, and investigate its meaning.

\begin{definition}
\label{def:related}
Consider two bi-submersions $(M, \mathcal{F}_{M}) \stackrel{p}{\leftarrow} W \stackrel{q}{\rightarrow} (N, \mathcal{F}_{N})$ and $(M,\mathcal{F}_{M}) \stackrel{p'}{\leftarrow} W' \stackrel{q'}{\rightarrow} (N,\mathcal{F}_{N})$. 

\begin{enumerate}
\item  $W$ and $W'$ are said to be \emph{equivalent} if there exists a manifold $V$ equipped with two surjective submersions $\delta \colon V \to W $ and $\eta \colon V \to W' $ such that the following diagram commutes:
\begin{equation}\label{weak-equivalence}
     \scalebox{0.7}{\xymatrix{&& (M,\mathcal{F}_M)&& \\W\ar[urr]^{p}\ar[drr]_{q}&&\ar@{->>}[ll]_\delta V\ar@{..>}[u]\ar@{..>}[d]\ar@{->>}[rr]^\eta && W'\ar[llu]_{p'}\ar[dll]^{q'}\\ && (N,\mathcal{F}_N)&&}}
 \end{equation}   
 \item The triple $(V,\delta,\eta)$ is then called a \emph{relation} between $W$ and $W'$.
\item  We say that two points $w \in W$ and $w' \in W'$ are \emph{equivalent} 
if there is an open neighborhood $\mathcal U\subseteq W$ of $w$ and an open neighborhood $\mathcal U'\subseteq W'$ of $w'$ which are equivalent.
\item  
We say 
that $w \in W$ and $w' \in W'$ are \emph{ related through a relation $(V,\delta,\eta)$} 
 if there exists $v \in V$ such that $\delta (v)=w$ and $ \eta(v)=w'$.
 \end{enumerate}
\end{definition}

\begin{remark}
Of course, if $w$ and $w'$ are related through  some relation $ (V,\delta,\eta)$, then they are equivalent. 
But even if $W$ and $W'$ are equivalent, two equivalent points $ w \in W$ and $w' \in W' $ do not need to be related through a given relation $ (V,\delta,\eta)$ between $W$ and $W'$. 
\end{remark}

\begin{remark}
The commutative diagram \eqref{weak-equivalence}
 induces a commutative diagram of vector bundle morphisms:
 \begin{equation}\label{diag:2}
    \scalebox{0.7}{\xymatrix{&& TM&& \\TW\ar[urr]^{Tp}\ar[drr]_{Tq}&&\ar@{->>}[ll]_{T\delta} TV\ar@{..>}[u]\ar@{..>}[d]\ar@{->>}[rr]^{T\eta} && TW'\ar[llu]_{Tp'}\ar[dll]^{Tq'}\\ && TN &&}}
 \end{equation}
 Restricting the above vector bundle morphisms to appropriate kernels, one obtains two complexes of vector bundle morphisms:
\begin{align}\label{diag:lesdeux}
&\ker T\delta \stackrel{T\eta }{\longrightarrow}\ker Tq' \stackrel{Tp' }{\longrightarrow} TM\\ \label{diag:lesdeux2}& \ker T\eta \stackrel{T\delta }{\longrightarrow}\ker Tp' \stackrel{Tq' }{\longrightarrow} TN.
\end{align}
 This will be useful in the sequel.
 \end{remark}

    The following two propositions are crucial for this paper, and reformulate similar results in \cite{AS}.

 \begin{proposition}[Local relation]\label{Local-equivalence}
Consider two bi-submersions $ (M,\mathcal F_M) \stackrel{p}{\leftarrow} W \stackrel{q}{\rightarrow} (N,\mathcal F_N) $ and $ (M,\mathcal{F}_M) \stackrel{p'}{\leftarrow} W' \stackrel{q'}{\rightarrow} (N,\mathcal{F}_N) $. 
For any two points $
w\in W$ and $w' \in W' $, the following statements\footnote{Notice that all these statements imply $ p(w)=p(w')$ and $ q(w)=q(w')$.} are equivalent:
\begin{enumerate}
    \item[(i)] $w \in W$ is equivalent to $w' \in W'$.
    \item[(ii)] There is a neighborhood $\mathcal U$ of $w$ in $W$  together with a morphism of bi-submersions $ \varphi\colon \mathcal U \to W'$ mapping $w$ to $w'$.
    \item[(iii)] There is a neighborhood\, $\mathcal U'$ of $w'$ in $W'$ together with a morphism of bi-submersions  $ \varphi' \colon \mathcal U' \to W$ mapping $w'$ to $w$.
    
    \item[(iv)] There exist bi-transversals through $w$ and $w'$ that induce the same germ of diffeomorphisms of transverse singular foliations. 
 \end{enumerate}
\end{proposition}

The proof requires the following convention.

\begin{convention}
    In what follows, we use the notation $\exp{(\lambda \xi)}:= \Phi^{\xi_1}_{\lambda_1}\circ\cdots\circ \Phi^{\xi_d}_{\lambda_d}$ for all $\xi=(\xi_1, \ldots, \xi_d)\in \mathfrak X(M)^d$ and $\lambda =(\lambda_1,\cdots, \lambda_d)\in \mathbb{R}^d$ where the flows make sense.
\end{convention}

\begin{proof}
$(i)\Rightarrow (ii)\; \text{and}\; (iii)$ : Assume that $w\in W$ and $w'\in W'$ are related through a manifold $(V,\delta, \eta)$ and let $\mathcal U\subseteq W$  be an open neighborhood of $w\in W$ and $\mathcal U'\subseteq W'$ an open neighborhood of $w'\in W'$. The composition of the map $\eta\colon V \to W' $ with a local section of $\delta \colon {V} \to \mathcal{U}\subseteq W $ is a morphism of a bi-submersion as in $(ii)$. The composition of the map $\delta\colon V \to W $ with a local section of $\eta\colon V \to \mathcal{ U}'\subseteq W' $ is a morphism as in $(iii)$.\\

\noindent
$(ii)\Rightarrow (iv)$ : By Lemma \ref{lemma:bisection}, the image of any bi-transversal $(\Sigma, w) $ of $W$ through a morphism $\phi \colon \mathcal U \to W' $ mapping $w$ to $w'$ (with $\mathcal U $ an open neighborhood of $w\in W$) is a bi-transversal   $(\Sigma', w')$ of $ W'$, and $ \underline{\phi(\Sigma)} = \underline{\Sigma}$. Likewise, we have $(iii)\Rightarrow (iv)$.\\

\noindent
$(iv)\Rightarrow (i)$ : Consider bi-transversals $(\Sigma, w)$ of $\mathcal{U}\subseteq W$ and $(\Sigma',w')$ of $\mathcal{ U}'\subseteq W'$ such that $\underline{\Sigma}=\underline{\Sigma'}$. In particular, the restriction $$\phi=\left(p'_{|_{\Sigma'}}\right)^{-1}\circ p_{|_{\Sigma}}\colon\Sigma\longrightarrow \Sigma'$$ 
    is a diffeomorphism that makes the following diagram commutes:
    $$ \xymatrix{ & M& \\  \Sigma \ar[rr]^{\phi} \ar[ru]^{p_{|_{\Sigma}}} \ar[rd]_{q_{|_{\Sigma}}}& & \Sigma' \ar[lu]_{p'_{|_{\Sigma'}}} \ar[ld]^{q'_{|_{\Sigma'}}} \\ & N &}.$$

     Let $\{{X}_i, i= 1,\ldots, n\}$ be a set of generators of $\mathcal F_N$ in a neighborhood of $ p(w)=q(w')$ and let  $\{{Y}_j, j= 1,\ldots, r\}$ be a set of generators of $\mathcal F_M$ in a neighborhood of $ q(w)=q'(w')$. Let  
      $$(\overleftarrow{X};\overrightarrow{Y};Z)=(\overleftarrow{X_1}, \dots, \overleftarrow{X_{n}};\overrightarrow{Y_1},\dots, \overrightarrow{Y_{r}}; Z_1, \dots,Z_{\kappa}) $$
      and
      $$(\overleftarrow{X}';\overrightarrow{Y}';Z')=(\overleftarrow{X_1'}, \dots, \overleftarrow{X_{n}'};\overrightarrow{Y_1'},\dots, \overrightarrow{Y_{r}'}; Z_1', \dots,Z_{\kappa'}') $$
      be as in Lemma \ref{lem:projectable-generators}, generators in a neighborhood of $w$ and $w'$ of the pull-back foliations \begin{equation*}
          \mathcal{ P}_\mathcal{U}:=p^{-1}({\mathcal F_N}|_{p(\mathcal{U})})=q^{-1}({\mathcal F_M}|_{q(\mathcal{U})})\quad\text{and} \quad \mathcal{P}'_{\mathcal{U}'}:={(p')}^{-1}({\mathcal F_N}|_{p'(\mathcal{U'})})=(q')^{-1}({\mathcal F_M}|_{q'(\mathcal{U}')}).
      \end{equation*}
     Without any loss of generality, one can assume $\kappa=\kappa'$: it suffices for instance to add $\kappa'-\kappa $ times the vector field $0$ if $ \kappa' >\kappa$.
     Let $V^{n,r, \kappa}_{w,w'}:=\mathrm{Graph}(\phi) \times \mathbb R^{n} \times \mathbb R^{r} \times \mathbb R^\kappa $,
     where $\mathrm{Graph}(\phi)$ is the graph of the diffeomorphism $\phi\colon\Sigma\longrightarrow \Sigma'$.  Define two maps $\delta_{w,w'}$ and $\eta_{w,w'}$ by:
     $$ \begin{array}{rrcl}\delta_{w,w'}\colon & \left((\sigma,\sigma'),\lambda, \mu, \nu\right)\in V^{n,r, \kappa}_{w,w'} &\mapsto& \exp{(\lambda X)}\circ \exp{(\mu Y)} \circ \exp{(\nu Z)}(\sigma)\in \mathcal U\\\\\eta_{w,w'}\colon&\left((\sigma,\sigma'), \lambda, \mu,\nu\right)\in V^{n,r, \kappa}_{w,w'}  &\mapsto &\exp{(\lambda X')}\circ \exp{(\mu Y')} \circ \exp{(\nu Z')}(\sigma')\in \mathcal U'  \end{array}$$
     Above $ (\sigma,\sigma')$ is an element of the graph of $ \phi$, i.e., $ \sigma'=\phi(\sigma)$, and $(\lambda,\mu,\nu)\in \mathbb R^{n} \times \mathbb R^{r} \times \mathbb R^\kappa $ are such that the flow considered is well-defined, which is always true in some open ball centered at zero. The diagram below commutes \begin{equation}\label{}
    \scalebox{0.7}{ \xymatrix{&& M&& \\\mathcal U\subseteq W\ar[urr]^{p}\ar[drr]_{q}&&\ar[ll]_{\delta_{w,w'}}  V^{n,r, \kappa}_{w,w'}\ar@{..>}[u]\ar@{..>}[d]\ar[rr]^{\eta_{w,w'}} && \mathcal{ U}'\subseteq W'\ar[llu]_{p'}\ar[dll]^{q'}\\ && N&& }}
 \end{equation}because of these easy following facts
\begin{enumerate} 
\item the vector fields $\overleftarrow{X_i}\in \Gamma(\ker Tq)$ (resp. $\overleftarrow{X_i'}\in \Gamma(\ker Tq')$) are $p$-related  (resp. $p'$-related) to the same vector fields in $\mathcal F_N|_{p(\mathcal U)}$, namely $X_1,\ldots,X_{n}$  and $q$-related  with the same vector fields on $ \mathcal F_M|_{q(\mathcal{U})}$, namely zero.
\item the vector fields $\overrightarrow{Y_j}\in\Gamma(\ker Tp)$ (resp. $\overrightarrow{Y_j'}\in \Gamma(\ker Tp')$) are $q$-related  (resp. $q'$-related) to the same vector fields in $\mathcal F_M|_{q(\mathcal U)}$, namely $Y_1,\ldots,Y_{n}$  and $p'$-related with the same vector fields on $ \mathcal F_M|_{q(\mathcal{U})}$, namely zero,
\item the vector fields $Z_l\in \Gamma(\ker Tp)\cap \Gamma(\ker Tq) $  are $p$- and $q$-related to zero; the vector fields $\zeta_l'\in \Gamma(\ker Tp')\cap \Gamma(\ker Tq')$ are $p'$- and $q'$-related to zero,
\item $p(\sigma)=p'(\sigma')$ and $q(\sigma)=q'(\sigma')$ for all $(\sigma,\sigma')\in \mathrm{Graph}(\phi)$.
\end{enumerate}  For any $\sigma\in \Sigma$ consider the point $O_\sigma := ((\sigma,\sigma'), 0_{\mathbb R^{n}}, 0_{\mathbb R^{r}}, 0_{\mathbb R^\kappa})$ in $V^{n,r, k}_{w,w'}$. It is not hard to see that the image of the differential of $\eta_{w,w'}$ (resp. $\delta_{w,w'}$) at the point  $O_\sigma$ is generated by $ T_\sigma\Sigma$ and $ T_\sigma \mathcal P$ (resp. $ T_{\sigma}\Sigma'$ and $ T_{\sigma'}\mathcal P'$). Both $\eta_{w,w'},\delta_{w,w'}$ 
are therefore surjective submersions (upon shrinking $\mathcal{U}$ and $\mathcal{U'}$ if needed) in an open neighborhood of the section   $\sigma \longmapsto O_\sigma$ in $\mathrm{Graph}(\phi) \times \mathbb R^{n} \times \mathbb R^{r} \times \mathbb R^\kappa$ that we  denote by $V^{n,r, \kappa}_{w,w'}$ again.    
That is, $w\in W$ and $w'\in W'$ are related through the following diagram
\begin{equation}\label{}
    \scalebox{0.7}{ \xymatrix{&& M&& \\\mathcal U\subseteq W\ar[urr]^{p}\ar[drr]_{q}&&\ar@{->>}[ll]_>>>>>>>>>>>>>{\delta_{w,w'}} V^{n,r, \kappa}_{w,w'}\subseteq \mathrm{Graph}(\phi) \times \mathbb R^{n} \times \mathbb R^{r} \times \mathbb R^\kappa\ar@{..>}[u]\ar@{..>}[d]\ar@{->>}[rr]^>>>>>>>>>>>>{\eta_{w,w'}} && \mathcal{ U}'\subseteq W'\ar[llu]_{p'}\ar[dll]^{q'}\\ && N&& }}\end{equation}
\end{proof}


\begin{proposition}[Global relation]\label{Global-equivalence}Let $ (M,\mathcal F_M) \stackrel{p}{\leftarrow} W \stackrel{q}{\rightarrow} (N,\mathcal F_N) $ and $ (M,\mathcal F_M) \stackrel{p'}{\leftarrow} W' \stackrel{q'}{\rightarrow} (N,\mathcal F_N)$ be bi-submersions. Assume that the singular foliations $\mathcal{F}_M$ and $\mathcal{F}_N$ are finitely generated. The following statements are
equivalent:
    \begin{enumerate}
        \item[(i)] $W$ and $W'$ are equivalent bi-submersions.
        \item[(ii)] Every $w\in W$ is equivalent to some $w'\in W'$ and every $w'\in W'$ is equivalent to some $w\in W$ .
    \end{enumerate}
\end{proposition}
\begin{proof}
    $(i)\Rightarrow (ii)$ is obvious.\\
\noindent
$(ii).\Rightarrow (i).$ There exists a countable family $ (w_l)_{l \in \mathbb N}$ of $W$ such that the open subsets  $(\mathcal U_l)_{l\in \mathbb N} 
$ (be as in the proof of Proposition \ref{Local-equivalence}) cover $W$. The restriction of the bi-submersion $W$ to $\mathcal{U}_l\subseteq W$ is related through a manifold denoted $V_l$ to the restriction of the bi-submersion $W'$ to an open subset of $W'$. The bi-submersion $W$ is related through the disjoint union $V_1:=\coprod_{l \in \mathbb N} V_l$ to the restriction of $W'$ to an open subset of $W'$. Now, by the same reasoning, $ W'$  is related through a manifold $V_2$ to the restriction of $W$ to  an open subset of $W$. Therefore, the bi-submersions $W$ and $W'$ are related  through disjoint union $V_1 \coprod V_2$.
\end{proof}

\subsection{Bi-vertical vector fields and equivalent points of a given bi-submersion}

The next theorem extends a theorem of Claire Debord 
\cite{Debord,DEBORD2013613} about holonomy groupoids, and, as in \cite{Debord,Debord,DEBORD2013613}, is a consequence of the Period Bounding Lemma (see Abraham and Robbin \cite{AR} or Ozols \cite{Ozols}). 
There is a long history of using Period Bounding Lemma for integration results, see e.g., \cite{Crainic-Fernandes}, \cite{Debord,DEBORD2013613}.

\begin{theorem}
\label{th:whereoneusesAnalyse}
Consider a bi-submersion $(M, \mathcal{F}_M)\stackrel{p}{\leftarrow}W\stackrel{q}{\rightarrow}(N, \mathcal F_N) $ such that 
$\mathcal{BV}_W $  is a singular foliation.
Every point $x\in W$ admits a neighborhood $ \mathcal U$ such that the following items are equivalent for any two $w,w' \in \mathcal U$
\begin{enumerate}
    \item[(i)] $w$ and $w'$ are equivalent,
    \item[(ii)] $w$ and $w'$ are in the same leaf of the bi-vertical singular foliation $\mathcal{BV}_W $.
\end{enumerate}
\end{theorem}


We start with a result, mostly inspired by a similar result by Claire Debord: it extends Proposition 1.1 in \cite{DEBORD2013613} (we state it for anchored bundle and not only for the trivial vector bundle associated to a finite family of vector fields, and Debord's result holds at a point not in a neighborhood).

Consider an anchored bundle $ (A,\rho)$ over a relatively compact open subset $U_M$ of the manifold $M$ with $A$ a trivial vector bundle over $U_M$ equipped with a flat connection so that through any $m\in U_M$ and any $ a_m \in A_m$ there exists an unique parallel section denoted by $ \widehat{a_m} \in \Gamma(A)$. Those always exist in a neighborhood of a point.

\begin{lemma} \label{lem:asDebord}
There exists a neighborhood $ \mathcal A \subset A$ of the zero section where the following are equivalent for any section $a$ valued in $ \mathcal A \subset A \to M$:
\begin{enumerate}
\item $\rho(a)=0$  
\item the orbit of the vector field $\rho(\widehat{a(m)})$ starting at an arbitrary point $m \in U_M $ is periodic of period $1$.   
\end{enumerate} 
\end{lemma}
\begin{proof}The first condition implies the second one for any neighborhood $ \mathcal A$, since the vector field $\rho(\widehat{a(m)})$ vanishes at $m$ for all $ m\in U_M$. We have to check that there is a neighborhood of $ U_M$ in $A$ where the second one implies the first one.  
Let us identify the restriction of $ A $ to $ U_M$ to a trivial product $A_0 \times \mathcal U_M$ with $A_0$ a vector space whose dimension is the rank of $A$.
The flat connection induces a vector field $X$ on  $A $ whose value at $ x=(a,y) \in A $ is $ (0,\rho_y(a))$. 
By the period bounding Lemma, the zero section has a neighborhood $\mathcal B $ on which the periods of $ X$ are bounded below, i.e., there exists $ \epsilon>0 $ such that if the integral curve of $  X$ starting from a point $x=(y,a) \in A \simeq U_M \times A_0 $ is periodic of period  $ \epsilon' < \epsilon$ then $X$ vanishes at this point $ x=(y,a)$, i.e., $ \rho(a)=0$. Now, let $ \mathcal A$ be the image of $ \mathcal B$ through the map  $(y,a) \mapsto (y,\eta a )$. Then $ \mathcal A$ satisfies the same property as $ \mathcal B$, but the minimal value of the periods is now $ \epsilon/\eta  $. Choosing $ \eta < \epsilon$ completes the proof.
\end{proof}

Notice that in Lemma \ref{lem:asDebord}, we do not really need a flat connection on a trivial vector bundle $A$: all we need is a natural manner to define the a vector field $X$. Any connection would be sufficient.
 Last,  the submersions $ (a,m) \mapsto m$ and $ (a,m) \mapsto \phi_1^{\rho(\widehat{a(m)})}(m)$ define a bisubmersion at least on a neighborhood of the zero section, and without any loss of generality,  $ \mathcal A$ can be chosen to be of this type. 
 Lemma \ref{lem:asDebord} then implies the following statement:
 
 \begin{lemma}\label{lem:asDebord2}
     Any bisection included in  $\mathcal A$ that induces the identity map of an open subset $ U_M$ has to be given by a section of $ A \to U_M$ valued in the kernel of $ \rho$.  
 \end{lemma}
 \begin{proof}
 A bisection has to be given by a section of $A$ valued in  $ \mathcal A$. It induces the identity map if and only if the second item in Lemma \ref{lem:asDebord} holds at every point $m$. This completes the proof.
 \end{proof}

\begin{lemma}\label{lem:whereoneusesAnalyse}
Theorem \ref{th:whereoneusesAnalyse} holds true when $ M=N$ and $ \mathcal F_M=\mathcal F_N$, and when there exists a local bisection $\Sigma$ through $ x$ which induces the identity map.
\end{lemma}
\begin{proof}
The direction $(ii)\Longrightarrow (i)$ comes from the fact that the flow of vector fields in $\mathcal{BV}_W$ is a  local bi-submersion morphism $W\to W$, since its flow leaves the source and target maps invariant. Let us prove the converse implication. According to \cite[Lemma 4.17]{LLL2}, {the restriction $A\to s(\Sigma)$ of} the normal bundle 
of $ \Sigma$ in $W$ {to $s(\Sigma)$}, identified with ${\mathrm{ker}}(Ts)|_{s(\Sigma)}$ is an anchored bundle of the {restricted singular foliation $\mathcal F|_{s(\Sigma)}$}, when equipped with $ \rho=Tt$, that we denote by $(A,\rho)$. In particular, upon choosing a  connection on it, it has a neighborhood $\mathcal A \subset A $ of the zero section which is a bi-submersion. Moreover, this bi-submersion is diffeomorphic to a neighborhood of $\Sigma$ in $W$ through a diffeomorphism that intertwines source and target maps, and maps the zero section to $ \Sigma$. It suffices therefore to prove the result when $ W$ is of this type. Upon restricting to neighborhood of a point, we can assume that $A$ is trivial, the connection is flat, and  it suffices therefore the prove the result for bi-submersions which are as in Lemmas \ref{lem:asDebord}-\ref{lem:asDebord2}.
In particular, a “right invariant action”  $ \xi \mapsto \overrightarrow{\xi}$ can be chosen. Moreover, vector fields of this type act transitively on the fibers of $ \mathcal A \to U_M$.  
Last, we can assume that $ \mathcal A$ is chosen such that the two conditions in Lemma  \ref{lem:asDebord} are equivalent, so that the statement in Lemma \ref{lem:asDebord2} holds. Let $ w \in \mathcal A $ be a point equivalent to a point in the zero section, which then has to be its base point $m$. By definition, there exists a bisection through $w$ inducing the identity map on an open subset of $U_M$. By Lemma  \ref{lem:asDebord2}, there is a local section $b$ valued in  the kernel of $ \rho$ through $w$.   Since the bi-vertical singular foliation is generated by the constant vector fields $ b$ with $b$ a section of $A$ valued  in the kernel of $ \rho$, it implies that $w$ and $ m$ are in the same leaf of the bi-vertical singular foliation. 
Now, let $w \in \mathcal A$ be a point that belongs to the fiber $A_m$ of  $ A \to U_M$.
There exists a left-invariant vector field $\overrightarrow{\xi}$ such that the image through the flow $ \phi^{\overrightarrow{\xi}}_1(w)$ at time $1$ of $\overrightarrow{\xi}$ belongs to the zero section, i.e. $ \phi^{\overrightarrow{\xi}}_1(w)=m$. 
Now, let $w'$ be another point in $ \mathcal A$ equivalent to $w$. Upon replacing $ \mathcal A$ by a smaller neighborhood of $U_M$, we can assume that $   \phi^{\overrightarrow{\xi}}_1(w')$ is well-defined. Then  $m= \phi^{\overrightarrow{\xi}}_1(w)$ and $   \phi^{\overrightarrow{\xi}}_1(w')$ are equivalent again. Since $m$ belongs to the zero section, they belong to the same leaf of the bivertical singular foliation. 
%
%
Since $\overrightarrow{-\xi}$ is a symmetry of the bivertical singular foliation, $w$ and $w'$ also have to belong to the same leaf of the bivertical singular foliation. This completes the proof.
%
%
\end{proof}

\begin{proof}[Proof of Theorem \ref{th:whereoneusesAnalyse}]
Let $ (S,\mathcal S)$ and $ (T,\mathcal T)$ be $ L$-cuts of the leaves through $s(x)$ and $t(x)$. Then $s^{-1}(S) \cap t^{-1}(T) $ is a bi-submersions between them, and any bisection though $x$ defines an isomorphism  $ (S,\mathcal S) \simeq (T,\mathcal T)$, so that we are in the situation of Lemma \ref{lem:whereoneusesAnalyse}. 
 The result now follows from the local decomposition theorem of a bi-submersion, see Appendix \ref{Local-structureofbis}.
\end{proof}

\subsection{Bi-vertical vector fields and equivalences of two given bi-submersions}\label{sec:totalrelation}


\begin{definition}\label{def:total-relation}
    Consider two equivalent bi-submersions $(M, \mathcal{F}_{M}) \stackrel{p}{\leftarrow} W \stackrel{q}{\rightarrow} (N, \mathcal{F}_{N})$ and $(M,\mathcal{F}_{M}) \stackrel{p'}{\leftarrow} W' \stackrel{q'}{\rightarrow} (N,\mathcal{F}_{N})$. 
    A relation $ (Z,\delta,\eta)$ between $W$ and $W'$ is said to be  \underline{total} if the next two conditions hold
\begin{enumerate}
    \item
$w \in W$ and $ w'\in W'$ are equivalent if and only if they are  related through $ (Z,\delta,\eta)$, i.e., if and only if there exists $v \in Z$ such that $ \delta(v)=w$ and $\eta(v)=w' $.
\item $\mathcal{BV}_{W}$ and $\mathcal{BV}_{W'}$ are singular foliations\footnote{See discussion in Section \ref{sec:Bi-verticalFoliation} about this condition.} and 
$$\left(W,\mathcal{BV}_W\right) \stackrel{\delta}{\longleftarrow} Z \stackrel{\eta}{\longrightarrow} \left(W',\mathcal{BV}_{W'}\right)$$ 
is a bi-submersion.
 \end{enumerate}
\end{definition}

\begin{remark}
     Consider 
     two equivalent 
     bi-submersions $(M, \mathcal{F}_{M}) \stackrel{p}{\leftarrow} W \stackrel{q}{\rightarrow} (N, \mathcal{F}_{N})$ and $(M,\mathcal{F}_{M}) \stackrel{p'}{\leftarrow} W' \stackrel{q'}{\rightarrow} (N,\mathcal{F}_{N})$. A relation $(Z,\delta,\eta)$ between $W$ and $W'$ satisfies item 2 of Definition \ref{def:total-relation} if and only if one of
   the two short sequences in \eqref{diag:lesdeux} and \eqref{diag:lesdeux2} is exact in the middle at the level of sections.\end{remark}

A crucial result is that the existence of a relation implies the existence of a total one.

\begin{theorem}\label{thm:bi-sub-equiv}
Let $(M,\mathcal{F}_{M}) \stackrel{p}{\leftarrow} W \stackrel{q}{\rightarrow} (N,\mathcal{F}_{N})$ and $(M,\mathcal{F}_{M}) \stackrel{p'}{\leftarrow} W' \stackrel{q'}{\rightarrow} (N,\mathcal{F}_{N})$ be bi-submersions such that $\mathcal{BV}_W\subseteq \mathfrak X(W)$  and $\mathcal{BV}_{W'}\subseteq \mathfrak X(W') $ 
 are singular foliations. Then, the following conditions are equivalent:
 \begin{enumerate}
     \item[(i)] $W$ and $W'$ are equivalent.
     \item[(ii)]
 There exists a total relation $(Z, \delta,\eta)$ between $W$ and $W'$.
\end{enumerate}

\end{theorem}
\begin{proof} 
The implication $(ii)\Rightarrow (i)$ is obvious. Now, let us show the other implication.

Let us assume that $W=W'$ and $(M, \mathcal{F}_M)=(N, \mathcal{F}_N)$. We will show that there exists a total relation $(B, s,t)$ between $W$ and $W$.
Let $w$ and $w'$ be equivalent points. By Proposition \ref{Local-equivalence}, there exists a local relation $(B_{w,w'}, \delta_{w,w'}, \eta_{w,w'})$ between them.  We can assume that the images $ \mathcal U_w $ and $ \mathcal U_w '$ of both submersions are open subsets  as in Theorem \ref{th:whereoneusesAnalyse} where equivalent points have to be in the same leaves of their respective bi-vertical singular foliation. Upon replacing $B_{w,w'} $ by the fibered product with a holonomy atlas $\mathcal{A}_w $ of the bi-vertical foliation of $ \mathcal U_w $, i.e., considering
\begin{equation}\label{Bww}  B_{w,w'}':= \mathcal{A}_w \times_W B_{w,w'} .\end{equation}
one obtains a new equivalence that now by construction satisfies the following property: two points in $ \mathcal U_w $ and $ \mathcal U_w '$ respectively are equivalent if and only if they are related through $B_{w,w'}'$.
We can cover the subset of $ W\times W$ made by pair of points $(w,w')$ that are equivalent by an open cover made of open sets of the form $ \mathcal U_{w}  \times \mathcal U_{w'}  $ as above. We extract a countable subfamily of those indexed by a set $ I$, and we denote by $B_i' $ the corresponding total relation between open subsets that we denote by $\mathcal U_{w_i} $ and $ \mathcal U_{w_i'}$.
Since \eqref{Bww} satisfies the second condition in Definition \ref{def:total-relation}, so does the disjoint union $B=\coprod_{i\in I} B_i'$. Hence $B$ is a total relation between $W$ and $W$. This completes this part of the proof.

Let us now assume $ W \neq W'$. Assume that $W$ and $W'$ are related through $(W, \mathcal{BV}_W) \stackrel{\varphi}{\leftarrow} V \stackrel{\psi}{\rightarrow} (W',\mathcal{BV}_{W'})$. Choose a total relation of bi-submersions  $(W, \mathcal{BV}_W)\stackrel{s}{\leftarrow}B\stackrel{t}{\rightarrow} (W,\mathcal{BV}_W)$  between $W$ and $W$ i.e., $B$ is a bi-submersion over $\mathcal{BV}_W$ such that the diagram    \begin{equation}\label{diag:bi-sub-equiv}
   \scalebox{0.7}{ \xymatrix{&& M&& \\W\ar[urr]^{p}\ar[drr]_{q}&&\ar@{->>}[ll]_{s}  B\ar@{..>}[u]\ar@{..>}[d]\ar@{->>}[rr]^{t} &&  W\ar[llu]_{p}\ar[dll]^{q}\\ && N&& }}
\end{equation}commutes and so that two points $(w,w')\in W\times W$ are related if and only if there exists an element $b\in B$ such that $s(b)=w$ and $t(b)=w'$. 

\begin{enumerate}
    \item By Proposition \ref{prop: inv-composition} (2), the fiber product $$(W, \mathcal{BV}_W) \stackrel{\delta}{\longleftarrow} B \times_{\varphi, W, t} V \stackrel{\eta}{\longrightarrow} (W',\mathcal{BV}_{W'})$$ is a bi-submersion, where $\delta=s\circ\mathrm{pr}_B$ and $\eta=\psi\circ \mathrm{pr}_V$. Moreover, $(B\times_{\varphi, W, t}V, \delta, \eta)$ is a relation between 
$W$ and $W'$. Indeed, for  $(b,v)\in B\times_{\varphi, W, t}V$ \begin{align*}
    p'\circ\eta(b,v)&=p'\circ\psi(v)=p\circ \varphi(v)=p\circ t(b),\quad\text{since $(V, \varphi,\psi)$ is a relation between $W$ and $W'$}\\&=p\circ s(b),\quad\text{since Diagram \eqref{diag:bi-sub-equiv} commutes}\\&=p\circ \delta(b,v).
\end{align*}Likewise, we have $q'\circ \eta =q\circ \delta$.
\item The triple $(B\times_{\varphi, W, t}V, s,t)$ satisfies  the first condition in Definition \ref{def:total-relation}: let  $(w,w')\in W\times W'$ be equivalent points and $v\in V$ such that $\psi(v)=w'$. The pair of points $(w,\varphi(v))\in W\times W$ are equivalent. Since the relation $B$ is total, there exists $b\in B$ such that $s(b)=w$ and $t(b)=\varphi(v)$. Thus, $(b,v)\in B\times_{\varphi, W, t}V$ projects to $w$ and $W$.
\end{enumerate}
Since $B$ satisfies the second condition in Definition \ref{def:total-relation}, so does $(B\times_{\varphi, W, t}V, s,t)$. 
 This completes the proof. \end{proof}

\begin{remark}
    Notice that in the proof of Theorem \ref{thm:bi-sub-equiv} we only needed, for example, a bi-submersion over $\mathcal{BV}_W\subseteq \mathfrak X(W)$ and not over $\mathcal{BV}_{W'}\subseteq \mathfrak X(W')$. Therefore, we do not need the assumption that both $\mathcal{BV}_W$ and $\mathcal{BV}_{W'}$ are locally finitely generated a priori, although the bi-submersion axioms constrain it in the end.
\end{remark}


From Theorem \ref{thm:bi-sub-equiv}, we derive the following corollary.

\begin{corollary}
   If $(Z, \delta, \eta)$ is a total relation between two bi-submersions $W$ and $W'$ and $B$ a bi-submersion over $\mathcal{BV}_W$, then the fiber product $B\times_W Z$ is a total relation between $W$ and $W'$.
\end{corollary}

\subsection{Equivalent points in a bi-submersion and the bi-vertical singular foliation (II)}

We start with an example of equivalent points, which motivates the discussion in this section.

\begin{example}
Let $ (M,\mathcal F_M) \stackrel{p}{\leftarrow} W \stackrel{q}{\rightarrow} (N,\mathcal F_N) $ be a bi-submersion such that the space $\mathcal{BV}_W\subset \mathfrak X(W)$ of bi-vertical vector fields form a singular foliation. 
 Two points $ w,w'\in W$ in the same leaf of $\mathcal{BV}_W\subset \mathfrak X(W)$ are related. This is an immediate consequence of the fact that the flow of a bi-vertical vector field, if it exists, is a isomorphism of bi-submersions  between $W$ and itself, since it preserves $ p$ and $q$.  
\end{example}

 Assume that there exists a submanifold $W'\subset W$ transverse to  $\mathcal{BV}_W$ and intersecting every leaf of $\mathcal{BV}_W$  at least once, then $W'$ is again a bi-submersion by Proposition \ref{prop:reduced-bisub}.
 It follows from the previous example that $W$ and $W'$ are equivalent.



\begin{theorem}
\label{th:reducedimension}
Let $W$ be a bi-submersion between two singular foliations  $$(M, \mathcal{F}_M)\stackrel{p}{\leftarrow}W\stackrel{q}{\rightarrow}(N, \mathcal F_N) $$ of global finite ranks. 
There exists a (non-connected) submanifold $W' \subset W$ fulfilling the following three properties:
\begin{enumerate}
    \item Every connected component of $ W'$ has dimension  \begin{equation}\label{eq:dim:min0}  \frac{1}{2} \left( {\mathrm{rk}} (\mathcal{F}_M )  +  {\mathrm{rk}} (\mathcal{F}_N )  + {\mathrm{dim}}(M) +  {\mathrm{dim}}(N)\right).\end{equation}
    \item The manifold $W'$, equipped with the restrictions of $p$ and $q$, is still a bi-submersion  between the two singular foliations $ (M,\mathcal F_M)$ and $ (N,\mathcal F_N)$ .
    \item The bi-submersions $ W'$ and $W$ are equivalent. 
    
    \end{enumerate}
\end{theorem}
Here is an immediate consequence of Theorem \ref{th:reducedimension}.
\begin{corollary}\label{cor:reducedimension}
  Let $(W_i')_{i\in I}$ a countable family of bi-submersions between two singular foliations $ (M,\mathcal F_M), (N,\mathcal F_N)$ of global finite ranks.  Then, for every $i\in I$ there exists a submanifold $W_i\subseteq W'_{i}$ of dimension $\frac{1}{2} \left( {\mathrm{rk}} (\mathcal{F}_M )  +  {\mathrm{rk}} (\mathcal{F}_N )  + {\mathrm{dim}}(M) +  {\mathrm{dim}}(N)\right)$ such that \begin{enumerate}
      \item the disjoint union $W=\coprod_{i\in I}W_i$ is a bi-submersion between the singular foliations $ (M,\mathcal F_M), (N,\mathcal F_N)$,
      \item 
      the bi-submersions $ W'$ and $W$ are related. 
  \end{enumerate} 
\end{corollary}

Here is a very important remark.
\begin{remark}\label{rk:cut-dim} If there is an anchored bundle $E\to TM$ over $\mathcal{F}_{M}$, then by following the proof of Lemma \ref{lemma:codim}, the dimension of $W'$ in Theorem \ref{th:reducedimension} can be taken to be equal to $\mathrm{rk}(E)+ \dim N$.  
\end{remark}

We will prove Theorem \ref{th:reducedimension} making the additional condition that the bi-vertical vector fields for $W$ form a singular foliation that we denote by $\mathcal{BV}_W $. However, we claim that this condition can be avoided, since the notion of transversal to $\mathcal{BV}_W $ makes sense even if $\mathcal{BV}_W $ is not locally finitely generated. Since this additional condition is satisfied every time we will need to use Theorem \ref{th:reducedimension}, we consider that we are allowed to make this additional assumption. We start with a lemma about generic singular foliation.

\begin{lemma}\label{lemma:Existence-{W'}}
Let $k \in \mathbb N$. Let $ \mathcal F_P$ be a singular foliation on a manifold $P$ whose leaves are all of codimension $ \leq k$.  Then there exists a (non-connected) submanifold $P' $
\begin{enumerate}
    \item every connected component of $P'$ has dimension $k$
    \item $P'$ intersects $ \mathcal F_P$ cleanly,
    \item Every leaf of $ P$ intersects $P'$ at least once. 
\end{enumerate}
\end{lemma}
\begin{proof}
This is a direct consequence of the local splitting theorem of a singular foliation stated under various forms (cf. \cite{zbMATH03423310,AS} or \cite[\S 7.3]{LLL1}).  One of its forms says that every point $p\in P$ belongs to a local chart $({U}_p, x_1,\ldots,x_{l},y_{l+1},\ldots, y_{\dim P} )$ of $P$ centered at $p$ with $\dim T_p\mathcal{F}_P=l$, on which we have  $\mathcal{F}_P|_{U_p}=\mathrm{span}\left(\left(\frac{\partial}{\partial x_i}\right)_{i=1}^{l}, Z_{l+1},\ldots, Z_{\mathrm{rk}_p(\mathcal{F})-l}\right)$ with $Z_i\in \mathfrak X(x_1=\cdots=x_{l}=0)$ and $Z_i(y_{\ell+1}=0, \ldots, y_d=0)=0$ for all $i=\ell+1,\ldots, \mathrm{rk}_p(\mathcal{F})-l$. 
Consider the $k$-dimensional manifold $P'_{\mathcal U_p} \subset P$ given for every $p\in P'_{U_p}$ by
$$P_{\mathcal U_p}'=\{x_{k+l-d+1}=\cdots = x_{l}=0\}$$ where ${\mathcal U_p} $ is a coordinates chart of $P$ centered at $p$ as above. By construction, $P'_{\mathcal U_p} $ satisfies the required axioms 1, 2 of Lemma \ref{lemma:Existence-{W'}}. 
It does not satisfy axiom 3, but intersects all the leaves of $ \mathcal F_P$ that intersect $\mathcal U_p $.
Therefore, one has to choose a countable number of points $ (p_i)_{i\in I}$ such that each leaf of $ \mathcal F_P$  intersects $ \mathcal U_{p_i}$ and such that the closure are the submanifolds $ P_{\mathcal U_{p_i}}'$ and   $ P_{\mathcal U_{p_j}}'$ do not intersect each other for $ i \neq j$.
\end{proof}

\begin{proof}[Proof (of Theorem \ref{th:reducedimension})]
By Lemma \ref{lemma:codim}, applied to the bi-vertical singular foliation $\mathcal{BV}_W $ on $W$, there exists a submanifold $ W' \subset W$,  whose dimension is the integer in \eqref{eq:dim:min},  that satisfies the three conditions in Lemma \ref{lemma:Existence-{W'}}. 
In particular, Condition 2 in Lemma \ref{lemma:Existence-{W'}} implies that $ W'$ is still a bi-submersion. Since two points in the same leaf of $\mathcal{BV}_W $ are equivalent, Condition 3 in Lemma \ref{lemma:Existence-{W'}} implies that $W'$ and $W$ are equivalent bi-submersions.
\end{proof}

Theorem \ref{th:whereoneusesAnalyse} admits the following corollary.

\begin{corollary}
Let $W_1$ and $W_2 $ be bi-submersions between singular foliations $ (M,\mathcal F_M), (N,\mathcal F_N)$ be singular foliations. 
Let $W$ be an equivalence between $W_1$ and $W_2$, i.e., a bi-submersion  $W_1\stackrel{\delta}{\leftarrow}W\stackrel{\eta}{\rightarrow} W_2 $ between the bi-vertical singular foliations on $W_1$ and on $W_2$. 
Then for every $w \in W$, $ \delta(w)$ and $ \eta(w)$ admit neighborhoods $\mathcal W_1 $ and $\mathcal W_2$ on which for any two $w_1 \in \mathcal W_1$ and $w_2 \in \mathcal W_2$ the following are equivalent:
\begin{enumerate}
    \item[(i)]  $w_1$ and $ w_2$ are equivalent,
    \item[(ii)] there exists $w' \in W $ such that  $\delta(w')=w_1$ and $\eta(w')=w_2$.
\end{enumerate}
Moreover, $w$ is mapped to $ (p(w),0_{O},q(w) )$
where $0_O$ is the zero element of the fiber of $O$.
 
\end{corollary}
\begin{proof}
The inclusion $(ii) \implies (i)$ is obvious. Let us prove the converse.  Consider a neighborhood $ \mathcal W$ of $w$ as in Theorem \ref{th:whereoneusesAnalyse}. Let  $ w_1 \in \delta(\mathcal W)$ and $w_2 \in  \eta(\mathcal W)$ be two points which are equivalent. Without loss of generality, one can assume that there exists a submanifold $ \Sigma$ of $W$ through  $w$ such that  $ \eta : \Sigma \to \delta(\mathcal W)$ is a diffeomorphism and $ \eta(\Sigma) = \eta(\mathcal W_2)$.
For any $w_1 ,w_2$ in that open subset, the following items are equivalent. 
\begin{enumerate}
    \item[(i)]  $w_1$ and $ w_2$ are equivalent,
    \item[(ii)]  $w_1$ and $\delta \circ \eta_\Sigma^{-1}(w_2) $ are equivalent.
\end{enumerate}
This means that $w_1$ and $\delta \circ \eta_\Sigma^{-1}(w_2)$ belong to the same leaf of the bi-vertical singular foliation of $\mathcal{BV}_{W_1}$.
There exists $w' \in \mathcal W $ such that  $\delta(w')=w_1$ and $\eta(w')=w_2$.
\end{proof}

Let us conclude this discussion with an important result.

\begin{corollary}\label{cor:bi-sub-equiv2}
Let $(M,\mathcal{F}_{M}) \stackrel{p}{\leftarrow} W \stackrel{q}{\rightarrow} (N,\mathcal{F}_{N})$ and $(M,\mathcal{F}_{M}) \stackrel{p'}{\leftarrow} W' \stackrel{q'}{\rightarrow} (N,\mathcal{F}_{N})$ be bi-submersions between singular foliations 
of finite ranks which are coherent sheaves. Then, the following conditions are equivalent:
 \begin{enumerate}
     \item[(i)] $W$ and $W'$ are equivalent.
     \item[(ii)] There exists a total relation between $ W $ and $W' $ given by a manifold of dimension 
      $$ \frac{1}{2} \left( {\mathrm{rk}} (\mathcal{BV}_W )  +  {\mathrm{rk}} (\mathcal{BV}_{W'} )  + {\mathrm{dim}}(W) +  {\mathrm{dim}}(W')\right)$$
\end{enumerate}
\end{corollary}

Theorem \ref{th:whereoneusesAnalyse} can be even made more sophisticated. 

\begin{proposition}
\label{prop:rightactionisiso}
Consider an anchored bundle $ (A,\rho)$ over a singular  foliation $\mathcal F_M$.
    Let $(M,\mathcal{F}_{M}) \stackrel{p}{\leftarrow} W \stackrel{q}{\rightarrow} (N,\mathcal{F}_{N})$ be a bi-submersion. The manifold $W'$ in Theorem \ref{th:whereoneusesAnalyse}
can be chosen such that there exists a right action
     $p^* A \simeq \ker(Tq)$ which is an isomorphism of vector bundles.
\end{proposition}

\begin{proof}
    This is a direct consequence of Theorem \ref{th:reducedimension} and Lemma \ref{lem:biss_anchored}; \ref{lem:anchoredbundleforbivertical}.
\end{proof}

\subsection{The Holonomy groupoid $\mathcal{H}(M, \mathcal{F})\rightrightarrows M$ of a singular foliation}

We recall the construction of the holonomy groupoid of a singular foliation introduced in \cite{AS}. In the process, we will review the notion of an atlas and prove several new results about those.

\vspace{0.5cm}

Let $(M,\mathcal F)$ be a singular foliation.

\begin{definition} \cite{AS}
A bi-submersion 
$M \stackrel{p}{\leftarrow} \mathcal{A} \stackrel{q}{\rightarrow} M$ over $\mathcal F $ is an \emph{atlas of $\mathcal F $} when 
\begin{enumerate}
    \item the composition $\mathcal{A}* \mathcal{A}$ is equivalent to $\mathcal{A}$,
    \item $\mathcal{A}$ is equivalent to its inverse $ \mathcal{A}^{-1}$
\end{enumerate}
 Two atlases are said to be \emph{equivalent} when they are equivalent as bi-submersions over  $(M,\mathcal F) $.
\end{definition}

\begin{lemma}(Existence of local units of an atlas)\label{lem:units}
   An atlas $M \stackrel{p}{\leftarrow} \mathcal{A} \stackrel{q}{\rightarrow} M$ admits local units, i.e., every point\footnote{We speak of \emph{global unit} when $\mathcal U=M $, and \emph{local unit} otherwise.}   $x \in \mathcal{A}$ there is  a neighborhood  $\mathcal{U}$ of $m=q(x)$ and a smooth
   map $\epsilon \colon \mathcal U \longrightarrow \mathcal{A} $ which is a section of both $p$ and $q$.
\end{lemma}
\begin{proof}
We need to show $W$ admits bisections whose basic diffeomorphism is the identity map. Let $m\in M$ and $x\in \mathcal{A}$ such that $q(x)=m$. Let $\Sigma$ be a bisection through $x$. Let $\Phi\colon \mathcal{V}\subseteq \mathcal{A}\rightarrow \mathcal{A}^{-1}$ be a morphism of bi-submersions   on an open neighborhood $\mathcal{V}$ of $x$. We can assume that $\Sigma\subset \mathcal{V}$. By Lemma \ref{lemma:bisection}, the image $\Sigma':=\Phi(\Sigma)$ is a bisection of $\mathcal{A}$ through $\Phi(x)$ and $\underline{\Sigma}'=\underline{\Sigma}^{-1}$. The graph $\Sigma'':=\mathrm{graph}(\Phi|_\Sigma)\subset \Sigma\times \Sigma'$ is a bisection of $\mathcal{A}*\mathcal{A}$ through $(x, \Phi(x))$ whose basic diffeomorphism is $\underline{\Sigma}''=\underline{\Sigma}\circ\underline{\Sigma}'$, that coincides with the identity map on a neighborhood of $m$. For any $\mu\colon \mathcal{A}*\mathcal{A}\rightarrow \mathcal{A}$ morphism of bi-submersions  in a neighborhood of $\Sigma''$ that sends $(x, \Phi(x))$ to $x$  the image $\mu(\Sigma'')$ is a bisection  of $\mathcal{A}$ through $x$ that carries the identity map in a neighborhood of $m$.
\end{proof}

\begin{remark}
\label{rmk:3hornspaces} If $\mathcal{A}$ is a bi-submersion atlas over a singular foliation $\mathcal F$, then all the bi-submersion products $\mathcal{A}*\mathcal{A},\;\mathcal{A}*\mathcal{A}^{-1} $ and $\mathcal{A}^{-1}*\mathcal{A}$ are equivalent to $\mathcal A$.
\end{remark}

Here is a result from \cite{AS}.

\begin{proposition}{(Groupoid associated with an atlas)}
    Let $ M \stackrel{p}{\leftarrow} \mathcal{A} \stackrel{q}{\rightarrow} M $ of $\mathcal F $ be an atlas over $\mathcal F $. There is a topological groupoid $\mathcal{A}/_\sim\rightrightarrows M$ \emph{associated to $\mathcal{A}$} defined as follows: 

    \begin{enumerate}
        \item $\mathcal{A}/_\sim$ is the quotient of the equivalence relation  on $\mathcal{A}$ given by $ x \sim x'$ if and only if $x$ and $x'$ are equivalent,
\item the source $s$ and target $t$ are defined respectively by $s([x])=q(x)$ and $t([x])=r(x)$ for any representative $x\in \mathcal{A}$,
        \item for every $x,y\in \mathcal{A}$ such that $s([x])=t([y])$, their product is  $[x]\cdot [y]:= [\mu(x,y)] $ for any $ \mu(x,y)\in \mathcal A$ equivalent to $ (x,y)\in \mathcal A \times \mathcal A$. Moreover,  one can choose a local map $\mu\colon \mathcal{A}*\mathcal{A}\mapsto \mathcal{A}$ smooth in the neighborhood of $(x,y)$,
        \item The unit map is induced by the local units of Lemma \ref{lem:units}. The inverse of $ [x]$ is any element $ [\iota (x)] \in \mathcal A^{-1}$ equivalent to $x$. Again, $\iota\colon\mathcal{A}\rightarrow \mathcal{A}^{-1}$ can be chosen smooth in the neighborhood of any point.    \end{enumerate}
\end{proposition}

\begin{example}\cite{AS}\label{ex:atlas}
Let $(M,\mathcal{F})$ be a singular foliation. Let  $(M, \mathcal{F})\stackrel{p_i}{\leftarrow}\mathcal{W}_i\subseteq M\times \mathbb R^{n_i}\stackrel{q_i}{\rightarrow}(M, \mathcal{F})$ for $i\in I$ a family made of countably many path holonomy bi-submersions as in Example \ref{ex:holonomy-biss} such that $\bigcup_{i\in I} p_i(\mathcal{W}_i)=M$. 
 In \cite{AS}, it is stated that the disjoint union $$\mathcal W=\coprod_{n\in \mathbb N} \mathcal{W}_{i_1,\ldots,i_n}^{\pm}$$ (where $\mathcal{W}_{i_1,\ldots,i_n}^{\pm}:={\mathcal W_{i_1}^{\pm} * \dots * \mathcal W_{i_n}^{\pm}}\;\;({n-\text{times} })$ and  $^{\pm}$ means that we consider $\mathcal W_i$ or its inverse $\mathcal W_i^{-1} $ for $i\in I$) is an atlas over $(M,\mathcal{F})$, denoted by $M \stackrel{p}{\leftarrow} \mathcal{W} \stackrel{q}{\rightarrow} M$ and called a \emph{path holonomy atlas}.
 Notice that it is not an atlas in our sense since the manifolds do not have all the same dimensions, but we did not dare to change the name which is already widely used.
 \end{example}

The holonomy groupoid $\mathcal{H}(M, \mathcal{F})\rightrightarrows M$ is not smooth in general. However, it is longitudinally smooth, see \cite{Debord,AS} and \cite{DEBORD2013613}.
 
\begin{definition}\cite{AS}
    The \emph{holonomy groupoid $\mathcal{H}(M, \mathcal{F})\rightrightarrows M$ of $\mathcal{F}$} is the groupoid associated to a fundamental  atlas of $\mathcal{F}$.
\end{definition}

The path holonomy atlas has an issue: its connected components have a non-bounded dimensions. It makes more sense to have an atlas of the following form:

 \begin{definition}
     \label{def:holonomy_atlas}
 We call \emph{holonomy atlas} an atlas which is equivalent to  path holonomy atlas, and is given by a manifold whose connected components are all the same dimension.
 \end{definition}

Theorem \ref{th:reducedimension} admits the following immediate consequence.

 \begin{corollary}
     Every singular foliation of finite rank admits a holonomy atlas of dimension  ${\mathrm{rk}}(\mathcal F)+{\mathrm{dim}}(M) $.
 \end{corollary}
 \begin{proof}{We will only prove the claim with the additional assumption that $\mathcal F$ is a coherent sheaf.
 In this case, let $ \mathcal A_{i_1, \dots, i_n} \subset \mathcal{W}_{i_1,\ldots,i_n}^{\pm}$ be for all  $(i_1,\ldots,i_n) \in I^n$
 a non-connected submanifold of dimension $k$ that intersects cleanly the bi-vertical foliation and intersect each of its leaf at least once. 
 The bi-submersion $\mathcal{W}_{i_1,\ldots,i_n}^{\pm}$  is related to the bi-submersion $ \Sigma_{i_1, \dots, i_n}$ in view of Proposition \ref{prop:reduced-bisub}. The disjoint union $\mathcal{A}=\coprod_{n\in \mathbb N} \mathcal{A}_{i_1, \dots, i_n}$ is a bi-submersion for $ (M,\mathcal F)$ made of sub-manifolds all of the same dimension $k$ and related to the path-holonomy atlas of Example \ref{ex:atlas}. Moreover, it admits a globally defined unit $\epsilon\colon M\hookrightarrow \mathcal{A}$ in the case where $I$ has only one element, i.e., if the singular foliation $\mathcal F$ is globally finitely generated.}

 Now, the assumption we made (that such an integer $k$ exists) holds if the following condition holds: there exists vector bundles $ E_{-1}, E_{-2}$ and morphisms 
   $$ E_{-2} \stackrel{d}{\longrightarrow} E_{-1} \stackrel{\rho}{\longrightarrow} 
 TM $$
 such that $ \rho(\Gamma(E_{-1})) = \mathcal F$ and the following sequence is exact:
  $$   \Gamma(E_{-2}) \longrightarrow \Gamma(E_{-1})  \longrightarrow \mathcal F $$
This comes from item 2 of Lemma \ref{lem:biss_anchored}.
In this case,  $\mathcal F$ is \emph{globally bounded}, i.e.,  the semi-continous map $m\in M \to  \dim(\frac{\mathcal{F}}{\mathcal I_m\mathcal{F}})$ is bounded. Notice that this map associates to a point the rank of holonomy Lie algebroid if the leaf through this point (cf. \cite[\S 2.3]{AZ}). Moreover, the integer $k$ can be chosen to be the sum of the dimension of $M$ with the maximum value of that function, that is $\dim M+\mathrm{rk}(\mathcal{F})$.
\end{proof}

  %
 %


 Theorem \ref{th:whereoneusesAnalyse} also admits the following side result, which describes the local topology of the holonomy groupoid of Androulidakis-Skandalis.

\begin{corollary} 
Let $ \mathcal F$ be a singular foliation on $M$, and assume that it is a coherent sheaf with local resolution:
  $$    E_{-2}\stackrel{\dd}{\longrightarrow} E_{-1}\stackrel{\rho}{\longrightarrow} TM$$
   Every point $ \gamma$ is the holonomy groupoid of $ \mathcal F$ admits a neighborhood $ \mathcal U$ homeomorphic, as a topological space, to a neighborhood of the target of $ \gamma$ in the quotient space  $ E_{-1}/\dd(E_{-2}) $ of $ E_{-1}$  by the image of $\dd$.
\end{corollary}

    The word “target” could of course be replaced by “source” in this corollary: it would still be true.
   Notice that for every anchored bundle  $(A,\rho)$ over $ \mathcal F$,  $ E_{-1}/\dd(E_{-2}) $ is homemorphic to the quotient of $ A$  by the strong kernels of $A$.

\section{Bi-submersion towers over singular foliations} \label{sec:tower}
In this section, we introduce the concept of a bi-submersion tower, which extends the notion of a bi-submersion over a singular foliation. This generalization will be useful for future developments.

\begin{definition}\label{def:bi-submersion-tower}
Let $(M, \mathcal{F}_M)$ and $(N,\mathcal{F}_N)$ be foliated manifolds. A \emph{bi-submersion tower between $(M, \mathcal{F}_M)$ and $(N,\mathcal{F}_N)$} is a pair of families of foliated manifolds $(K_i, \mathcal{F}_i)_{i\geq 0}$ and $(\Lambda_i, \mathcal
F^i)_{i\geq 0}$ (the latter being referred to as \emph{horn foliations}) that come equipped with bi-submersion structures  
\begin{enumerate}
    \item $\left((\Lambda_i, \mathcal{F}^{i}) \stackrel{a_i}{\leftarrow} K_{i+1} \stackrel{b_i}{\rightarrow} (K_i,\mathcal{F}_i)\right)_{i\geq 0}$ 

    \item $\left((\Lambda_i, \mathcal{F}^{i}) \stackrel{\delta_i}{\leftarrow} \Lambda_{i+1} \stackrel{\eta_i}{\rightarrow} (K_i,\mathcal{F}_i )\right)_{i\geq 0}$
    \end{enumerate}
    with $K_0:=M,\; \Lambda_0:= N$  and $\mathcal{F}_0=\mathcal{F}_M,\; \mathcal{F}_{\Lambda_0}:=\mathcal{F}_N$ such that

    \begin{enumerate}
        \item  for every $i\geq 0$ the following diagram commutes

        \begin{equation} 
    \scalebox{0.7}{\xymatrix{&&  (K_i,\mathcal{F}_i)  &&\\ &&\ar[dll]_{\delta_{i+1}}\Lambda_{i+2}\ar@{->}[drr]^{\eta_{i+1}} &&\\ (\Lambda_{i+1}, \mathcal{F}^{{i+1}})\ar[uurr]^{\eta_i} \ar[ddrr]_{\delta_i}&& &&  (K_{i+1}, \mathcal{F}_{i+1}) \ar[ddll]^{a_i}\ar[uull]_{b_i} \\&& \ar@{->}[ull]^{a_{i+1}}K_{i+2}\ar@{->}[urr]_{b_{i+1}}&&\\&&(\Lambda_i,\mathcal{F}^{i}) && }}
\end{equation}

 \item   for all $i \geq 1$, $\mathcal F_{i}\subseteq \Gamma(\ker Ta_{i-1})\cap\Gamma(\ker Tb_{i-1})$ and $\mathcal F^{i}\subseteq \Gamma(\ker T\delta_{i-1})\cap\Gamma(\ker T\eta_{i-1})$.
    \end{enumerate}
For small $i$, the picture looks like    
\begin{equation} \label{eq:tower1}
   \scalebox{0.7}{ \xymatrix{&&&&  M  &&&&\\ &&&& \ar@[]@{-->}@/_2pc/[dllll]_{\delta_1}\Lambda_2\ar@[]@{-->}@/^2pc/[drrrr]^{\eta_1} \ar@[]@{<.}@/_1pc/[dlll]_{a_2}\ar@[]@{<.}@/^1pc/[drrr]^{\delta_2}&&&&\\ \Lambda_1\ar@[]@/^2pc/[uurrrr]^{\eta_0} \ar@[]@/_2pc/[ddrrrr]_{\delta_0}& K_3 \ar@[]@{<.}[r]^{b_3} &\cdots  & & \cdots && \cdots & \ar@[]@{<.}[l]_{a_3} \Lambda_3\ar@[]@{.>}@/^1pc/[dlll]^{\eta_2}& K_1 \ar@[]@/^2pc/[ddllll]^{a_0}\ar@[]@/_2pc/[uullll]_{b_0} \\&&&& \ar@[]@{-->}@/^2pc/[ullll]^{a_1}K_2\ar@[]@{-->}@/_2pc/[urrrr]_{b_1}\ar@[]@{<.}@/^1pc/[ulll]^{b_2}&&&&\\&&&& N &&&& }}
\end{equation}
A bi-submersion tower between $(M, \mathcal{F})$ and $(M,\mathcal{F})$ is called \emph{bi-submersion tower over  $\mathcal{F}$}.

\begin{remark}
    Morally, a bi-submersion tower over two singular foliations should be understood as a sequence of equivalences  between equivalences of bi-submersions.
\end{remark}

\end{definition}
\begin{definition} In the notation of Definition \ref{def:bi-submersion-tower},
     a bi-submersion tower is said to be 
     \emph{exact} if for all $i \geq 1$, $\mathcal F_{i}= \Gamma(\ker Ta_{i-1})\cap\Gamma(\ker Tb_{i-1})$.

\end{definition}
\begin{example}
    We recover a bi-submersion tower over  a singular foliation in the sense of  \cite[Definition 5.3]{RubenSymetries} as follows: a bi-submersion tower over a single singular foliation $(M,\mathcal{F})$ can be obtained as a sequence of manifolds and maps 
\begin{equation}\label{eq:tower1}T_\bullet\colon \xymatrix{\cdots\phantom{X}\ar@/^/[r]^{p_{i+1}}\ar@/_/[r]_{q_{i+1}}&{T_{i+1}}\ar@/^/[r]^{p_{i}}\ar@/_/[r]_{q_{i}}&{T_{i}}\ar@/^/[r]^{p_{i-1}}\ar@/_/[r]_{q_{i-1}}&{\cdots}\ar@/^/[r]^{p_1}\ar@/_/[r]_{q_1}&T_1\ar@/^/[r]^{p_0}\ar@/_/[r]_{q_0}&T_0,}
\end{equation}
together with a sequence $\mathcal F_i$ of singular foliations on $T_i$, with the convention that $T_0=M$ and $\mathcal F_0 = \mathcal F$, such that
\begin{enumerate}
\item  for all $i \geq 1$, $\mathcal F_i\subseteq\Gamma(\ker Tp_{i-1})\cap\Gamma(\ker Tq_{i-1})$, 
\item  for each $i\geq 0$, $\xymatrix{{T_{i+1}}\ar@/^/[r]^{p_{i}}\ar@/_/[r]_{q_{i}}&{T_{i}}}$ is a bi-submersion over $\mathcal F_i$.
\item for all $i\geq 1$, the following diagram commutes $$\scalebox{0.7}{ \xymatrix{&& T_{i-1}&& \\T_i\ar[urr]^{p_{i-1}}\ar[drr]_{q_{i-1}}&&\ar@{->}[ll]_{p_i}  T_{i+1}\ar@{..>}[u]\ar@{..>}[d]\ar@{->}[rr]^{q_i} &&  T_i\ar[llu]_{p_{i-1}}\ar[dll]^{q_{i-1}}\\ && T_{i-1}&& }}.$$
\end{enumerate} 
    
Here, \eqref{eq:tower1} corresponds to a particular case of Definition \ref{def:bi-submersion-tower} with $K_i=\Lambda_i:= T_i$ and $\mathcal{F}_i=\mathcal{F}_{\Lambda_i}$.    

A geometrical consequence is that for $i\geq 1$,
two points $x,x'\in T_i$ of the same leaf of $\mathcal F_i$ satisfy $p_{i-1}(x)=p_{i-1}(x')$ and $q_{i-1}(x)=q_{i-1}(x')$ since for all $x\in T_i,\; T_x\mathcal{F}_i\subset (\ker Tq_{i-1})_{|_x}\cap(\ker Tq_{i-1})_{|_x}$. Therefore, if the bi-submersions of item 2  are made of leaf preserving bi-submersions (e.g., path holonomy bi-submersions), then item 3 is systematically satisfied. That is,  we have for all $i\geq 1$, $p_{i-1}\circ p_i=p_{i-1}\circ q_i$ and $q_{i-1}\circ p_i=q_{i-1}\circ q_{i}$.

\end{example}

The following theorem states that exact bi-submersion towers can be constructed over singular foliations that admit geometric resolutions. We refer the reader to \cite{RubenSymetries} for a proof.
\begin{theorem}[\cite{RubenSymetries}]\label{thm:equivalence-tower}
Let $\mathcal F$ be a singular foliation on $M$. The following items are equivalent:
\begin{enumerate}
    \item $\mathcal F$ admits a geometric resolution.
    \item There exists an exact path holonomy bi-submersion tower over $(M, \mathcal F)$ (i.e., a bi-submersion tower generated by path holonomy bi-submersions) of the form \begin{equation}\label{eq:tower}T_\bullet\colon \xymatrix{\cdots\phantom{X}\ar@/^/[r]^{p_{i+1}}\ar@/_/[r]_{q_{i+1}}&{T_{i+1}}\ar@/^/[r]^{p_{i}}\ar@/_/[r]_{q_{i}}&{T_{i}}\ar@/^/[r]^{p_{i-1}}\ar@/_/[r]_{q_{i-1}}&{\cdots}\ar@/^/[r]^{p_1}\ar@/_/[r]_{q_1}&T_1\ar@/^/[r]^{p_0}\ar@/_/[r]_{q_0}&M}.
\end{equation}
\end{enumerate}
\end{theorem}

\bibliographystyle{alpha}
\bibliography{main}
\vfill
\begin{center}\textsc{Institut \'Elie Cartan de Lorraine, UMR 7502, Universit\'e de Lorraine, Metz, France.\\Department of Mathematics, University of Illinois Urbana-Champaign\\ 1409 W. Green Street, Urbana, IL 61801, USA.}\end{center}
\end{document}